\documentclass[a4paper,final,12pt]{article}

\usepackage{amsmath,amssymb,amsthm,amsfonts,
geometry,listings,graphicx,showkeys,bbm,
enumerate,nicefrac}

\newtheorem{lemma}{Lemma}[section]
\newtheorem{theorem}[lemma]{Theorem}
\newtheorem{proposition}[lemma]{Proposition}
\newtheorem{corollary}[lemma]{Corollary}

\newcommand{\XO}[1]{\ensuremath{\bar{X}_{#1}}}
\newcommand{\Y}[1]{\ensuremath{\mathcal{X}_{#1}}}
\newcommand{\YM}[1]{\ensuremath{\mathcal{X}_{#1}^M}}
\newcommand{\YMM}[1]{\ensuremath{\tilde{\mathcal{X}}_{#1}^M}}
\renewcommand{\O}[1]{\ensuremath{\mathcal{O}_{#1}}}
\newcommand{\OM}[1]{\ensuremath{\mathcal{O}_{#1}^M}}

\newcommand{\YO}[1]{\ensuremath{\bar{\mathcal{X}}_{#1}}}
\newcommand{\YOM}[1]{\ensuremath{\bar{\mathcal{X}}_{#1}^M}}
\newcommand{\Yb}[1]{\ensuremath{Y_{#1}}}
\newcommand{\YX}[1]{\ensuremath{\mathbf{X}_{#1}}}
\newcommand{\YXM}[1]{\ensuremath{\mathbf{X}_{#1}^M}}

\newcommand{\fl}[2][T\!/\!M]{
  \ensuremath{ \lfloor #2 \rfloor_{#1} }
}

\newcommand{\cl}[2][T\!/\!M]{
  \ensuremath{ \lceil #2 \rceil_{#1} }
}

\renewcommand{\P}[0]{\ensuremath{\mathbb{P}}}
\newcommand{\R}[0]{\ensuremath{\mathbb{R}}}
\newcommand{\N}[0]{\ensuremath{\mathbb{N}}}
\newcommand{\E}[0]{\ensuremath{\mathbb{E}}}

\newcommand{\Rm}[0]{\ensuremath{\Psi}}
\newcommand{\Um}[0]{\ensuremath{\Phi}}
\newcommand{\Vm}[0]{\ensuremath{\mathcal{V}}}

\newcommand{\D}[0]{\ensuremath{\mathcal{D}}}
\newcommand{\B}[0]{\ensuremath{\mathcal{B}}}
\newcommand{\M}[0]{\ensuremath{\mathcal{M}}}
\newcommand{\C}[0]{\ensuremath{\mathcal{C}}}
\newcommand{\F}[0]{\ensuremath{\mathcal{F}}}
\newcommand{\Ff}[0]{\ensuremath{\mathbf{F}}}
\renewcommand{\L}[0]{\ensuremath{\mathcal{L}}}
\renewcommand{\S}[0]{\ensuremath{\mathcal{S}}}
\newcommand{\SM}[0]{\ensuremath{\mathcal{S}^M}}

\newcommand{\BaMi}[0]{\ensuremath{V}}
\newcommand{\BaLa}[0]{\ensuremath{W}}

\newcommand{\one}[0]{\ensuremath{\mathbbm{1}}}

\newcommand{\qf}[0]{\ensuremath{\phi}}
\renewcommand{\sf}[0]{\ensuremath{\rho}}

\newcommand{\qs}[0]{\ensuremath{\varphi}}
\newcommand{\sr}[0]{\ensuremath{\rho}}
\newcommand{\srr}[0]{\ensuremath{\varrho}}

\newcommand{\cc}[0]{\ensuremath{\alpha}}
\newcommand{\nr}[0]{\ensuremath{\theta}}

\newcommand{\qp}[0]{\ensuremath{\varphi}}
\renewcommand{\sp}[0]{\ensuremath{\rho}}
\newcommand{\spp}[0]{\ensuremath{\varrho}}

\newcommand{\cp}[0]{\ensuremath{\alpha}}
\newcommand{\np}[0]{\ensuremath{\theta}}

\newcommand{\qm}[0]{\ensuremath{\phi}}
\newcommand{\qmm}[0]{\ensuremath{\varphi}}
\newcommand{\sm}[0]{\ensuremath{\rho}}
\newcommand{\smm}[0]{\ensuremath{\varrho}}

\newcommand{\cm}[0]{\ensuremath{\chi}}
\newcommand{\nm}[0]{\ensuremath{\theta}}

\newcommand{\nmm}[0]{\ensuremath{\vartheta}}

\newcommand{\Id}[0]{\ensuremath{\textup{Id}}}

\usepackage{courier}
\title{Strong convergence rates for nonlinearity-truncated
Euler-type approximations of stochastic \\ Ginzburg-Landau equations}
\author{Sebastian Becker and Arnulf Jentzen}

\begin{document}

\maketitle

\begin{abstract}
This article proposes and analyzes
explicit and easily implementable temporal numerical approximation schemes
for additive noise-driven stochastic partial differential equations (SPDEs) with polynomial nonlinearities
such as, e.g., stochastic Ginzburg-Landau equations.
We prove essentially sharp strong convergence rates for the considered 
approximation schemes.
Our analysis is carried out for abstract stochastic evolution equations
on separable Banach and Hilbert spaces including the above mentioned
SPDEs as special cases.
We also illustrate our strong convergence rate
results by means of a numerical simulation in {\sc Matlab}.
\end{abstract}

\tableofcontents

\section{Introduction}
\label{sec:intro}

In this article we study strong numerical approximations
of semilinear stochastic evolution equations (SEEs) with superlinearly growing nonlinearities.
The explicit Euler scheme and the linear-implicit Euler scheme are 
known to diverge strongly and numerically weakly 
in the case of such SEEs; see Theorem~2.1 in 
Hutzenthaler et al.~\cite{hjk11},
Theorem~2.1 in 
Hutzenthaler al.~\cite{HutzenthalerJentzenKloeden2013},
and Section~5.1 in Kurniawan~\cite{Kurniawan2014}.
Fully drift-implicit Euler schemes have been shown to converge strongly 
even in the case of some SEEs with superlinearly growing nonlinearities
and thereby overcome the lack of strong convergence of the explicit 
and the linear-implicit Euler scheme;
see, e.g.,
Theorem~2.4 in Hu~\cite{Hu1996},
Theorem~2.10 Gy{\"o}ngy \& Millet~\cite{gm05},
and Theorem~1.1 in Kov{\'a}cs et al.~\cite{kll2015}.
Fully drift-implicit Euler schemes can, however, often only be simulated approximatively
as a nonlinear equation has to be solved in each time step
and the resulting approximations of the fully drift-implicit Euler approximations 
require additional computational effort (particularly, when the state space of the considered
SEE is high dimensional, see, e.g., Figure~4 
in Hutzenthaler et al.~\cite{HutzenthalerJentzenKloeden2012})
and have not yet been shown to converge strongly.
Recently, a series of explicit and easily implementable 
time-discrete approximation schemes
have been proposed and shown to converge strongly in the case of SEEs 
with superlinearly growing nonlinearities; see, e.g.,
Hutzenthaler et al.~\cite{HutzenthalerJentzenKloeden2012},
Wang \& Gan~\cite{WangGan2013},
Hutzenthaler \& Jentzen~\cite{HutzenthalerJentzen2012},
Tretyakov \& Zhang~\cite{TretyakovZhang2013},
Halidias~\cite{Halidias2013},
Sabanis~\cite{Sabanis2013,Sabanis2013E},
Halidias \& Stamatiou~\cite{HalidiasStamatiou2013},
Hutzenthaler et al.~\cite{HutzenthalerJentzenWang2013},
Szpruch \& Zh{\= a}ng~\cite{SzpruchZhang2013},
Halidias~\cite{Halidias2015},
Liu \& Mao~\cite{LiuMao2013},
Hutzenthaler \& Jentzen~\cite{HutzenthalerJentzen2014},
Zhang~\cite{Zhang2014},
Dareiotis et al.~\cite{DareiotisKumarSabanis2014},
Kumar \& Sabanis~\cite{KumarSabanis2014},
Beyn et al.~\cite{BeynIsaakKruse2014},
Zong et al.~\cite{ZongWuHuang2014},
Song et al.~\cite{SongLuLiu2015},
Ngo \& Luong~\cite{NgoLuong2015},
Tambue \& Mukam~\cite{TambueMukam2015},
Mao~\cite{Mao2015},
Beyn et al.~\cite{BeynIsaakKruse2015},
Kumar \& Sabanis~\cite{KumarSabanis2016},
and Mao~\cite{Mao2016}
in the case of finite dimensional SEEs
and see, e.g.,
Gy{\"o}ngy et al.~\cite{GoengySabanisS2014}
and
Jentzen \& Pu{\v s}nik \cite{jp2015}
in the case of infinte dimensional SEEs.
These schemes are suitable modified versions of the explicit Euler scheme that somehow 
tame/truncate the superlinearly growing nonlinearities of the considered SEE 
and thereby prevent the considered tamed scheme from strong divergence.
However, each of the above mentioned temporal strong convergence results
for implicit (see \cite{Hu1996,gm05,kll2015}) 
and explicit (see \cite{HutzenthalerJentzenKloeden2012,
WangGan2013,HutzenthalerJentzen2012,TretyakovZhang2013,
Halidias2013,Sabanis2013,Sabanis2013E,
HalidiasStamatiou2013,HutzenthalerJentzenWang2013,SzpruchZhang2013,Halidias2015,
LiuMao2013,HutzenthalerJentzen2014,Zhang2014,
DareiotisKumarSabanis2014,KumarSabanis2014,GoengySabanisS2014,
BeynIsaakKruse2014,ZongWuHuang2014,SongLuLiu2015,jp2015,NgoLuong2015,
TambueMukam2015,Mao2015,BeynIsaakKruse2015,KumarSabanis2016,Mao2016}) schemes applies
merely to trace noise class driven SEEs and excludes the important case of 
the more irregular space-time white noise.
In particular, none of these results applies to space-time white noise 
driven stochastic Ginzburg-Landau equations.
In this work we intend to close this gap and we propose 
(see~\eqref{eq:exp_scheme} and~\eqref{eq:li_scheme} below) 
and analyze (see Theorem~\ref{thm:main} below) suitable explicit and strongly
convergent approximation schemes
for possibly space-time white noise driven SEEs.
In particular, we establish essentially sharp strong convergence rates 
for suitable explicit nonlinearity-truncated approximation schemes
for space-time white noise driven stochastic Ginzburg-Landau equations 
(see~\eqref{eq:strong_convergence_space_time}, Corollary~\ref{cor:exp_euler_convergence_2}, 
and Corollary~\ref{cor:li_euler_convergence_2} below).

To illustrate the main result 
of this article (see Theorem~\ref{thm:main}
in Section~\ref{sec:main_result_theorem} below)
in more detail,
we consider the following example 
of our
general setting (see Section~\ref{sec:main_result_setting})
in this introductory section.
Let $ T \in (0,\infty) $, $ n \in \{1,3,5,\ldots\} $,
$ a_0, a_1, \ldots, a_{n-1} \in \R $,
$ a_n \in (-\infty, 0) $,
$ \xi \in H^2_0( (0,1); \R) $,
$ H = L^2( (0,1); \R ) $,
let $ F \colon L^{2n}( (0,1); \R ) \rightarrow H $
be the function with the property that for all
$ v \in L^{2n}( (0,1); \R ) $
it holds that
$ F(v) = \sum_{k=0}^n a_k v^k $,
let $ A \colon D(A) \subseteq H \rightarrow H $
be the Laplacian with Dirichlet boundary conditions on $ H $,
let $ (\Omega, \F, \P ) $ be a probability space
with a normal filtration $ (\F_t)_{ t\in[0,T] } $,
and let $ (W_t)_{ t\in[0,T] } $ be an 
$ \Id_{ H } $-cylindrical
$ (\F_t)_{t\in[0,T]} $-Wiener process.
The assumptions 
that $ n $ is odd 
and that $ a_n < 0 $ ensure
that there exists an up to indistinguishability
unique $ (\F_t)_{t\in[0,T]} $-adapted stochastic process
$ X \colon [0,T] \times \Omega \rightarrow L^{2n}((0,1);\R) $
with continuous sample paths
which satisfies that for all $ t \in [0,T] $
it holds $ \P $-a.s.\ that
\begin{align}
\label{eq:intro_SPDE_mild} 
 &X_t
  =
  e^{tA} \xi
  +
  \int_0^t
  e^{(t-s)A}
  F( X_s ) \, ds
  +
  \int_0^t
  e^{(t-s)A} \, dW_s
\end{align}
(cf., e.g., Section~7.2 in Da Prato \& Zabczyk~\cite{dz92}
and Chapter~6 in Cerrai~\cite{Cerrai2001}).
The stochastic process $ X $ is thus 
a mild solution of the SPDE
\begin{align}
\label{eq:intro_SPDE}
 &dX_t(x)
  =
  \left[
    \tfrac{\partial^2}{\partial x^2}
    X_t(x)
    +
    \textstyle\sum\limits_{ k=0 }^n
    a_k
    \left(
      X_t(x)
    \right)^k
  \right] dt
  +
  dW_t(x)
\end{align}
with $ X_0(x) = \xi(x) $ and $ X_t(0) = X_t(1) = 0 $ for
$ x \in (0,1) $, $ t \in [0,T] $.
In the case $ n = 3 $ and $ a_0 = a_2 = 0 $, 
the SPDE~\eqref{eq:intro_SPDE} 
is often referred to as
stochastic Ginzburg-Landau equation in the 
literature (see also \eqref{eq:GinzburgLandau}~below). 
In this article we introduce the following 
nonlinearity-truncated exponential scheme to
approximate the solution process $ X $
of the SPDE~\eqref{eq:intro_SPDE}.
Let $ \fl[h]{\cdot} \colon \R \rightarrow \R $, 
$ h \in (0,\infty) $,
be the mappings with the property that 
for all $ h \in (0,\infty) $, $ t \in \R $ it holds that 
$
  \fl[h]{t}
  =
  \max\!\left(
    \{ 0, h, -h, 2h, -2h, \ldots \} \cap (-\infty, t]
  \right)
$,
let $ \chi \in (0,\frac{1}{2n}] $,
and
let $ Y^{N} \colon [0,T] \times \Omega \to L^{2n^2}\!((0,1);\R) $,
$ N \in \N $,
be stochastic processes which satisfy 
that for all $ t \in [0,T] $, $ N \in \N $
it holds $\P$-a.s.\ that
\begin{align}
\label{eq:exp_scheme}
\begin{split}
 &Y_t^{N}
  =
  e^{tA} \xi
  +
  \int_0^t
  e^{(t-\fl[{T\!/\!N}]{s})A} \,
  \one_{ 
    \{
      \| Y_{\fl[{T\!/\!N}]{s}}^{N} \|_{L^{2n^2}\!((0,1);\R)}
      \leq
      (N/T)^{\chi}
    \}
  } \,
  F(Y_{\fl[{T\!/\!N}]{s}}^{N}) \, ds
  +
  \int_0^t
  e^{(t-\fl[{T\!/\!N}]{s})A} \, dW_s .
\end{split}
\end{align}
The approximation scheme~\eqref{eq:exp_scheme} is a 
nonlinearity-truncated modification
(cf. (2)~in Jentzen \& Pu{\v s}nik \cite{jp2015})
of a time-continuous version
(cf., e.g., (130)--(134)~in Da Prato et al.~\cite{pjr10})
of the classical exponential Euler approximation
scheme for semilinear SPDEs
(cf., e.g., (3.2)~and~(3.6)~in Lord \& Rougemont~\cite{lr04}).
We also propose the following 
nonlinearity-truncated linear-implicit scheme to
approximate the solution process $ X $
of the SPDE~\eqref{eq:intro_SPDE}.
Let $ Z^{N} \colon [0,T] \times \Omega \to L^{2n^2}\!((0,1);\R) $,
$ N \in \N $, 
be stochastic processes which satisfy 
that for all $ t \in [0,T] $, $ N \in \N $
it holds $\P$-a.s.\ that
\begin{align}
\label{eq:li_scheme}
\begin{split}
 &Z_t^{N}
  =
  \big(
    \Id_H - (t-\fl[{T\!/\!N}]{t})A
  \big)^{-1}
  \big(
    \Id_H - \tfrac{T}{N}A
  \big)^{ -\fl[{T\!/\!N}]{t}\frac{N}{T} } 
  \xi
\\&+
  \int_0^t
  \big(
    \Id_H - (t-\fl[{T\!/\!N}]{t})A
  \big)^{-1}
  \big(
    \Id_H - \tfrac{T}{N}A
  \big)^{ (\fl[{T\!/\!N}]{s}-\fl[{T\!/\!N}]{t})\frac{N}{T} } \,
  \one_{ 
    \{
      \| Z_{\fl[{T\!/\!N}]{s}}^{N} \|_{L^{2n^2}\!((0,1);\R)}
      \leq
      (N/T)^{\chi}
    \}
  } \,
  F(Z_{\fl[{T\!/\!N}]{s}}^{N}) \, ds
\\&+
  \int_0^t
  \big(
    \Id_H - (t-\fl[{T\!/\!N}]{t})A
  \big)^{-1}
  \big(
    \Id_H - \tfrac{T}{N}A
  \big)^{ (\fl[{T\!/\!N}]{s}-\fl[{T\!/\!N}]{t})\frac{N}{T} } \, dW_s .
\end{split}
\end{align}
The approximation scheme~\eqref{eq:li_scheme} is a 
nonlinearity-truncated modification
(cf. (2)~in Jentzen \& Pu{\v s}nik \cite{jp2015})
of a time-continuous version
(cf., e.g., (142)--(146)~in Da Prato et al.~\cite{pjr10})
of the classical linear-implicit Euler approximation
scheme for semilinear SPDEs.
Both approximation schemes,
\eqref{eq:exp_scheme} and~\eqref{eq:li_scheme},
are easy to implement (cf.~\eqref{eq:exp_time_step}
and~\eqref{eq:li_time_step} below and cf.,
e.g., also Section~4 in Lord \& Tambue~\cite{lt13}).
In Corollary~\ref{cor:exp_euler_convergence_2} 
and Corollary~\ref{cor:li_euler_convergence_2}
below we prove 
that for all $ p \in (0,\infty) $,
$ \theta \in [ 0, \nicefrac{ 1 }{ 4 } ) $
there exists a real number $ C \in \R $
such that
for all $ N \in \N $
it holds that
\begin{equation}
\label{eq:strong_convergence_space_time}
  \sup\nolimits_{ t \in [0,T] }
  \left(
    \E\!\left[ 
      \| X_t - Y^{N}_t \|_{H}^p
    \right]
  \right)^{ \nicefrac{ 1 }{ p } }
  +
  \sup\nolimits_{ t \in [0,T] }
  \left(
    \E\!\left[ 
      \| X_t - Z^{N}_t \|_{H}^p
    \right]
  \right)^{ \nicefrac{ 1 }{ p } }
  \leq 
  C
  \, N^{ - \theta }
  .
\end{equation}
Our proofs of~\eqref{eq:strong_convergence_space_time},
Corollary~\ref{cor:exp_euler_convergence_2}, 
and Corollary~\ref{cor:li_euler_convergence_2}, respectively, 
are based on the well known idea to substract
the Ornstein-Uhlenbeck process from the solution
process of the SPDE (cf.\ Section~\ref{sec:a_priori} below and, e.g.,
(14.2.2)--(14.2.3) in
Section~14.2 in Da Prato \& Zabczyk~\cite{dz96}),
are based on some arguments in
Jentzen \& Kurniawan~\cite{JentzenKurniawan2015}
and Jentzen \& Pu{\v s}nik~\cite{jp2015}, 
and are based on an
appropriate Lyapunov-type condition
for one-step approximation schemes on 
Banach spaces 
(see~\eqref{eq:assume_u} in Section~\ref{sec:a_priori_setting}, 
\eqref{eq:assume_u_main} in Section~\ref{sec:main_result_setting},
and Corollary~\ref{cor:U_cor1}).
To the best of our knowledge,
inequality~\eqref{eq:strong_convergence_space_time},
Corollary~\ref{cor:exp_euler_convergence_2}, 
and Corollary~\ref{cor:li_euler_convergence_2},
respectively,
are the first results in the literature that
establish strong convergence for temporal
numerical approximations of the SPDE~\eqref{eq:intro_SPDE}
in the case where $ n > 1 $.

In the following we
illustrate~\eqref{eq:strong_convergence_space_time}
through a numerical example.
For this we consider the choice $ T = 1 $, $ n = 3 $, 
$ a_0 = 0 $, $ a_1 = 1 $, $ a_2 = 0 $, $ a_3 = -1 $,
$ \xi = 0 $, $ \chi = \nicefrac{1}{6} = \nicefrac{1}{2n} $
and we note that the SPDE~\eqref{eq:intro_SPDE}
reduces in this case to the stochastic Ginzburg-Landau
equation
\begin{align}
\label{eq:GinzburgLandau}
\begin{split}
  &dX_t(x)
  =
  \left[
    \tfrac{\partial^2}{\partial x^2}
    X_t(x)
    +
    X_t(x)
    -
    \left( X_t(x) \right)^3
  \right] dt
  +
  dW_t(x)
\end{split}
\end{align}
with $ X_0(x) = 0 $ and $ X_t(0) = X_t(1) = 0 $ for
$ x \in (0,1) $, $ t \in [0,T] $.
Note that the approximation processes
in~\eqref{eq:exp_scheme} and~\eqref{eq:li_scheme} satisfy 
in the case of~\eqref{eq:GinzburgLandau} that
for all $ n \in \{ 1, 2, \ldots, N-1 \} $, $ N \in \N $
it holds $ \P $-a.s.\ that
$ Y_0^N = Z_0^N = 0 $ and
\begin{align}
\label{eq:exp_time_step}
  Y_{\nicefrac{(n+1)}{N}}^N
 &=
  e^{\frac{T}{N}A}
  \left[
    Y_{\nicefrac{n}{N}}^N
    +
    \tfrac{T}{N} \,
    \one_{ 
      \{
	\| Y_{\nicefrac{n}{N}}^{N} \|_{L^{18}((0,1);\R)}
	\leq
	N^{\nicefrac{1}{6}}
      \}
    }
    \left(
      Y_{\nicefrac{n}{N}}^{N}
      -
      \left[ 
        Y_{\nicefrac{n}{N}}^{N}
      \right]^3
    \right)
    +
    \smallint_{{\nicefrac{n}{N}}}^{\nicefrac{(n+1)}{N}}
    dW_s
  \right],
\\
\label{eq:li_time_step}
  Z_{\nicefrac{(n+1)}{N}}^N
 &=
  \big(
    \Id_H - \tfrac{T}{N}A
  \big)^{ -1 } 
  \left[
    Z_{\nicefrac{n}{N}}^N
    +
    \tfrac{T}{N} \,
    \one_{ 
      \{
	\| Z_{\nicefrac{n}{N}}^{N} \|_{L^{18}((0,1);\R)}
	\leq
	N^{\nicefrac{1}{6}}
      \}
    }
    \left(
      Z_{\nicefrac{n}{N}}^{N}
      -
      \left[ 
        Z_{\nicefrac{n}{N}}^{N}
      \right]^3
    \right)
    +
    \smallint_{{\nicefrac{n}{N}}}^{\nicefrac{(n+1)}{N}}
    dW_s
  \right] .
\end{align}
In Figure~\ref{fig:convergence_plot}
we plot approximatively the strong root mean 
square approximation errors
\begin{align}
\label{eq:error_approx}
\begin{split}
 &\left(
    \E\!\left[
      \|
        X_T
        -
        Y_T^N
      \|_H^2
    \right] 
  \right)^{\nicefrac{1}{2}}
  \qquad
  \text{and}
  \qquad
  \left(
    \E\!\left[
      \|
        X_T
        -
        Z_T^N
      \|_H^2
    \right] 
  \right)^{\nicefrac{1}{2}}
\end{split}
\end{align}
against the number of time steps 
$ N \in \{ 64, 128, 256, \ldots, 262144 \} $.
The infinite dimensional Hilbert space
$ H = L^2( (0,1); \R ) $ in~\eqref{eq:error_approx}
is approximated in Figure~\ref{fig:convergence_plot}
through the finite dimensional subspace spanned by the
first $ 1024 $ eigenfunctions of the Laplacian,
the unknown exact solution process $ (X_t)_{ t \in [0,T] } $
of~\eqref{eq:GinzburgLandau} is approximated
in Figure~\ref{fig:convergence_plot} through a 
nonlinearity-truncated linear-implicit
Crank-Nicolson approximation with $ 1048576 $
time steps, and the expectations 
in~\eqref{eq:error_approx} are 
approximated in Figure~\ref{fig:convergence_plot} by a Monte Carlo
approximation with $ 25 $ Monte Carlo runs.
Details can be found in Figure~\ref{fig:matlab_code}
in which we present the {\sc Matlab} code
that has been used to create Figure~\ref{fig:convergence_plot}.
The order lines in Figure~\ref{fig:convergence_plot}
correspond to the convergence orders
$ \nicefrac{1}{8} $, $ \nicefrac{1}{4} $,
and $ \nicefrac{1}{2} $.
Figure~\ref{fig:convergence_plot} thus
essentially agrees with our strong
convergence result~\eqref{eq:strong_convergence_space_time}
which proves that both schemes,
\eqref{eq:exp_time_step} and~\eqref{eq:li_time_step},
converge in the sense of~\eqref{eq:error_approx}
with order $ \nicefrac{1}{4}- $.

\begin{figure}[ht]
\centering
\includegraphics[width=\textwidth]{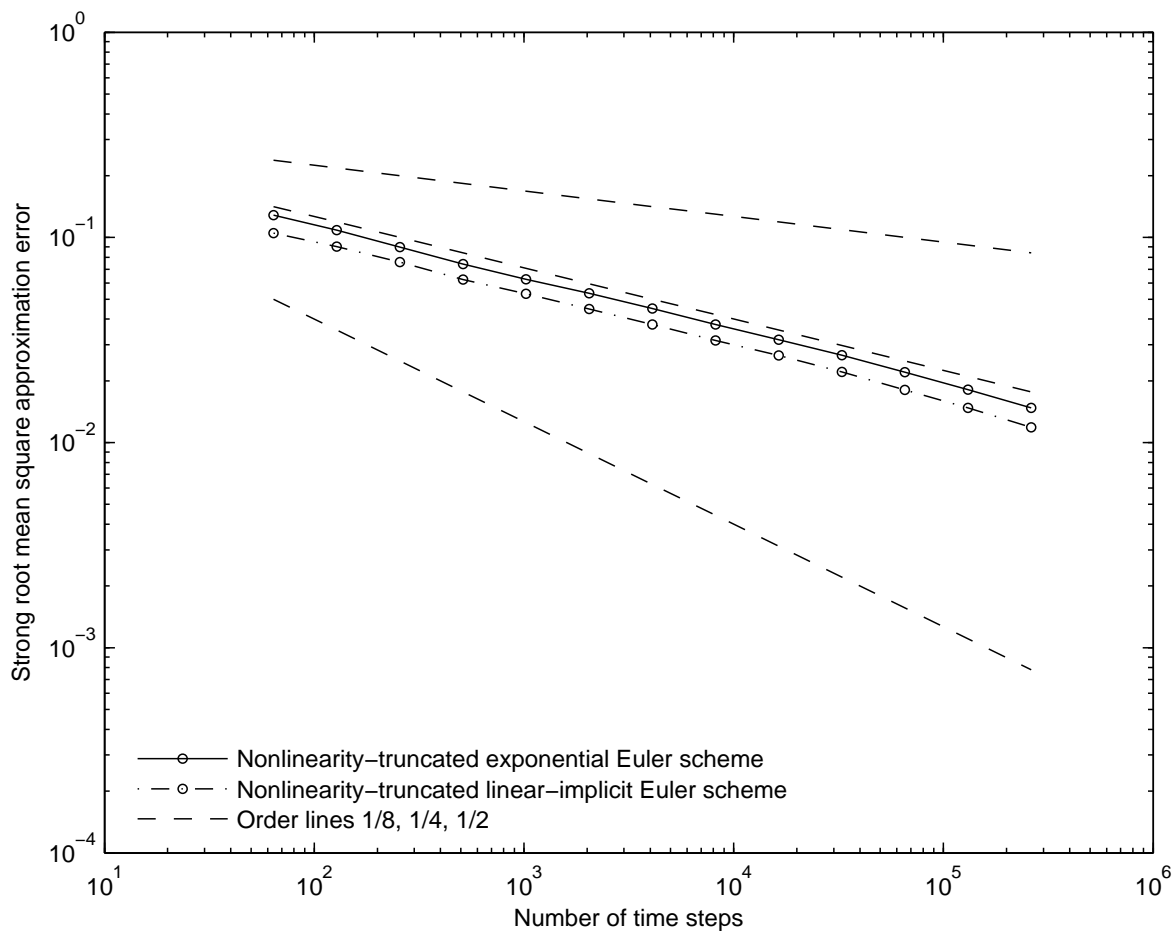}
\caption{Strong approximation errors~\eqref{eq:error_approx} against the 
number of time steps $ N \in \{ 64, 128, $ $ 256, \ldots, 262144 \} $.}
\label{fig:convergence_plot}
\end{figure}

\begin{figure}[p]
\begin{lstlisting}
H_dim = 2^10; N_ref = H_dim^2; N_apx = 2.^(6:18); MC_runs = 25; 
A = - (1:H_dim).^2 .* pi^2; F = @(x) x - x.^3; rng('default');
X = zeros(1, H_dim); Y = zeros(numel(N_apx), H_dim);
Z = zeros(numel(N_apx), H_dim); W = zeros(numel(N_apx), H_dim);
Err_Y = zeros(1,numel(N_apx)); Err_Z = zeros(1,numel(N_apx));

for r=1:MC_runs 
  X = X*0; Y = Y*0; Z = Z*0; W = W*0;

  for i = 1:N_ref
    dW = randn(1, H_dim) / sqrt(N_ref);
    x = F( dst( X ) * sqrt(2) );
    x = x*( ( sum( x.^18 )/(H_dim+1) ) <= N_ref^3 );
    X = ( (1+A/N_ref/2) .* X + idst( x ) / sqrt(2) / N_ref + dW ) ...
	./ (1-A/N_ref/2);
    
    for k = 1:numel(N_apx)   
      W(k,:) = W(k,:) + dW;
      if mod( i, N_ref / N_apx(k) ) == 0   
	x = F( dst( Y(k,:) ) * sqrt(2) );
	x = x*( ( sum( x.^18 )/(H_dim+1) ) <= N_apx(k)^3 );
	Y(k,:) = exp(A/N_apx(k)) ...
		 .* ( Y(k,:) + idst( x ) / sqrt(2) / N_apx(k) + W(k,:) );
	
	x = F( dst( Z(k,:) ) * sqrt(2) );
	x = x*( ( sum( x.^18 )/(H_dim+1) ) <= N_apx(k)^3 );
	Z(k,:) = ( Z(k,:) + idst( x ) / sqrt(2) / N_apx(k) + W(k,:) ) ...
	         ./ (1-A/N_apx(k));
	
	W(k,:) = W(k,:)*0;         
      end 
    end 
  end
  
  for k = 1:numel(N_apx)
    Err_Y(k) = Err_Y(k) + norm(X - Y(k,:))^2;
    Err_Z(k) = Err_Z(k) + norm(X - Z(k,:))^2;
  end
end

Err_Y = sqrt( Err_Y / MC_runs ); Err_Z = sqrt( Err_Z / MC_runs );
loglog( N_apx, Err_Y, 'o- k', 'MarkerSize', 3); hold on;
loglog( N_apx, Err_Z, 'o-. k', 'MarkerSize', 3);
loglog( N_apx, 0.4*[N_apx.^(-1/8); N_apx.^(-1/4); N_apx.^(-1/2)], '-- k');
xlabel('Number of time steps');
ylabel('Strong root mean square approximation error');  
legend('Location', 'SouthWest', ...
       'Nonlinearity-truncated exponential Euler scheme', ...
       'Nonlinearity-truncated linear-implicit Euler scheme', ...
       'Order lines 1/8, 1/4, 1/2'); legend('boxoff');
\end{lstlisting}
\caption{ {\sc Matlab} code used to create Figure~\ref{fig:convergence_plot}. }
\label{fig:matlab_code}
\end{figure}

The remainder of this article is organized as follows.
In Section~\ref{sec:a_priori} the required a priori moments for the nonlinearity-truncated 
approximation schemes are established,
in Section~\ref{sec:pathwise_error} 
the error analysis is performed in the pathwise sense under the hypothesis of suitable a priori bounds
for the approximation processes,
and in Section~\ref{sec:strong_error}
the error analysis is carried out in the strong $ L^p $-sense 
under the hypothesis of appropriate a priori moment bounds
for the approximation processes.
Section~\ref{sec:main_result} combines the results of Section~\ref{sec:a_priori}
and Section~\ref{sec:strong_error} and thereby establishes Theorem~\ref{thm:main}
which is the main result of this article.
The analysis in Sections~\ref{sec:a_priori}--\ref{sec:main_result}
is carried out for abstract stochastic evolution equations
on separable Banach and Hilbert spaces, respectively.
Section~\ref{sec:examples} then verifies that the assumptions of Theorem~\ref{thm:main}
in Section~\ref{sec:main_result} are satisfied in the case of concrete 
stochastic partial differential equations of the type~\eqref{eq:intro_SPDE} and, in particular,
establishes Corollary~\ref{cor:exp_euler_convergence_2} 
and Corollary~\ref{cor:li_euler_convergence_2}.

\subsection{Notation}
Throughout this article the following notation is used.
For measurable spaces $ (\Omega_1, \F_1) $
and $ (\Omega_2, \F_2) $ we denote by
$ \M( \F_1, \F_2 ) $ the
set of all $ \F_1 / \F_2 $-measurable mappings.
For a set $ A $ we denote by $ \mathcal{P}(A) $
the power set of $ A $.
For a set $ A $ and a subset $ \mathcal{A} \subseteq \mathcal{P}(A) $
we denote by $ \sigma_A( \mathcal{A} ) $ the smallest 
sigma-algebra on $ A $ which contains $ \mathcal{A} $.
For a topological space $ (X, \tau) $ we denote by 
$ \B(X) $ the set given by 
$ \B(X) = \sigma_X( \tau ) $.
For a natural number $ d \in \N $ and a 
set $ A \in \B(\R^d) $
we denote by $ \lambda_A \colon \B(A) \rightarrow [0, \infty] $
the Lebesgue-Borel measure on $ A $.
We denote by $ \fl[h]{\cdot} \colon \R \rightarrow \R $, 
$ h \in (0,\infty) $, and
$ \cl[h]{\cdot} \colon \R \rightarrow \R $, 
$ h \in (0,\infty) $,
the mappings with the property that 
for all $ h \in (0,\infty) $, $ t \in \R $ it holds that 
$
  \fl[h]{t}
  =
  \max\!\left(
    \{ 0, h, -h, 2h, -2h, \ldots \} \cap (-\infty, t]
  \right)
$
and
$
  \cl[h]{t}
  =
  \min\!\left(
    \{ 0, h, -h, 2h, -2h, \ldots \} \cap [t,\infty)
  \right)
$. 
For a measure space
$ (\Omega, \F, \mu) $,
a measurable space 
$ (S, \S) $,
a set $ R \subseteq S $,
and a
function $ f \colon \Omega \rightarrow R $
we denote by $ [f]_{\mu,\S} $ the set given by
$ 
  [f]_{\mu,\S} 
  =
  \{
    g \in \M(\F,\S)
    \colon
    ( 
      \exists \, A \in \F
      \colon
      \mu(A) = 0
      \text{ and }
      \{
        \omega \in \Omega
        \colon
        f(\omega)
        \neq
        g(\omega)
      \}
      \subseteq A
    )
  \}
$.

\section{A priori bounds}
\label{sec:a_priori}
\subsection{Setting}
\label{sec:a_priori_setting}
Let $ T, \qf \in (0,\infty) $, 
$ c, C \in [0,\infty) $,
$ M \in \N $,
let $ (\BaMi, \left\| \cdot \right\|_{\BaMi}) $
and $ (\BaLa, \left\| \cdot \right\|_{\BaLa} ) $
be separable $ \R $-Banach spaces with 
$ \BaMi \subseteq \BaLa $ densely and continuously,
let $ \Um, \Rm, \Vm \in \M\big( \B(\BaMi), \B([0,\infty)) \big) $,
$ F \in \M\big( \B(\BaMi), \B(\BaLa) \big) $,
$ \S \in \M\big( \B(\{ (s, t) \in [0,T]^2 \colon s < t \}), \B(L(\BaLa,\BaMi)) \big) $
satisfy that for all $ u,v \in \BaMi $,
$ r_1,r_2,r_3 \in [0,T] $,
$ s \in \{ 0, \frac{T}{M}, \frac{2T}{M}, \ldots, \frac{(M-1)T}{M} \} $,
$ t \in (s,s+\frac{T}{M} ] $ with
$ r_1 < r_2 < r_3 $ it holds that
$ 
  \left\| F(u + v) \right\|_{ \BaLa } 
  \leq 
  C( 1 + |\Um(u)|^{\qf} + |\Um(v)|^{\qf})
$,
$ 
  \S_{r_1,r_3}
  =
  \S_{r_2,r_3}
  \S_{r_1,r_2}
$,
and
\begin{align}
\label{eq:assume_u}
\begin{split}
 &\Um\!\left(
    \S_{s,t}
    \left[
      u
      +
      (t-s) \,
      \one_{ 
	[0, M/T ]
      }
      ( \Vm(u+v) ) \,
      F( u + v )
    \right]
  \right)
  \leq
  e^{c(t-s)}
  \left[
    \Um(u)
    +
    (t-s)
    \Rm(v)
  \right], 
\end{split}
\end{align}
let $ (\Omega, \F, \P ) $ 
be a probability space,
and let 
$ O, \Yb{} \colon [0,T] \times \Omega \rightarrow \BaMi $ 
be stochastic processes such that for all 
$ t \in (0,T] $ it holds that
\begin{align}
\label{eq:bootstrap_Y}
\begin{split}
  \Yb{t}
 &=
  \S_{0,t}
  \Yb{0}
  +
  \int_0^t
  \S_{\fl{s},t} \,
  \one_{ 
    \{
      \Vm(
        \Yb{\fl{s}}
        +
        O_{\fl{s}}
      )
      \leq
      M/T
    \}
  } \,
  F( \Yb{\fl{s}} + O_{\fl{s}} ) \, ds .
\end{split}
\end{align}

\subsection{A priori bounds based on variational arguments}
\label{sec:variational_a_priori}

\begin{lemma}
\label{lem:variational_a_priori}
Assume the setting in Section~\ref{sec:a_priori_setting}. 
Then for all $ t \in [0,T] $ it holds that
\begin{align}
\begin{split}
  \Um(\Yb{t})
 &\leq
  e^{ct}
  \Um( \Yb{0} )
  +
  \int_0^{t}
  e^{c(t-\fl{s})}
  \Rm( O_{\fl{s}} ) \, ds .
\end{split}
\end{align}
\end{lemma}
\begin{proof}[Proof of Lemma~\ref{lem:variational_a_priori}]
Note that~\eqref{eq:bootstrap_Y}
shows that for all 
$ s \in \{ 0, \frac{T}{M}, \frac{2T}{M}, \ldots, \frac{(M-1)T}{M} \} $,
$ t \in (s, s+\frac{T}{M}] $
it holds that
\begin{align}
\begin{split}
 &\Yb{t}
  =
  \S_{s,t}
  \left[
    \Yb{s}
    +
    (t-s) \,
    \one_{ 
      \{
	\Vm(
	  \Yb{s}
	  +
	  O_{s}
	)
	\leq
	M/T
      \}
    } \,
    F( \Yb{s} + O_{s} )
  \right] .
\end{split}
\end{align}
This and~\eqref{eq:assume_u} 
prove that for all
$ s \in \{ 0, \frac{T}{M}, \frac{2T}{M}, \ldots, \frac{(M-1)T}{M} \} $,
$ t \in (s, s+\frac{T}{M}] $
it holds that
\begin{align}
\label{eq:a_priori_h1}
\begin{split}
  \Um(\Yb{t})
 &\leq 
  e^{c(t-s)}
  \left[
    \Um(\Yb{s})
    +
    (t-s)
    \Rm( O_{s} )
  \right] .
\end{split}
\end{align}
In particular, \eqref{eq:a_priori_h1}
implies that for all
$ m \in \{ 1, 2, \ldots, M \} $
it holds that
\begin{align}
\label{eq:a_priori_h2}
\begin{split}
  \Um(\Yb{\frac{mT}{M}})
 &\leq 
  e^{c\frac{T}{M}}
  \left[
    \Um(\Yb{\frac{(m-1)T}{M}})
    +
    \tfrac{T}{M}
    \Rm( O_{\frac{(m-1)T}{M}} )
  \right]
  \leq
  \ldots
  \leq
  e^{c\frac{mT}{M}}
  \Um( \Yb{0} )
  +
  \tfrac{T}{M}
  \sum_{l=0}^{m-1}
  e^{c(m-l)\frac{T}{M}}
  \Rm( O_{\frac{lT}{M}} )
\\&=
  e^{c\frac{mT}{M}}
  \Um( \Yb{0} )
  +
  \sum_{l=0}^{m-1}
  \int_{\frac{lT}{M}}^{\frac{(l+1)T}{M}}
  e^{c(\frac{mT}{M}-\fl{s})}
  \Rm( O_{\fl{s}} ) \, ds
\\&=
  e^{c\frac{mT}{M}}
  \Um( \Yb{0} )
  +
  \int_0^{\frac{mT}{M}}
  e^{c(\frac{mT}{M}-\fl{s})}
  \Rm( O_{\fl{s}} ) \, ds .
\end{split}
\end{align}
Moreover, \eqref{eq:a_priori_h1} 
and~\eqref{eq:a_priori_h2} yield
that for all 
$ t \in [0,T] $
it holds that
\begin{align}
\begin{split}
 &\Um(\Yb{t})
\\&\leq 
  e^{c(t-\fl{t})}
  \left[
    \Um(\Yb{\fl{t}})
    +
    (t-\fl{t})
    \Rm( O_{\fl{t}} )
  \right]
\\&\leq 
  e^{c(t-\fl{t})}
  \left[
    e^{c\fl{t}}
    \Um( \Yb{0} )
    +
    \int_0^{\fl{t}}
    e^{c(\fl{t}-\fl{s})}
    \Rm( O_{\fl{s}} ) \, ds
    +
    (t-\fl{t})
    \Rm( O_{\fl{t}} )
  \right]
\\&=
  e^{ct}
  \Um( \Yb{0} )
  +
  \int_0^{t}
  e^{c(t-\fl{s})}
  \Rm( O_{\fl{s}} ) \, ds .
\end{split}
\end{align}
This completes the proof of Lemma~\ref{lem:variational_a_priori}.
\end{proof}

\subsection{A priori bounds based on bootstrap-type arguments}
\label{sec:bootstrap_a_priori}

\begin{lemma}
\label{lem:bootstrap_F}
Assume the setting in Section~\ref{sec:a_priori_setting} 
and let $ t \in [0,T] $. Then
\begin{align}
\begin{split}
 &\big\|
    F( \Yb{t} + O_t )
  \big\|_{ \BaLa }
  \leq
  C
  \left(
    1
    +
    2^{[\qf-1]^+}
    e^{c\qf t}
    \left[ 
      \big|
	\Um( \Yb{0} )
      \big|^{\qf}
      +
      \left|  
	\smallint_0^{t}
	\Rm( O_{\fl{s}} ) \, ds
      \right|^{\qf}
    \right]
    +
    \big|
      \Um( O_{t} )
    \big|^{\qf}
  \right) .
\end{split}
\end{align}
\end{lemma}
\begin{proof}[Proof of Lemma~\ref{lem:bootstrap_F}]
Note that the assumption that 
$ 
  \forall \, u,v \in \BaMi 
  \colon
  \left\| F(u + v) \right\|_{ \BaLa } 
  \leq 
  C( 1 + |\Um(u)|^{\qf} + |\Um(v)|^{\qf})
$,
Lemma~\ref{lem:variational_a_priori},
and the fact that
$ 
  \forall \,
  x,y \in \R
$,
$
  r \in (0, \infty) 
  \colon
  | x + y |^r 
  \leq 
  2^{[r-1]^+} 
  | x |^r
  +
  2^{[r-1]^+} 
  | y |^r
$
imply that
\begin{align}
\begin{split}
  \big\|
    F( \Yb{t} + O_t )
  \big\|_{ \BaLa }
 &\leq
  C
  \left(
    1
    +
    \big|
      \Um( \Yb{t} )
    \big|^{\qf}
    +
    \big|
      \Um( O_{t} )
    \big|^{\qf}
  \right)
\\&\leq 
  C
  \left(
    1
    +
    \left|
      e^{ct}
      \Um( \Yb{0} )
      +
      \int_0^t
      e^{c(t-\fl{s})}
      \Rm( O_{\fl{s}} ) \, ds
    \right|^{\qf}
    +
    \big|
      \Um( O_{t} )
    \big|^{\qf}
  \right)
\\&\leq
  C
  \left(
    1
    +
    2^{[{\qf}-1]^+}
    \big|
      e^{ct}
      \Um( \Yb{0} )
    \big|^{\qf}
    +
    2^{[{\qf}-1]^+}
    \left|  
      \int_0^{t}
      e^{c(t-\fl{s})}
      \Rm( O_{\fl{s}} ) \, ds
    \right|^{\qf}
    +
    \big|
      \Um( O_{t} )
    \big|^{\qf}
  \right) .
\end{split}
\end{align}
This completes the proof 
of Lemma~\ref{lem:bootstrap_F}.
\end{proof}

\begin{corollary}
\label{cor:bootstrap_Y_1}
Assume the setting in Section~\ref{sec:a_priori_setting}
and let $ t \in (0,T] $, $ \sf \in [0,1) $. Then
\begin{align}
\begin{split}
  \left\|
    \Yb{t}
  \right\|_{\BaMi}
 &\leq
  \left\|
    \S_{0,t}
  \right\|_{ L(\BaMi) }
  \left\|
    \Yb{0}
  \right\|_{ \BaMi }
  +
  C
  \left[
    \sup_{ s \in [0,t) }
    \left(t-s\right)^{\sf}
    \big\|
      \S_{\fl{s},t}
    \big\|_{ L( \BaLa, \BaMi ) }
  \right]
  \Bigg[
    \int_0^t
    \left( t-s \right)^{-\sf}
    \big|
      \Um( O_{\fl{s}} )
    \big|^{\qf} \, ds
\\&\quad+
    \frac{t^{(1-\sf)}}{(1-\sf)}
    \left(
      1
      +
      2^{[\qf-1]^+}
      e^{c\qf t}
      \left[
        \big|
          \Um( \Yb{0} )
        \big|^{\qf}
        +
        \Big|
          \smallint_0^{t}
          \Rm( O_{\fl{u}} ) \, du 
        \Big|^{\qf}
      \right]
    \right)
  \Bigg] .
\end{split}
\end{align}
\end{corollary}
\begin{proof}[Proof of Corollary~\ref{cor:bootstrap_Y_1}]
Note that Lemma~\ref{lem:bootstrap_F} implies
\begin{align}
\begin{split}
 \left\|
    \Yb{t}
  \right\|_{\BaMi}
 &\leq
  \left\|
    \S_{0,t}
    \Yb{0}
  \right\|_{ \BaMi }
  +
  \int_0^t
  \big\|
    \S_{\fl{s},t}
  \big\|_{ L( \BaLa, \BaMi ) }
  \big\|
    \one_{ 
      \{
        \Vm(
          \Yb{\fl{s}}
          +
          O_{\fl{s}}
        )
        \leq
        M/T
      \}
    } \,
    F( \Yb{\fl{s}} + O_{\fl{s}} )
  \big\|_{ \BaLa } \, ds
\\&\leq
  \left\|
    \S_{0,t}
  \right\|_{ L(\BaMi) }
  \left\|
    \Yb{0}
  \right\|_{ \BaMi }
\\&\quad+
  \left[
    \sup_{ s \in [0,t) }
    \left(t-\fl{s}\right)^{\sf}
    \big\|
      \S_{\fl{s},t}
    \big\|_{ L( \BaLa, \BaMi ) }
  \right]
  \int_0^t
  \big( t-\fl{s} \big)^{-\sf}
  \big\|
    F( \Yb{\fl{s}} + O_{\fl{s}} )
  \big\|_{ \BaLa } \, ds
\\&\leq
  \left\|
    \S_{0,t}
  \right\|_{ L(\BaMi) }
  \left\|
    \Yb{0}
  \right\|_{ \BaMi }
  +
  C
  \left[
    \sup_{ s \in [0,t) }
    \left(t-s\right)^{\sf}
    \big\|
      \S_{\fl{s},t}
    \big\|_{ L( \BaLa, \BaMi ) }
  \right]
  \Bigg[
    \int_0^t
    \left( t-s \right)^{-\sf}
    \Bigg(
      1
      +
      \big|
        \Um( O_{\fl{s}} )
      \big|^{\qf}
\\&\quad+
      2^{[\qf-1]^+}
      e^{c\qf\fl{s}}
      \left[
	\big|
	  \Um( \Yb{0} )
	\big|^{\qf}
	+
	\left|  
	  \smallint_0^{\fl{s}}
	  \Rm( O_{\fl{u}} ) \, du
	\right|^{\qf}
      \right]
    \Bigg) \, ds
  \Bigg] .
\end{split}
\end{align}
This completes the proof of Corollary~\ref{cor:bootstrap_Y_1}.
\end{proof}

\begin{corollary}
\label{cor:a_priori_Lp}
Assume the setting in Section~\ref{sec:a_priori_setting} 
and let $ p \in [\max\{1,\nicefrac{1}{\qf}\}, \infty) $, 
$ t \in (0,T] $, $ \sf \in [0,1) $. Then
\begin{align}
\begin{split}
 &\left\|
    \Yb{t}
  \right\|_{\L^p( \P; \BaMi )}
  \leq
  \left\|
    \S_{0,t}
  \right\|_{ L(\BaMi) }
  \left\|
    \Yb{0}
  \right\|_{ \L^p( \P; \BaMi ) }
  +
  \frac{C t^{(1-\sf)}}{(1-\sf)}
  \left[
    \sup_{ s \in [0,t) }
    \left(t-s\right)^{\sf}
    \big\|
      \S_{\fl{s},t}
    \big\|_{ L( \BaLa, \BaMi ) }
  \right]
\\&\quad\cdot
  \Bigg(
    1
    +
    \sup_{ s \in [0,t) }
    \big\|
      \Um( O_{s} )
    \big\|_{ \L^{p\qf}(\P; \R) }^{\qf}
    +
    2^{[\qf-1]^+}
    e^{c\qf t}
    \left[ 
      \big\|
	\Um( \Yb{0} )
      \big\|_{ \L^{p\qf}(\P;\R) }^{\qf}
      +
      t^{\qf}
      \sup_{ s \in [0,t) }
      \big\|
	\Rm( O_{s} )
      \big\|_{ \L^{p\qf}(\P; \R) }^{\qf}
    \right]
  \Bigg) .
\end{split}
\end{align}
\end{corollary}
\begin{proof}[Proof of Corollary~\ref{cor:a_priori_Lp}]
Observe that
Corollary~\ref{cor:bootstrap_Y_1} implies
\begin{align}
\begin{split}
 &\left\|
    \Yb{t}
  \right\|_{\L^p( \P; \BaMi )}
  \leq
  \left\|
    \S_{0,t}
  \right\|_{ L(\BaMi) }
  \left\|
    \Yb{0}
  \right\|_{ \L^p( \P; \BaMi ) }
\\&\quad+
  C
  \left[
    \sup_{ s \in [0,t) }
    \left(t-s\right)^{\sf}
    \big\|
      \S_{\fl{s},t}
    \big\|_{ L( \BaLa, \BaMi ) }
  \right]
  \Bigg[
    \int_0^t
    \left( t-s \right)^{-\sf}
    \big\|
      \Um( O_{\fl{s}} )
    \big\|_{ \L^{p\qf}(\P; \R) }^{\qf} \, ds
\\&\quad+
    \frac{t^{(1-\sf)}}{(1-\sf)}
    \left(
      1
      +
      2^{[\qf-1]^+}
      e^{c\qf t}
      \left[
        \big\|
          \Um( \Yb{0} )
        \big\|_{ \L^{p\qf}(\P;\R) }^{\qf}
        +
        \left|
        \smallint_0^{t}
	  \big\|
	    \Rm( O_{\fl{u}} )
	  \big\|_{ \L^{p\qf}(\P; \R) } \, du
        \right|^{\qf}
      \right]
    \right)
  \Bigg] .
\end{split}
\end{align}
The proof of Corollary~\ref{cor:a_priori_Lp} 
is thus completed.
\end{proof}

\section{Pathwise error estimates}
\label{sec:pathwise_error}
\subsection{Setting}
\label{sec:pathwise_setting}
Let $ T, C, \qs \in (0,\infty) $, 
$ L \in [0,\infty) $, 
$ \sr \in [0,1) $,
$ M \in \N $,
let $ ( H, \left< \cdot, \cdot \right>_H, \left\| \cdot \right\|_H ) $
be a separable $ \R $-Hilbert space,
let $ A \colon D(A) \subseteq H \rightarrow H $
be a generator of an analytic semigroup 
with
$ 
  \textup{spectrum}(A) 
  \subseteq 
  \{
    z \in \mathbb{C} 
    \colon 
    \textup{Re}(z) < 0
  \}
$,
let $ ( H_r, \left< \cdot, \cdot \right>_{ H_r }, \left\| \cdot \right\|_{ H_r } ) $, 
$ r \in \R $, be a family of interpolation spaces associated to $ -A $
(cf., e.g., Definition~3.5.26 in~\cite{Jentzen2015SPDElecturenotes}),
let $ (\BaMi, \left\| \cdot \right\|_{\BaMi}) $
be a separable $ \R $-Banach space with 
$ H_1 \subseteq \BaMi \subseteq H $ 
densely and continuously,
let $ \Vm \in \M\big( \B(\BaMi), \B( [0,\infty) ) \big) $,
$ \S \in \M\big( \B(\{ (s, t) \in [0,T]^2 \colon s < t \}), \B(L(H, \BaMi) ) \big) $,
$ F \in \C(\BaMi,H) $
satisfy for all $ x, y \in \BaMi $
with $ x-y \in H_1 $ that
$
  \left\|
    F(x) - F(y)
  \right\|_{H}^2
  \leq
  L
  \left\|
    x - y
  \right\|_{\BaMi}^2
  \left(
    1
    +
    \left\| x \right\|_{\BaMi}^{\qs}
    +
    \left\| y \right\|_{\BaMi}^{\qs}
  \right)
$
and
$
  \left<
    x - y,
    A(x-y) + F(x)
    - F(y)
  \right>_H
  \leq
  C
  \left\| x - y \right\|_H^2
$,
assume that
$
  \sup_{ t \in (0,T) }
  t^{\sr}
  \|
    e^{tA}
  \|_{ L(H,\BaMi) }
  <
  \infty
$,
let $ (\Omega, \F, \P ) $ 
be a probability space,
let $ X \colon [0,T] \times \Omega \rightarrow \BaMi $
be a stochastic process with continuous sample paths,
let $ O \colon [0,T] \times \Omega \rightarrow \BaMi $
be a stochastic process
which satisfies for all $ \omega \in \Omega $ that
$
  \limsup_{ r \searrow 0 }
  \sup_{ 0 \leq s < t \leq T }
  \frac{ 
    s \| O_t(\omega) - O_s(\omega) \|_{ \BaMi } 
  }{
    (t-s)^{r}
  }
  <
  \infty
$,
let 
$ 
    \O{}, \XO{}, \Y{}, \YO{}, \YX{} 
    \colon 
    [0,T] \times \Omega 
    \rightarrow \BaMi 
$
be stochastic processes,
and assume for all 
$ t \in [0,T] $ that
$
  X_t
  =
  \int_0^t
  e^{(t-s)A}
  F(X_s) \, ds
  +
  O_t
$,
$ \XO{t} = X_t - O_t $, 
$
  \YX{t}
  =
  \int_0^t
  e^{(t-s)A}
  F( \Y{\fl{s}} ) \, ds
$,
$
  \Y{t}
  =
  \int_0^t
  \S_{\fl{s},t} \,
  \one_{ 
    \{
      \Vm( \Y{\fl{s}} )
      \leq
      M/T
    \}
  } \,
  F( \Y{\fl{s}} ) \, ds
  +
  \O{t} 
$,
and
$ \YO{t} = \Y{t} - \O{t} $.

\subsection{Regularity of the exact solution}
\label{sec:exact_regularity}

The next elementary lemma 
is a slightly modified version of
Theorem~3.5 in Section~4.3 in Pazy~\cite{p83}.

\begin{lemma}
\label{lem:in_DA_continuous}
Let $ T \in (0,\infty) $,
$ \alpha \in (0,1] $,
let 
$ 
  ( W, \left\| \cdot \right\|_W ) 
$
be an $ \R $-Banach space,
let $ A \colon D(A) \subseteq W \rightarrow W $
be a generator of an analytic semigroup 
with
$ 
  \textup{spectrum}(A) 
  \subseteq 
  \{
    z \in \mathbb{C} 
    \colon 
    \textup{Re}(z) < 0
  \}
$, let
$ f \colon [0, T] \rightarrow W $ be a
continuous function which satisfies
$
  \sup_{ 0 \leq s < t \leq T }
  \frac{ 
    s
    \left\| f(t) - f(s) \right\|_W
  }{ 
    \left(t-s\right)^{\alpha}
  }
  <
  \infty
$,
and let
$ x \colon [0,T] \rightarrow W $ be the
function with the property that
for all $ t \in [0,T] $ it holds that
$ x(t) = \int_0^t e^{(t-s)A} f(s) \, ds $.
Then

\begin{enumerate}[(i)]
\item\label{it:in_DA_1}
it holds for all $ t \in [0,T] $
that $ x(t) \in D(A) $,
\item\label{it:in_DA_2}
it holds that the function
$
  (0,T]
  \ni
  t
  \mapsto
  x(t)
  \in 
  D(A)
$
is continuous,
\item\label{it:in_DA_3}
it holds that the function
$
  [0,T]
  \ni
  t
  \mapsto
  x(t)
  \in 
  W
$
is continuous,
\item\label{it:in_DA_4}
it holds that the function
$
  (0,T]
  \ni
  t
  \mapsto
  x(t)
  \in 
  W
$
is continuously differentiable, and
\item\label{it:in_DA_5}
it holds for all $ t \in (0,T] $
that
$
  x'(t)
  =
  Ax(t)
  +
  f(t)
$.
\end{enumerate}
\end{lemma}
\begin{proof}[Proof of Lemma~\ref{lem:in_DA_continuous}]
Throughout this proof assume w.l.o.g.\ that 
$ W \neq \{ 0 \} $ and let $ K \in (0,\infty) $ be 
the real number which satisfies
$
  K
  =
  \sup_{ 0 \leq s < u \leq T }
  \sup_{r \in \{0,\frac{\alpha}{2},\frac{1}{2},1\}}
  \big[
    \|
      (-uA)^r
      e^{uA}
    \|_{ L(W) }
    +
    \|
      (-uA)^{-r}
      (
        e^{uA}
        -
        \Id_W
      )
    \|_{ L(W) }
    +
    \|
      f(u)
    \|_W
    +
    \frac{ 
      s
      \left\| f(u) - f(s) \right\|_W
    }{ 
      \left(u-s\right)^{\alpha}
    }
  \big]
$.
The assumption that
$ A \colon D(A) \subseteq W \rightarrow W $
is a generator of an analytic semigroup 
with 
$ 
  \textup{spectrum}(A) 
  \subseteq 
  \{
    z \in \mathbb{C} 
    \colon 
    \textup{Re}(z) < 0
  \}
$
and the assumption that
$ f \colon [0, T] \rightarrow W $ is a
continuous function with the property
that 
$
  \sup_{ 0 \leq s < u \leq T }
  \frac{ 
    s
    \left\| f(u) - f(s) \right\|_W
  }{ 
    \left(u-s\right)^{\alpha}
  }
  <
  \infty
$
assure that such a real number
does indeed exist.
Next note that for all
$ t \in (0,T] $, $ \varepsilon \in (0,t) $
it holds that
\begin{align}
\label{eq:in_DA_continuous_1}
\begin{split}
  x(t)
 &=
  e^{(t-\varepsilon)A}
  \int_0^{\varepsilon}
  e^{(\varepsilon-s)A}
  f(s) \, ds
  +
  \int_{\varepsilon}^t
  e^{(t-s)A}
  f(s) \, ds
\\&=
  e^{(t-\varepsilon)A}
  x(\varepsilon)
  +
  \int_{\varepsilon}^t
  e^{(t-s)A}
  \left[ 
    f(s)
    -
    f(t)
  \right] ds
  +
  \int_0^{(t-\varepsilon)}
  e^{sA}
  f(t) \, ds
\\&=
  e^{(t-\varepsilon)A}
  x( \varepsilon )
  +
  \int_{\varepsilon}^t
  e^{(t-s)A}
  \left[ 
    f(s)
    -
    f(t)
  \right] ds
  +
  A^{-1}
  \left(
    e^{(t-\varepsilon)A}
    -
    \Id_W
  \right)
  f(t) .
\end{split}
\end{align}
Moreover, observe that for all
$ t \in (0,T] $, $ \varepsilon \in (0,t) $
it holds that
\begin{align}
\label{eq:in_DA_continuous_2}
\begin{split}
  \int_{\varepsilon}^{t}
  \left\|
    A
    e^{(t-s)A}
    \left[
      f(s)
      -
      f(t)
    \right] 
  \right\|_W ds
 &\leq
  \int_{\varepsilon}^{t}
  \big\|
    A
    e^{(t-s)A}
  \big\|_{L(W)}
  \big\|
    f(s)
    -
    f(t)
  \big\|_W \, ds
\\&\leq
  \frac{ 
    K^2
  }{ 
    \varepsilon
  }
  \int_{\varepsilon}^{t}
  \left(
    t-s
  \right)^{(\alpha-1)} ds
  \leq
  \frac{
    K^2 t^{\alpha}
  }{
    \alpha
    \varepsilon
  }
  <
  \infty .
\end{split}
\end{align}
Combining 
the fact that 
$ x(0) = 0 $,
\eqref{eq:in_DA_continuous_1},
and \eqref{eq:in_DA_continuous_2} 
proves~\eqref{it:in_DA_1}.
In addition, note that
for all $ t \in (0,T] $, $ \varepsilon \in (0,t) $,
$ t_1, t_2 \in (\varepsilon,T] $
with $ t_1 \leq t \leq t_2 $ it holds that
\begin{align}
\begin{split}
 &x(t_2)
  -
  x(t_1)
\\&=
  \int_{t_1}^{t_2}
  e^{(t_2-s)A}
  f(s) \, ds
  +
  \left(
    e^{(t_2-t_1)A}
    -
    \Id_W
  \right)
  \int_0^{t_1}
  e^{(t_1-s)A}
  f(s) \, ds  
\\&=
  \int_{t_1}^{t_2}
  e^{(t_2-s)A}
  \left[
    f(s)
    -
    f(t_2)
  \right] ds
  +
  \int_{t_1}^{t_2}
  e^{(t_2-s)A}
  f(t_2) \, ds
  +
  \left(
    e^{(t_2-t_1)A}
    -
    \Id_W
  \right)
  \int_0^{\varepsilon}
  e^{(t_1-s)A}
  f(s) \, ds
\\&\quad+
  \left(
    e^{(t_2-t_1)A}
    -
    \Id_W
  \right)
  \int_{\varepsilon}^{t_1}
  e^{(t_1-s)A}
  f(t_1) \, ds 
  +
  \left(
    e^{(t_2-t_1)A}
    -
    \Id_W
  \right)
  \int_{\varepsilon}^{t_1}
  e^{(t_1-s)A}
  \left[
    f(s)
    -
    f(t_1)
  \right] ds .
\end{split}
\end{align}
Therefore, we obtain that
for all $ t \in (0,T] $, $ \varepsilon \in (0,t) $,
$ t_1, t_2 \in (\varepsilon,T] $
with $ t_1 \leq t \leq t_2 $ it holds that
\begin{align}
\begin{split}
 &x(t_2)
  -
  x(t_1)
\\&=
  \int_{t_1}^{t_2}
  e^{(t_2-s)A}
  \left[
    f(s)
    -
    f(t_2)
  \right] ds
  +
  A^{-1}
  \left(
    e^{(t_2-t_1)A}
    -
    \Id_W
  \right)
  \left[ 
    f(t_2)
    -
    f(t)
  \right]
\\&+
  A^{-1}
  e^{(t_1-\varepsilon)A}
  \left(
    e^{(t_2-t_1)A}
    -
    \Id_W
  \right)
  f(t)
  +
  A^{-1}
  \left(
    e^{(t_1-\varepsilon)A}
    -
    \Id_W
  \right)
  \left(
    e^{(t_2-t_1)A}
    -
    \Id_W
  \right)
  \left[ 
    f(t_1)
    -
    f(t)
  \right]
\\&+
  \left(
    e^{(t_2-t_1)A}
    -
    \Id_W
  \right)
  e^{(t_1-\varepsilon)A}
  \int_0^{\varepsilon}
  e^{(\varepsilon-s)A}
  f(s) \, ds
  +
  \left(
    e^{(t_2-t_1)A}
    -
    \Id_W
  \right)
  \int_{\varepsilon}^{t_1}
  e^{(t_1-s)A}
  \left[
    f(s)
    -
    f(t_1)
  \right] ds .
\end{split}
\end{align}
This implies that
for all $ t \in (0,T] $, $ \varepsilon \in (0,t) $,
$ t_1, t_2 \in (\varepsilon,T] $
with $ t_1 \leq t \leq t_2 $ it holds that
\begin{align}
\begin{split}
 &\left\|
    A
    \left(
      x(t_2)
      -
      x(t_1)
    \right)
  \right\|_W
\\&\leq
  \int_{t_1}^{t_2}
  \left\|
    A
    e^{(t_2-s)A}
  \right\|_{ L(W) }
  \left\|
    f(s)
    -
    f(t_2)
  \right\|_W ds
  +
  \left\|
    e^{(t_2-t_1)A}
    -
    \Id_W
  \right\|_{ L(W) }
  \left\|
    f(t_2)
    -
    f(t)
  \right\|_W
\\&\quad+
  \left\|
    e^{(t_1-\varepsilon) A}
  \right\|_{ L(W) }
  \left\|
    \left(
      e^{(t_2-t_1)A}
      -
      \Id_W
    \right)
    f(t)
  \right\|_W
\\&\quad+
  \big\|
    e^{(t_1-\varepsilon) A}
    -
    \Id_W
  \big\|_{ L(W) }
  \big\|
    e^{(t_2-t_1)A}
    -
    \Id_W
  \big\|_{ L(W) }
  \big\| 
    f(t_1)
    -
    f(t)
  \big\|_W
\\&\quad+
  \big\|
    \left( -A \right)^{-\frac{1}{2}}
    \big(
      e^{(t_2-t_1)A}
      -
      \Id_W
    \big)
  \big\|_{ L(W) }
  \big\|
    A
    e^{(t_1-\varepsilon)A}
  \big\|_{ L(W) }
  \int_0^{\varepsilon}
  \big\|
    \left( -A \right)^{\frac{1}{2}}
    e^{(\varepsilon-s)A}
  \big\|_{ L(W) }
  \big\| 
    f(s)
  \big\|_W \, ds
\\&\quad+
  \big\|
    \left( -A \right)^{-\frac{\alpha}{2}}
    \big(
      e^{(t_2-t_1)A}
      -
      \Id_W
    \big)
  \big\|_{ L(W) }
  \int_{\varepsilon}^{t_1}
  \big\|
    \left( -A \right)^{(1+\frac{\alpha}{2})}
    e^{(t_1-s)A}
  \big\|_{ L(W) }
  \big\| 
    f(s)
    -
    f(t_1)
  \big\|_W \, ds .
\end{split}
\end{align}
The fact that $ K \geq 1 $
hence assures that
for all $ t \in (0,T] $, $ \varepsilon \in (0,t) $,
$ t_1, t_2 \in (\varepsilon,T] $
with $ t_1 \leq t \leq t_2 $ it holds that
\begin{align}
\begin{split}
 &\left\|
    A
    \left(
      x(t_2)
      -
      x(t_1)
    \right)
  \right\|_W
\\&\leq
  \frac{ K^2 }{ t_1 }
  \int_{t_1}^{t_2}
  \left| t_2 - s \right|^{(\alpha-1)} ds
  +
  \frac{ K^2 }{ t }
  \left| t_2 - t \right|^{\alpha}
  +
  K
  \left\|
    \left(
      e^{(t_2-t_1)A}
      -
      \Id_W
    \right)
    f(t)
  \right\|_W
  +
  \frac{ K^3 }{ t_1 }
  \left| t - t_1 \right|^{\alpha}
\\&\quad+
  \frac{ 
    K^3
    \left| t_2 - t_1 \right|^{\frac{1}{2}} 
  }{ (t_1 - \varepsilon) }
  \left[
    \sup_{ s \in (0,\varepsilon) }
    \big\|
      f(s)
    \big\|_W
  \right]
  \int_0^{\varepsilon}
  \left( \varepsilon - s \right)^{-\frac{1}{2}} ds
\\&\quad+
  K^2
  \left| t_2 - t_1 \right|^{\frac{\alpha}{2}} 
  \int_{\varepsilon}^{t_1}
  \big\|
    \left( -A \right)
    e^{\frac{(t_1-s)}{2}A}
  \big\|_{ L(W) }
  \big\|
    \left( -A \right)^{\frac{\alpha}{2}}
    e^{\frac{(t_1-s)}{2}A}
  \big\|_{ L(W) }
  \frac{\left| t_1 - s \right|^{\alpha}}{s} \, ds
\\&\leq
  \frac{
    K^2
    \left|t_2-t_1\right|^{\alpha}
  }{\alpha \varepsilon }
  +
  \frac{ 
    K^2
    \left| t_2 - t \right|^{\alpha}
  }{ \varepsilon }
  +
  K
  \left\|
    \left(
      e^{(t_2-t_1)A}
      -
      \Id_W
    \right)
    f(t)
  \right\|_W
  +
  \frac{ 
    K^3 
    \left| t - t_1 \right|^{\alpha}
  }{ \varepsilon }
\\&\quad+
  \frac{ 
    2 K^4
    \varepsilon^{\frac{1}{2}}
    \left| t_2 - t_1 \right|^{\frac{1}{2}} 
  }{ (t_1 - \varepsilon) }
  +
  \frac{
    4 K^4
    \left| t_2 - t_1 \right|^{\frac{\alpha}{2}}
  }{ \varepsilon}
  \int_{\varepsilon}^{t_1}
  \left| t_1 - s \right|^{(\frac{\alpha}{2}-1)} ds 
\\&\leq 
  \frac{
    \left( 
      3 K^3
      +
      8
      K^4
      T^{\frac{\alpha}{2}}
    \right)
    \left|t_2-t_1\right|^{\alpha}
  }{\alpha \varepsilon }
  +
  K
  \left\|
    \left(
      e^{(t_2-t_1)A}
      -
      \Id_W
    \right)
    f(t)
  \right\|_W
  +
  \frac{ 
    2 K^4
    \sqrt{T}
    \left| t_2 - t_1 \right|^{\frac{1}{2}} 
  }{ (t_1 - \varepsilon) } .
\end{split}
\end{align}
This ensures that
for all $ t \in (0,T] $, $ \varepsilon \in (0,t) $
it holds that
\begin{align}
\begin{split}
 &\limsup_{
    \substack{ 
      (t_1,t_2) \rightarrow (t,t), \\ 
      (t_1,t_2) \in (\varepsilon,t] \times [t,T] 
    }
  }
  \left\|
    A
    \left(
      x(t_2)
      -
      x(t_1)
    \right)
  \right\|_W
  =
  0 .
\end{split}
\end{align}
Therefore, we obtain that the 
function 
$
  (0,T]
  \ni
  t
  \mapsto
  x(t)
  \in 
  D(A)
$
is continuous.
Moreover, note that
for all $ t \in (0,T] $, $ \varepsilon \in (0,t) $,
$ t_1, t_2 \in (\varepsilon,T] $
with $ t_1 \leq t \leq t_2 $ it holds that
\begin{align}
\begin{split}
 &x(t_2)
  -
  x(t_1)
  -
  A x(t)
  \left(
    t_2 - t_1
  \right)
  -
  f(t)
  \left(
    t_2 - t_1
  \right)
\\&=
  \left(
    e^{(t_2-t_1)A}
    -
    \Id_W
  \right)
  x(t_1)
  -
  A x(t)
  \left(
    t_2 - t_1
  \right)
  +
  \int_{t_1}^{t_2}
  e^{(t_2-s)A}
  f(s) \, ds
  -
  f(t)
  \left(
    t_2
    -
    t_1
  \right)
\\&=
  \left(
    e^{(t_2-t_1)A}
    -
    \Id_W
    -
    A
    \left( 
      t_2 - t_1
    \right)
  \right)
  x(t)
  +
  \left(
    e^{(t_2-t_1)A}
    -
    \Id_W
  \right)
  \left(
    x(t_1)
    -
    x(t)
  \right)
\\&\quad+
  \int_{t_1}^{t_2}
  e^{(t_2-s)A}
  \left[ 
    f(s)
    -
    f(t)
  \right] ds
  +
  \int_{t_1}^{t_2}
  \left(
    e^{(t_2-s)A}
    -
    \Id_W
  \right)
  f(t) \, ds
\\&=
  \int_0^{(t_2-t_1)}
  \left(
    e^{sA}
    -
    \Id_W
  \right) 
  A
  x(t) \, ds
  +
  \int_0^{(t_2-t_1)}
  e^{sA}
  A
  \left(
    x(t_1)
    -
    x(t)
  \right) ds
\\&\quad+
  \int_{t_1}^{t_2}
  e^{(t_2-s)A}
  \left[ 
    f(s)
    -
    f(t)
  \right] ds
  +
  \int_0^{(t_2-t_1)}
  \left(
    e^{sA}
    -
    \Id_W
  \right)
  f(t) \, ds .
\end{split}
\end{align}
This shows that
for all $ t \in (0,T] $, $ \varepsilon \in (0,t) $,
$ (t_1, t_2) \in \big( (\varepsilon,t] \times [t,T] \big) \backslash \{ (t,t) \} $
it holds that
\begin{align}
\begin{split}
 &\frac{
    \left\|
      x(t_2)
      -
      x(t_1)
      -
      A x(t)
      \left(
	t_2 - t_1
      \right)
      -
      f(t)
      \left(
	t_2 - t_1
      \right)
    \right\|_W
  }{
    \left( t_2 - t_1 \right)
  }
\\&\leq
  \sup_{ s \in (0,t_2-t_1) }
  \left\|
    \left(
      e^{sA}
      -
      \Id_W
    \right) 
    A
    x(t)
  \right\|_W
  +
  \sup_{ s \in (0,t_2-t_1) }
  \left\|
    e^{sA}
  \right\|_{ L(W) }
  \left\|
    A
    \left(
      x(t_1)
      -
      x(t)
    \right)
  \right\|_W
\\&\quad+
  \int_{t_1}^{t_2}
  \frac{ 
    \big\|
      e^{(t_2-s)A}
    \big\|_{L(W)}
    \left\| 
      f(s)
      -
      f(t)
    \right\|_W
  }{
    \left( t_2 - t_1 \right)
  } \, ds
  +
  \sup_{ s \in (0,t_2-t_1) }
  \left\|
    \left(
      e^{sA}
      -
      \Id_W
    \right) 
    f(t)
  \right\|_W  
\\&\leq
  \sup_{ s \in (0,t_2-t_1) }
  \left\|
    \left(
      e^{sA}
      -
      \Id_W
    \right) 
    A
    x(t)
  \right\|_W
  +
  K
  \left\|
    A
    \left(
      x(t_1)
      -
      x(t)
    \right)
  \right\|_W
  +
  \frac{ 
    K^2 
    \left( t_2 - t_1 \right)^{\alpha}
  }{ \left(1+\alpha\right) t_1 }
\\&\quad+
  \sup_{ s \in (0,t_2-t_1) }
  \left\|
    \left(
      e^{sA}
      -
      \Id_W
    \right) 
    f(t)
  \right\|_W .
\end{split}
\end{align}
This and~\eqref{it:in_DA_2} ensure
that for all $ t \in (0,T] $, $ \varepsilon \in (0,t) $
it holds that
\begin{align}
\begin{split}
 &\limsup_{
    \substack{ 
      (t_1,t_2) \rightarrow (t,t), \\ 
      (t_1,t_2) \in \left( (\varepsilon,t] \times [t,T] \right) \backslash \{ (t,t) \} 
    }
  }
  \frac{
    \left\|
      x(t_2)
      -
      x(t_1)
      -
      A x(t)
      \left(
	t_2 - t_1
      \right)
      -
      f(t)
      \left(
	t_2 - t_1
      \right)
    \right\|_W
  }{
    \left( t_2 - t_1 \right)
  }
  =
  0 .
\end{split}
\end{align}
Combining this with~\eqref{it:in_DA_2} 
proves~\eqref{it:in_DA_4} and~\eqref{it:in_DA_5}.
Finally, note that \eqref{it:in_DA_4} and the fact that
\begin{align}
\begin{split}
  \limsup_{ s \searrow 0 }
  \left\|
    x(s)
    -
    x(0)
  \right\|_W
&=
  \limsup_{ s \searrow 0 }
  \left\|
    \int_0^s
    e^{(s-u)A}
    f(u) \, du
  \right\|_W
\\&\leq 
  \limsup_{ s \searrow 0 }
  \left[ 
    \int_0^s
    \left\|
      e^{(s-u)A}
    \right\|_{ L(W) }
    \left\|
      f(u) 
    \right\|_W du
  \right] 
  \leq
  K^2
  \left[ 
    \limsup_{ s \searrow 0 }
    s
  \right]
  =
  0
\end{split}
\end{align}
show \eqref{it:in_DA_3}. The proof of 
Lemma~\ref{lem:in_DA_continuous} is thus completed.
\end{proof}

\begin{corollary}
\label{cor:XO_in_DA}
Assume the setting in Section~\ref{sec:pathwise_setting}. Then
\begin{enumerate}[(i)]
\item
it holds for all $ \omega \in \Omega $,
$ t \in [0,T] $ that $ \XO{t}(\omega) \in D(A) $,
\item
it holds for every $ \omega \in \Omega $ that the function
$
  (0,T]
  \ni
  t
  \mapsto
  \XO{t}(\omega)
  \in D(A)
$
is continuous,
\item
it holds for every $ \omega \in \Omega $ that the function
$
  [0,T]
  \ni
  t
  \mapsto
  \XO{t}(\omega)
  \in H
$
is continuous,
\item
it holds for every $ \omega \in \Omega $ that the function
$
  (0,T]
  \ni
  t
  \mapsto
  \XO{t}(\omega)
  \in H
$
is continuously differentiable, and
\item
it holds for all $ \omega \in \Omega $,
$ t \in (0,T] $
that
$
  \frac{d \XO{t}(\omega)}{ d t}
  =
  A \XO{t}(\omega)
  +
  F( X_t(\omega) )
$.
\end{enumerate}
\end{corollary}
\begin{proof}[Proof of Corollary~\ref{cor:XO_in_DA}]
Throughout this proof let
$ \kappa \in (0, 1-\sr) $ be a real number.
Next note that the assumptions that
$ A \colon D(A) \subseteq H \rightarrow H $
is a generator of an analytic semigroup 
with 
$ 
  \textup{spectrum}(A) 
  \subseteq 
  \{
    z \in \mathbb{C} 
    \colon 
    \textup{Re}(z) < 0
  \}
$,
that $ X \colon [0,T] \times \Omega \rightarrow \BaMi $
has continuous sample paths,
that $ F \colon \BaMi \rightarrow H $ is continuous,
that $ O \colon [0,T] \times \Omega \rightarrow \BaMi $
satisfies
$
  \forall \,
  \omega \in \Omega
  \colon
  \limsup_{ r \searrow 0 }
  \sup_{ 0 \leq s < t \leq T }
  \frac{ 
    s \| O_t(\omega) - O_s(\omega) \|_{ \BaMi } 
  }{
    (t-s)^{r}
  }
  <
  \infty
$,
and that
$ 
  \sup_{ t \in (0,T) } 
  t^{\sr} 
  \| e^{tA} \|_{ L(H, \BaMi) } 
  < \infty 
$
imply that there exist
functions $ \theta \colon \Omega \rightarrow (0,1] $,
$ K \colon \Omega \rightarrow [0,\infty) $
such that for all $ \omega \in \Omega $
it holds that
\begin{align}
\begin{split}
  K(\omega)
 &=
  \sup_{t \in (0,T] }
  \Big[
    \big\|
      (-tA)^{\kappa}
      e^{tA}
    \big\|_{ L(H) }
    +
    \big\|
      (-tA)^{-\kappa}
      (
        e^{tA}
        -
        \Id_H
      )
    \big\|_{ L(H) }
    +
    t^{\sr}
    \big\|
      e^{tA}
    \big\|_{ L(H,\BaMi) }
  \Big]
\\&\quad+
  \sup_{ 0 \leq s < t \leq T }
  \left[
    \big\|
      X_t(\omega)
    \big\|_{\BaMi}
    +
    \big\|
      F(X_t(\omega))
    \big\|_H
    +
    \frac{
      s
      \big\| O_{t}(\omega) - O_{s}(\omega) \big\|_{\BaMi}
    }{
      \left| t - s \right|^{\nr(\omega)}
    } 
  \right] .
\end{split}
\end{align}
Moreover, observe that
for all $ \omega \in \Omega $,
$ t_1, t_2 \in (0,T] $
with $ t_1 < t_2 $
it holds that
\begin{align}
\begin{split}
 &\left\|
    X_{t_2}(\omega) - X_{t_1}(\omega)
  \right\|_{\BaMi}
\\&\leq
  \int_{t_1}^{t_2}
  \left\|
    e^{(t_2-s)A}
    F(X_s(\omega))
  \right\|_{\BaMi} ds
  +
  \int_{0}^{t_1}
  \left\|
    \left(
      e^{(t_2-s)A}
      -
      e^{(t_1-s)A}
    \right)
    F(X_s(\omega))
  \right\|_{\BaMi} ds
\\&\quad+
  \left\|
    O_{t_2}(\omega)
    -
    O_{t_1}(\omega)
  \right\|_{\BaMi}
\\&\leq
  \int_{t_1}^{t_2}
  \left\|
    e^{(t_2-s)A}
  \right\|_{L(H,\BaMi)}
  \left\|
    F(X_s(\omega))
  \right\|_H ds
  +
  \frac{ 
    K(\omega)
    \left(
      t_2 - t_1
    \right)^{\nr(\omega)}
  }{ t_1 }
\\&\quad+
  \int_{0}^{t_1}
  \big\|
    e^{\frac{(t_1-s)}{2}A}
  \big\|_{L(H,\BaMi)}
  \big\|
    e^{\frac{(t_1-s)}{2}A}
    \left(
      e^{(t_2-t_1)A}
      -
      \Id_H
    \right) \!
  \big\|_{L(H)}
  \left\|
    F(X_s(\omega))
  \right\|_H ds .
\end{split}
\end{align}
This shows that
for all $ \omega \in \Omega $,
$ t_1, t_2 \in (0,T] $
with $ t_1 < t_2 $
it holds that
\begin{align}
\begin{split}
 &\left\|
    X_{t_2}(\omega) - X_{t_1}(\omega)
  \right\|_{\BaMi}
\\&\leq
  |K(\omega)|^2
  \int_{t_1}^{t_2}
  \left(
    t_2 - s
  \right)^{-\sr} ds
  +
  \frac{ 
    K(\omega)
    \left(
      t_2 - t_1
    \right)^{\nr(\omega)}
  }{ t_1 }
\\&\quad+
  2^{\sr}
  |K(\omega)|^2
  \int_{0}^{t_1}
  \left(
    t_1 - s
  \right)^{-\sr}
  \big\|
    \left(-A\right)^{\kappa}
    e^{\frac{(t_1-s)}{2}A}
  \big\|_{ L(H) }
  \big\|
    \left(-A\right)^{-\kappa}
    \left(
      e^{(t_2-t_1)A}
      -
      \Id_H
    \right) \!
  \big\|_{L(H)} \, ds
\\&\leq
  \frac{ 
    |K(\omega)|^2
    \left( t_2 - t_1 \right)^{ (1-\sr) }
  }{ \left( 1- \sr \right) }
  +
  \frac{ 
    K(\omega)
    \left(
      t_2 - t_1
    \right)^{\nr(\omega)}
  }{ t_1 } 
  +
  2^{(\sr+\kappa)}
  |K(\omega)|^4
  \left( t_2 - t_1 \right)^{\kappa}
  \int_{0}^{t_1}
  \left(
    t_1 - s
  \right)^{-(\sr+\kappa)} ds .
\end{split}
\end{align}
Hence, we obtain that
for all $ \omega \in \Omega $,
$ t_1, t_2 \in (0,T] $
with $ t_1 < t_2 $
it holds that
\begin{align}
\begin{split}
 &\left\|
    X_{t_2}(\omega) - X_{t_1}(\omega)
  \right\|_{\BaMi}
\\&\leq
  \frac{ 
    |K(\omega)|^2
    \left( t_2 - t_1 \right)^{ (1-\sr) }
  }{ \left( 1- \sr \right) }
  +
  \frac{ 
    K(\omega)
    \left(
      t_2 - t_1
    \right)^{\nr(\omega)}
  }{ t_1 }
  +
  \frac{ 
    2^{(\sr+\kappa)}
    |K(\omega)|^4 \,
    |t_1|^{(1-\sr-\kappa)}
    \left( t_2 - t_1 \right)^{\kappa}
  }{
    \left( 1 - \sr - \kappa \right) 
  }
\\&=
  \frac{ 
    \left( t_2 - t_1 \right)^{\min\{\kappa, \nr(\omega)\}}
  }{
    t_1
  }
  \Bigg(
    \frac{ 
      |K(\omega)|^2 \,
      t_1
      \left( t_2 - t_1 \right)^{ (1-\sr-\min\{\kappa,\nr(\omega)\}) }
    }{ \left( 1- \sr \right) }
    +
    K(\omega)
    \left(
      t_2 - t_1
    \right)^{\max\{0, \nr(\omega)-\kappa\}}
\\&\quad+
    \frac{ 
      2^{(\sr+\kappa)}
      |K(\omega)|^4 \,
      |t_1|^{(2-\sr-\kappa)}
      \left(
	t_2 - t_1
      \right)^{\max\{0,\kappa-\nr(\omega)\}}
    }{
      \left( 1 - \sr - \kappa \right) 
    }
  \Bigg) 
\\&\leq
  \frac{ 
    \left( t_2 - t_1 \right)^{\min\{\kappa,\nr(\omega)\}}
  }{
    t_1
  }
  \left(
    \frac{ 
      |K(\omega)|^2 \,
      (1+T)^2
    }{ \left( 1- \sr \right) }
    +
    K(\omega)
    (1+T)
    +
    \frac{ 
      2^{(\sr+\kappa)}
      |K(\omega)|^4 \,
      (1+T)^2
    }{
      \left( 1 - \sr - \kappa \right) 
    }
  \right) .
\end{split}
\end{align}
Combining this with the fact that
$ 
  \forall \,
  u, v \in \BaMi
  \colon
  \| F(u) - F(v) \|_H^2
  \leq 
  L
  \| u - v \|_{\BaMi}^2
  (
    1
    +
    \| u \|_{ \BaMi }^{\qs}
    +
    \| v \|_{ \BaMi }^{\qs}
  )
$
yields that
for all $ \omega \in \Omega $,
$ t_1, t_2 \in (0,T] $
with $ t_1 < t_2 $
it holds that
\begin{align}
\begin{split}
 &\left\|
    F(X_{t_2}(\omega))
    -
    F(X_{t_1}(\omega))
  \right\|_H
\\&\leq
  \sqrt{
    L
    \left(
      1
      +
      \left\| X_{t_2}(\omega) \right\|_{\BaMi}^{\qs}
      +
      \left\| X_{t_1}(\omega) \right\|_{\BaMi}^{\qs}
    \right)
  }
  \left\|
    X_{t_2}(\omega) - X_{t_1}(\omega)
  \right\|_{\BaMi}
\\&\leq
  \frac{ 
    \left( t_2 - t_1 \right)^{\min\{\kappa,\nr(\omega)\}}
  }{
    t_1
  }
  \sqrt{
    L
    \left(
      1
      +
      2
      |K(\omega)|^{\qs}
    \right)
  }
  (1+T)^2
  \left(
    \frac{ 
      |K(\omega)|^2
    }{ \left( 1- \sr \right) }
    +
    K(\omega)
    +
    \frac{ 
      2^{(\sr+\kappa)}
      |K(\omega)|^4 \,
    }{
      \left( 1 - \sr - \kappa \right) 
    }
  \right) .
\end{split}
\end{align}
This and Lemma~\ref{lem:in_DA_continuous}
complete the proof of Corollary~\ref{cor:XO_in_DA}.
\end{proof}

\subsection{Regularity of the semilinear integrated version
of the numerical approximation}
\label{sec:semilinear_regularity}

\begin{lemma}
\label{lem:YX_in_DA}
Assume the setting in Section~\ref{sec:pathwise_setting}. Then
\begin{enumerate}[(i)]
\item
it holds for all $ \omega \in \Omega $,
$ t \in [0,T] $
that
$ \YX{t}(\omega) \in D(A) $,
\item\label{eq:YX_in_DA_2}
it holds for every $ \omega \in \Omega $ 
that the function
$
  [0,T]
  \ni
  t
  \mapsto
  \YX{t}(\omega)
  \in D(A)
$
is continuous,
\item
it holds for every $ \omega \in \Omega $ 
that the function
$
  [0,T]
  \ni
  t
  \mapsto
  \YX{t}(\omega)
  \in H
$
is Lipschitz continuous,
\item
it holds for every $ \omega \in \Omega $ 
that the function
$
  [0,T] \backslash \{ 0, \frac{T}{M}, \frac{2T}{M}, \ldots, T \}
  \ni
  t
  \mapsto
  \YX{t}(\omega)
  \in H
$
is continuously differentiable,
\item
it holds for all 
$ \omega \in \Omega $,
$ t \in [0,T] \backslash \{ 0, \frac{T}{M}, \frac{2T}{M}, \ldots, T \} $
that
$
  \frac{ d \YX{t}(\omega) }{ d t }
  =
  A \YX{t}(\omega)
  +
  F( \Y{\fl{t}}(\omega) )
$, and
\item
it holds for all
$ \omega \in \Omega $,
$ t \in [0,T] $ that
$
  \YX{t}(\omega)
  =
  \int_0^t
  \big[
    A \YX{s}(\omega)
    +
    F( \Y{\fl{s}}(\omega) )
  \big] \, ds
$.
\end{enumerate}

\end{lemma}
\begin{proof}[Proof of Lemma~\ref{lem:YX_in_DA}]
Throughout this proof assume w.l.o.g.\
that $ H \neq \{ 0 \} $. Next note
that the assumption that
$ A \colon D(A) \subseteq H \rightarrow H $
is a generator of an analytic semigroup 
with 
$ 
  \textup{spectrum}(A) 
  \subseteq 
  \{
    z \in \mathbb{C} 
    \colon 
    \textup{Re}(z) < 0
  \}
$ ensures
that there exists a function
$ K \colon \Omega \rightarrow (0,\infty) $
such that for all $ \omega \in \Omega $
it holds that
\begin{align}
\begin{split}
  K(\omega)
  =
  \sup_{ t \in \{0,\frac{T}{M}, \ldots, T \} }
  \big\|
    F( \Y{t}(\omega) )
  \big\|_H
  +
  \sup_{ t \in (0,T] }
  \sup_{ r \in \{0,1\} }
  \left[
    \big\|
      (-tA)^r
      e^{tA}
    \big\|_{ L(H) }
    +
    \big\|
      (-tA)^{-r}
      (
        e^{tA}
        -
        \Id_H
      )
    \big\|_{ L(H) }
  \right] .
\end{split}
\end{align}
Moreover, observe that
for all $ \omega \in \Omega $,
$ t \in [0,T] $
it holds that
\begin{align}
\begin{split}
 &\YX{t}(\omega)
  =
  \int_{\fl{t}}^t
  e^{(t-s)A}
  F( \Y{\fl{t}}(\omega) ) \, ds
  +
  \sum_{l=0}^{\frac{M}{T}\fl{t}-1}
  \int_{\frac{lT}{M}}^{\frac{(l+1)T}{M}}
  e^{(t-s)A}
  F( \Y{\fl{s}}(\omega) ) \, ds
\\&=
  A^{-1}
  \big(
    e^{(t-\fl{t})A}
    -
    \Id_H
  \big)
  F( \Y{\fl{t}}(\omega) )
  +
  A^{-1}
  \sum_{l=0}^{\frac{M}{T}\fl{t}-1}
  e^{(t-(l+1)\frac{T}{M})A} 
  \big(
    e^{\frac{T}{M} A}
    -
    \Id_H
  \big)
  F( \Y{\frac{lT}{M}}(\omega) ) .
\end{split}
\end{align}
This implies that for all
$ \omega \in \Omega $,
$ t \in [0,T] $ it holds that
$ \YX{t}(\omega) \in D(A) $.
Furthermore, note that
for all $ \omega \in \Omega $,
$ s, t \in [0,T] $
with $ s \leq t $ it holds that
\begin{align}
\label{eq:YX_in_DA_h11}
\begin{split}
 &\YX{t}(\omega)
  -
  \YX{s}(\omega)
\\&=
  \int_{\fl{t}}^{t}
  e^{(t-u)A}
  F( \Y{\fl{t}}(\omega) ) \, du
  -
  \int_{\fl{s}}^{s}
  e^{(s-u)A}
  F( \Y{\fl{s}}(\omega) ) \, du
\\&\quad+
  \int_{\fl{s}}^{\fl{t}}
  e^{(t-u)A}
  F( \Y{\fl{u}}(\omega) ) \, du
  +
  \big(
    e^{(t-s)A}
    -
    \Id_H
  \big)
  \int_0^{\fl{s}}
  e^{(s-u)A}
  F( \Y{\fl{u}}(\omega) ) \, du
\\&=
  A^{-1}
  \big(
    e^{(t-\fl{t})A}
    -
    \Id_H
  \big)
  \big(
    F( \Y{\fl{t}}(\omega) )
    -
    F( \Y{\fl{s}}(\omega) )
  \big)
\\&\quad+
  A^{-1}
  \big(
    e^{(t-\fl{t})A}
    -
    e^{(s-\fl{s})A}
  \big)
  F( \Y{\fl{s}}(\omega) )
\\&\quad+
  A^{-1}
  \sum_{ l=\frac{M}{T}\fl{s} }^{ \frac{M}{T}\fl{t} - 1 }
  e^{(t-(l+1)\frac{T}{M})A}
  \big(
    e^{\frac{T}{M}A}
    -
    \Id_H
  \big)
  F( \Y{\frac{lT}{M}}(\omega) )
\\&\quad+
  A^{-1}
  \big(
    e^{(t-s)A}
    -
    \Id_H
  \big)
  \sum_{ l=0 }^{ \frac{M}{T}\fl{s} - 1 }
  e^{(s-(l+1)\frac{T}{M})A}
  \big(
    e^{\frac{T}{M}A}
    -
    \Id_H
  \big)
  F( \Y{\frac{lT}{M}}(\omega) ) .
\end{split}
\end{align}
This shows that
for all $ \omega \in \Omega $,
$ s, t \in [0,T] $
with $ s \leq t $ it holds that
\begin{align}
\label{eq:YX_in_DA_h1}
\begin{split}
 &\left\|
    A\left(
      \YX{t}(\omega)
      -
      \YX{s}(\omega)
    \right)
  \right\|_H
\\&\leq 
  \big\|
    e^{(t-\fl{t})A}
    -
    \Id_H
  \big\|_{ L(H) }
  \big\|
    F( \Y{\fl{t}}(\omega) )
    -
    F( \Y{\fl{s}}(\omega) )
  \big\|_H
\\&\quad+
  \big\|
    e^{\min\{t-\fl{t},s-\fl{s}\}A}
  \big\|_{ L(H) }
  \big\|
    \big(
      e^{|t-s-(\fl{t}-\fl{s})|A}
      -
      \Id_H
    \big)
    F( \Y{\fl{s}}(\omega) )
  \big\|_H
\\&\quad+
  \sum_{ l=\frac{M}{T}\fl{s} }^{ \frac{M}{T}\fl{t} - 1 }
  \big\|
    e^{(t-(l+1)\frac{T}{M})A}
  \big\|_{ L(H) }
  \big\|
    e^{\frac{T}{M}A}
    -
    \Id_H
  \big\|_{ L(H) }
  \big\|
    F( \Y{\frac{lT}{M}}(\omega) )
  \big\|_H
\\&\quad+
  \sum_{ l=0 }^{ \frac{M}{T}\fl{s} - 1 }
  \big\|
    e^{(s-(l+1)\frac{T}{M})A}
  \big\|_{ L(H) }
  \big\|
    e^{\frac{T}{M}A}
    -
    \Id_H
  \big\|_{ L(H) }
  \big\|
    \big(
      e^{(t-s)A}
      -
      \Id_H
    \big)
    F( \Y{\frac{lT}{M}}(\omega) )
  \big\|_H
\\&\leq
  K(\omega)
  \big\|
    F( \Y{\fl{t}}(\omega) )
    -
    F( \Y{\fl{s}}(\omega) )
  \big\|_H
\\&\quad+
  K(\omega)
  \big\|
    \big(
      e^{|t-s-(\fl{t}-\fl{s})|A}
      -
      \Id_H
    \big)
    F( \Y{\fl{s}}(\omega) )
  \big\|_H
\\&\quad+
  \frac{M |K(\omega)|^3}{T}
  \big(
    \fl{t} - \fl{s}
  \big)
  +
  \sup_{ l \in \{ 0, 1, 2, \ldots, M \} }
  M |K(\omega)|^2
  \big\|
    \big(
      e^{(t-s)A}
      -
      \Id_H
    \big)
    F( \Y{\frac{lT}{M}}(\omega) )
  \big\|_H .
\end{split}
\end{align}
In particular, this
implies that
for all $ \omega \in \Omega $,
$ s \in [0,T] $ it holds that
\begin{align}
\label{eq:YX_in_DA_h2}
\begin{split}
 &\limsup_{\substack{ t \rightarrow s, \\ t \in [s, T]  }}
  \left\|
    A\left(
      \YX{t}(\omega)
      -
      \YX{s}(\omega)
    \right)
  \right\|_H
  =
  0 .
\end{split}
\end{align}
Moreover, \eqref{eq:YX_in_DA_h1}
assures that for all $ \omega \in \Omega $,
$ t \in [0,T] \backslash \{ 0, \frac{T}{M}, \frac{2T}{M}, \ldots, T \} $
it holds that
\begin{align}
\label{eq:YX_in_DA_h3}
\begin{split}
 &\limsup_{\substack{ s \rightarrow t, \\ s \in [0, t]  }}
  \left\|
    A\left(
      \YX{t}(\omega)
      -
      \YX{s}(\omega)
    \right)
  \right\|_H
  =
  0 .
\end{split}
\end{align}
Next note that~\eqref{eq:YX_in_DA_h11}
implies that
for all $ \omega \in \Omega $,
$ t \in \{ \frac{T}{M}, \frac{2T}{M}, \ldots, T \} $,
$ s \in (t-\frac{T}{M}, t] $ it holds that
\begin{align}
\begin{split}
  \YX{t}(\omega)
  -
  \YX{s}(\omega)
 &=
  A^{-1}
  \big(
    \Id_H
    -
    e^{(s-(t-\frac{T}{M}))A}
  \big)
  F( \Y{(t-\frac{T}{M})}(\omega) )
  +
  A^{-1}
  \big(
    e^{\frac{T}{M}A}
    -
    \Id_H
  \big)
  F( \Y{(t-\frac{T}{M})}(\omega) )
\\&\quad+
  A^{-1}
  \big(
    e^{(t-s)A}
    -
    \Id_H
  \big)
  \sum_{ l = 0 }^{ \frac{tM}{T} - 2 }
  e^{ (s-(l+1)\frac{T}{M})A }
  \big(
    e^{\frac{T}{M}A}
    -
    \Id_H
  \big)
  F( \Y{ \frac{lT}{M} }(\omega) ) .
\end{split}
\end{align}
This ensures that for all $ \omega \in \Omega $,
$ t \in \{ \frac{T}{M}, \frac{2T}{M}, \ldots, T \} $,
$ s \in (t-\frac{T}{M}, t] $ it holds that
\begin{align}
\begin{split}
 &\left\|
    A\left(
      \YX{t}(\omega)
      -
      \YX{s}(\omega)
    \right)
  \right\|_H
\\&\leq 
  \big\|
    \big(
      e^{\frac{T}{M}A}
      -
      e^{(s-(t-\frac{T}{M}))A}
    \big)
    F( \Y{(t-\frac{T}{M})}(\omega) )
  \big\|_H
\\&\quad+
  \sum_{ l = 0 }^{ \frac{tM}{T} - 2 }
  \big\|
    e^{ (s-(l+1)\frac{T}{M})A }
  \big\|_{ L(H) }
  \big\|
    \big(
      e^{\frac{T}{M}A}
      -
      \Id_H
    \big)
  \big\|_{ L(H) }
  \big\|
    \big(
      e^{(t-s)A}
      -
      \Id_H
    \big)
    F( \Y{ \frac{lT}{M} }(\omega) )
  \big\|_H
\\&\leq
  \big\|
    e^{(s-(t-\frac{T}{M}))A}
  \big\|_{ L(H) }
  \big\|
    \big(
      e^{(t-s)A}
      -
      \Id_H
    \big)
    F( \Y{(t-\frac{T}{M})}(\omega) )
  \big\|_H
\\&\quad+
  \sup_{ l \in \{ 0, 1, 2, \ldots, M \} }
  | K(\omega) |^2 \,
  \big(
    \tfrac{tM}{T} - 1
  \big)
  \big\|
    \big(
      e^{(t-s)A}
      -
      \Id_H
    \big)
    F( \Y{ \frac{lT}{M} }(\omega) )
  \big\|_H 
\\&\leq
  \sup_{ l \in \{ 0, 1, 2, \ldots, M \} }
  K(\omega)
  \left( 1 + K(\omega) \right)
  M
  \big\|
    \big(
      e^{(t-s)A}
      -
      \Id_H
    \big)
    F( \Y{ \frac{lT}{M} }(\omega) )
  \big\|_H .
\end{split}
\end{align}
Hence, we obtain that for all
$ \omega \in \Omega $,
$ t \in \{ \frac{T}{M}, \frac{2T}{M}, \ldots, T \} $
it holds that
\begin{align}
\label{eq:YX_in_DA_h4}
\begin{split}
 &\limsup_{\substack{ s \rightarrow t, \\ s \in (t-T/M, t] }}
  \left\|
    A\left(
      \YX{t}(\omega)
      -
      \YX{s}(\omega)
    \right)
  \right\|_H
  =
  0 .
\end{split}
\end{align}
Combining~\eqref{eq:YX_in_DA_h2}, \eqref{eq:YX_in_DA_h3},
and~\eqref{eq:YX_in_DA_h4} proves
for every $ \omega \in \Omega $ that the function 
$ 
  [0,T]
  \ni
  t
  \mapsto
  \YX{t}(\omega)
  \in D(A)
$
is continuous.
Moreover, note
that for all
$ \omega \in \Omega $,
$ s, t \in [0,T] $ with
$ s < t $ it holds that
\begin{align}
\begin{split}
 &\YX{t}(\omega)
  -
  \YX{s}(\omega)
\\&=
  \int_{s}^{t}
  e^{(t-u)A}
  F( \Y{\fl{u}}(\omega) ) \, du
  +
  \left(
    e^{(t-s)A}
    -
    \Id_H
  \right)
  \int_{\fl{s}}^{s}
  e^{(s-u)A}
  F( \Y{\fl{s}}(\omega) ) \, du
\\&\quad+
  \left(
    e^{(t-s)A}
    -
    \Id_H
  \right)
  \int_0^{\fl{s}}
  e^{(s-u)A}
  F( \Y{\fl{u}}(\omega) ) \, du .
\\&=
  \int_{s}^{t}
  e^{(t-u)A}
  F( \Y{\fl{u}}(\omega) ) \, du
  +
  A^{-1}
  \left(
    e^{(t-s)A}
    -
    \Id_H
  \right)
  \big(
    e^{(s-\fl{s})A}
    -
    \Id_H
  \big)
  F( \Y{\fl{s}}(\omega) )
\\&\quad+
  A^{-1}
  \left(
    e^{(t-s)A}
    -
    \Id_H
  \right)
  \sum_{ l=0 }^{\frac{M}{T}\fl{s}-1}
  e^{(s-(l+1)\frac{T}{M})A}
  \big(
    e^{\frac{T}{M}A}
    -
    \Id_H
  \big)
  F( \Y{\frac{lT}{M}}(\omega) ) .
\end{split}
\end{align}
This shows that for all
$ \omega \in \Omega $,
$ s, t \in [0,T] $ with
$ s < t $ it holds that
\begin{align}
\begin{split}
 &\left\|
    \YX{t}(\omega)
    -
    \YX{s}(\omega)
  \right\|_H
\\&\leq
  \int_{s}^{t}
  \left\|
    e^{(t-u)A}
  \right\|_{ L(H) }
  \big\|
    F( \Y{\fl{u}}(\omega) )
  \big\|_H \, du
\\&\quad+
  \left\|
    A^{-1}
    \left(
      e^{(t-s)A}
      -
      \Id_H
    \right)
  \right\|_{ L(H) }
  \big\|
    e^{(s-\fl{s})A}
    -
    \Id_H
  \big\|_{ L(H) }
  \big\|
    F( \Y{\fl{s}}(\omega) )
  \big\|_H
\\&\quad+
  \left\|
    A^{-1} \!
    \left(
      e^{(t-s)A}
      -
      \Id_H
    \right)
  \right\|_{ L(H) }
  \sum_{ l=0 }^{\frac{M}{T}\fl{s}-1} \!\!\!
  \big\|
    e^{(s-(l+1)\frac{T}{M})A}
    \big(
      e^{\frac{T}{M}A}
      -
      \Id_H
    \big)
    F( \Y{\frac{lT}{M}}(\omega) )
  \big\|_H 
\\&\leq
  \left( t - s \right)
  | K(\omega) |^2
  \left(
    1
    +
    K(\omega)
  \right)
\\&\quad+
  \left( t - s \right)
  | K(\omega) |
  \sum_{ l=0 }^{\frac{M}{T}\fl{s}-1} \!\!\!
  \big\|
    e^{(s-(l+1)\frac{T}{M})A}
  \big\|_{ L(H) }
  \big\|
      e^{\frac{T}{M}A}
      -
      \Id_H
  \big\|_{ L(H) }
  \big\|
    F( \Y{\frac{lT}{M}}(\omega) )
  \big\|_H
\\&\leq
  \left( t - s \right)
  | K(\omega) |^2
  \left(
    1
    +
    K(\omega)
  \right)
  +
  \left( t - s \right)
  | K(\omega) |^4
  M .
\end{split}
\end{align}
Hence, we obtain
for every $ \omega \in \Omega $
that the function
$
  [0,T]
  \ni
  t
  \mapsto
  \YX{t}(\omega)
  \in H
$
is Lipschitz continuous.
Furthermore, observe that
for all $ \omega \in \Omega $,
$ t_1, t_2, t \in [0,T] $
with $ t_1 \leq t \leq t_2 $ it holds that
\begin{align}
\label{eq:YX_in_DA_h22}
\begin{split}
 &\YX{t_2}(\omega)
  -
  \YX{t_1}(\omega)
  -
  A\YX{t}(\omega)
  \left(t_2-t_1\right)
  -
  F( \Y{\fl{t}}(\omega) )
  \left( t_2 - t_1 \right) 
\\&=
  \left(
    e^{(t_2-t_1)A}
    -
    \Id_H
  \right)
  \YX{t_1}(\omega)
  -
  A\YX{t}(\omega)
  \left(t_2-t_1\right)
  +
  \int_{t_1}^{t_2}
  \left[ 
    e^{(t_2-s)A}
    F( \Y{\fl{s}}(\omega) )
    -
    F( \Y{\fl{t}}(\omega) )
  \right] ds
\\&=
  \left(
    e^{(t_2-t_1)A}
    -
    \Id_H
    -
    A
    \left(t_2-t_1\right)
  \right)
  \YX{t}(\omega)
  +
  \left(
    e^{(t_2-t_1)A}
    -
    \Id_H
  \right)
  \left(
    \YX{t_1}(\omega)
    -
    \YX{t}(\omega)
  \right)
\\&\quad+
  \int_{t_1}^{t_2}
  e^{(t_2-s)A}
  \left[ 
    F( \Y{\fl{s}}(\omega) )
    -
    F( \Y{\fl{t}}(\omega) )
  \right] ds
  +
  \int_{t_1}^{t_2}
  \left(
    e^{(t_2-s)A}
    -
    \Id_H
  \right)
  F( \Y{\fl{t}}(\omega) ) \, ds
\\&=
  \int_0^{(t_2-t_1)}
  \left(
    e^{sA}
    -
    \Id_H
  \right)
  A \YX{t}(\omega) \, ds
  +
  \int_0^{(t_2-t_1)}
  e^{sA}
  A
  \left(
    \YX{t_1}(\omega)
    -
    \YX{t}(\omega)
  \right) ds
\\&\quad+
  \int_{t_1}^{t_2}
  e^{(t_2-s)A}
  \left[ 
    F( \Y{\fl{s}}(\omega) )
    -
    F( \Y{\fl{t}}(\omega) )
  \right] ds
  +
  \int_{0}^{(t_2-t_1)}
  \left(
    e^{sA}
    -
    \Id_H
  \right)
  F( \Y{\fl{t}}(\omega) ) \, ds .
\end{split}
\end{align}
This implies that for all 
$ \omega \in \Omega $, $ t \in [0,T] $,
$ (t_1,t_2) \in \big( [0,t] \times [t,T] \big) \backslash \{ (t,t) \} $
it holds that
\begin{align}
\label{eq:YX_in_DA_h5}
\begin{split}
 &\frac{
    \big\|
      \YX{t_2}(\omega)
      -
      \YX{t_1}(\omega)
      -
      \left(t_2-t_1\right)
      A\YX{t}(\omega)
      -
      \left(t_2-t_1\right)
      F( \Y{\fl{t}}(\omega) )
    \big\|_H
  }{ \left(t_2 - t_1 \right) }
\\&\leq
  \sup_{ s \in (0,t_2-t_1) }
  \big\|
    \left(
      e^{sA}
      -
      \Id_H
    \right)
    A \YX{t}(\omega)
  \big\|_H
  +
  \sup_{ s \in (0,t_2-t_1) }
  \big\|
    e^{sA}
  \big\|_{ L(H) }
  \big\|
    A
    \left(
      \YX{t_1}(\omega)
      -
      \YX{t}(\omega)
    \right)
  \big\|_H
\\&+
  \sup_{ s \in (t_1,t_2) }
  \big\|
    e^{(t_2-s)A}
  \big\|_{ L(H) }
  \big\| 
    F( \Y{\fl{s}}(\omega) )
    -
    F( \Y{\fl{t}}(\omega) )
  \big\|_H
\\&+
  \sup_{ s \in (0,t_2-t_1) }
  \big\|
    \left(
      e^{sA}
      -
      \Id_H
    \right)
    F( \Y{\fl{t}}(\omega) )
  \big\|_H
\\&\leq
  \sup_{ s \in (0,t_2-t_1) }
  \big\|
    \left(
      e^{sA}
      -
      \Id_H
    \right)
    A \YX{t}(\omega)
  \big\|_H
  +
  K(\omega) \,
  \big\|
    A
    \left(
      \YX{t_1}(\omega)
      -
      \YX{t}(\omega)
    \right)
  \big\|_H
\\&+
  \sup_{ s \in (t_1,t_2) }
  K(\omega) \,
  \big\| 
    F( \Y{\fl{s}}(\omega) )
    -
    F( \Y{\fl{t}}(\omega) )
  \big\|_H
  +\!
  \sup_{ s \in (0,t_2-t_1) }
  \big\|
    \left(
      e^{sA}
      -
      \Id_H
    \right)
    F( \Y{\fl{t}}(\omega) )
  \big\|_H .
\end{split}
\end{align}
This
and~\eqref{eq:YX_in_DA_2}
ensure that for all 
$ \omega \in \Omega $, 
$ t \in [0,T] \backslash \{0,\frac{T}{M},\frac{2T}{M}, \ldots, T \} $ 
it holds that
\begin{align}
\label{eq:YX_in_DA_h6}
\begin{split}
 &\limsup_{\substack{ (t_1,t_2) \rightarrow (t,t), \\ (t_1,t_2) \in ([0,t]\times[t,T])\backslash \{(t,t)\}  }}
  \frac{
    \big\|
      \YX{t_2}(\omega)
      -
      \YX{t_1}(\omega)
      -
      \left(t_2-t_1\right)
      A\YX{t}(\omega)
      -
      \left(t_2-t_1\right)
      F( \Y{\fl{t}}(\omega) )
    \big\|_H
  }{ \left(t_2 - t_1 \right) }
  =
  0 .
\end{split}
\end{align}
The proof of Lemma~\ref{lem:YX_in_DA} is thus completed.
\end{proof}

\begin{lemma}[Temporal regularity]
\label{lem:YX_YX_diff}
Assume the setting in Section~\ref{sec:pathwise_setting}
and let $ \srr \in [0,1-\sr) $, $ t_1, t_2 \in [0,T] $ with $ t_1 < t_2 $. Then
\begin{align}
\begin{split}
  \big\|
    \YX{t_2}
    -
    \YX{t_1}
  \big\|_{\BaMi}
 &\leq
  \left[
    \sup_{ s \in (0,T) }
    s^{\sr}
    \big\|
      e^{sA}
    \big\|_{ L( H, \BaMi ) }
  \right] 
  \int_{t_1}^{t_2}
  \left(
    t_2 - s
  \right)^{-\sr}
  \big\|
    F( \Y{\fl{s}} )
  \big\|_{ H } \, ds
\\&\quad+
  2^{(\sr+\srr)}
  (t_2-t_1)^{\srr}
  \left[
    \sup_{ s \in (0,T) }
    s^{\sr}
    \big\|
      e^{sA}
    \big\|_{ L( H, \BaMi ) }
  \right]
  \left[
    \sup_{ s \in (0,T) }
    \big\|
      (-sA)^{\srr}
      e^{sA}
    \big\|_{ L(H) }
  \right]
\\&\quad\cdot
  \left[
    \sup_{ s \in (0,T) }
    \big\|
      (-sA)^{-\srr}
      \big(
	e^{sA}
	-
	\Id_H
      \big)
    \big\|_{ L(H) }
  \right]
  \int_{0}^{t_1}
  \left(
    t_1 - s
  \right)^{-(\sr+\srr)}
  \big\|
    F( \Y{\fl{s}} )
  \big\|_{H} \, ds .
\end{split}
\end{align}
\end{lemma}
\begin{proof}[Proof of Lemma~\ref{lem:YX_YX_diff}]
Note that
\begin{align}
\begin{split}
 &\big\|
    \YX{t_2}
    -
    \YX{t_1}
  \big\|_{\BaMi}
\\&\leq
  \int_{t_1}^{t_2}
  \big\|
    e^{(t_2-s)A}
    F( \Y{\fl{s}} )
  \big\|_{\BaMi} \, ds
  +
  \int_{0}^{t_1}
  \big\|
    \big(
      e^{(t_2-t_1)A}
      -
      \Id_H
    \big) \,
    e^{(t_1-s)A}
    F( \Y{\fl{s}} )
  \big\|_{\BaMi} \, ds
\\&\leq
  \left[
    \sup_{ s \in (0,t_2-t_1) }
    s^{\sr}
    \big\|
      e^{sA}
    \big\|_{ L( H, \BaMi ) }
  \right] 
  \int_{t_1}^{t_2}
  \left(
    t_2 - s
  \right)^{-\sr}
  \big\|
    F( \Y{\fl{s}} )
  \big\|_{ H } \, ds
\\&\quad+
  2^{\sr}
  \left[
    \sup_{ s \in (0,t_1) }
    s^{\sr}
    \big\|
      e^{sA}
    \big\|_{ L( H, \BaMi ) }
  \right] 
  \int_{0}^{t_1}
  \left(
    t_1 - s
  \right)^{-\sr}
  \big\|
    \big(
      e^{(t_2-t_1)A}
      -
      \Id_H
    \big) \,
    e^{\frac{(t_1-s)}{2}A}
    F( \Y{\fl{s}} )
  \big\|_{H} \, ds .
\end{split}
\end{align}
This implies that
\begin{align}
\begin{split}
 &\big\|
    \YX{t_2}
    -
    \YX{t_1}
  \big\|_{\BaMi}
\\&\leq
  \left[
    \sup_{ s \in (0,T) }
    s^{\sr}
    \big\|
      e^{sA}
    \big\|_{ L( H, \BaMi ) }
  \right] 
  \int_{t_1}^{t_2}
  \left(
    t_2 - s
  \right)^{-\sr}
  \big\|
    F( \Y{\fl{s}} )
  \big\|_{ H } \, ds
\\&\quad+
  2^{(\sr+\srr)}
  (t_2-t_1)^{\srr}
  \left[
    \sup_{ s \in (0,T) }
    s^{\sr}
    \big\|
      e^{sA}
    \big\|_{ L( H, \BaMi ) }
  \right] 
  \int_{0}^{t_1}
  \left(
    t_1 - s
  \right)^{-(\sr+\srr)}
\\&\quad\cdot
  \big\|
    (-(t_2-t_1)A)^{-\srr}
    \big(
      e^{(t_2-t_1)A}
      -
      \Id_H
    \big)
  \big\|_{ L(H) }
  \big\|
    (-\tfrac{(t_1-s)}{2}A)^{\srr}
    e^{\frac{(t_1-s)}{2}A}
  \big\|_{ L(H) }
  \big\|
    F( \Y{\fl{s}} )
  \big\|_{H} \, ds .
\end{split}
\end{align}
The proof of Lemma~\ref{lem:YX_YX_diff} 
is thus completed.
\end{proof}

\subsection{Analysis of the error between the 
numerical approximation and its semilinear
integrated version}
\label{sec:pathwise_error_ana}

\begin{lemma}
\label{lem:YX_YO_diff}
Assume the setting in Section~\ref{sec:pathwise_setting} 
and let $ \cc \in (0,\infty) $, $ t \in (0,T] $. Then
\begin{align}
\begin{split}
  \big\|
    \YX{t} - \YO{t}
  \big\|_{\BaMi}
 &\leq
  \frac{T^{\cc}}{M^{\cc}}
  \left[ 
    \sup_{s\in(0,T)}
    s^{\sr}
    \big\|
      e^{ sA }
    \big\|_{ L( H, \BaMi ) }
  \right]
  \int_0^t
  \left( t-s \right)^{-\sr }
  \big|
    \Vm( \Y{\fl{s}} )
  \big|^{\cc}
  \big\|
    F( \Y{\fl{s}} )
  \big\|_{ H } \, ds
\\&\quad+
  \int_0^t
  \big\|
    e^{ (t-s)A }
    -
    \S_{\fl{s},t}
  \big\|_{ L( H, \BaMi ) }
  \big\|
    F( \Y{\fl{s}} )
  \big\|_{ H } \, ds .
\end{split}
\end{align}
\end{lemma}
\begin{proof}[Proof of Lemma~\ref{lem:YX_YO_diff}]
Note that the triangle inequality implies
\begin{align}
\begin{split}
  \big\|
    \YX{t} - \YO{t}
  \big\|_{ \BaMi }
 &=
  \left\|
    \int_0^t
    \left(
      e^{ (t-s)A }
      F( \Y{\fl{s}} )
      -
      \S_{\fl{s},t} \,
      \one_{ 
        \{
          \Vm( \Y{\fl{s}} ) 
          \leq
          M/T
        \}
      } \,
      F( \Y{\fl{s}} )
    \right) ds
  \right\|_{ \BaMi }
\\&\leq
  \int_0^t
  \big\|
    e^{ (t-s)A } \,
    \one_{ 
      \{
        \Vm( \Y{\fl{s}} ) 
        >
        M/T
      \}
    } \,
    F( \Y{\fl{s}} )
  \big\|_{\BaMi} \, ds
\\&\quad+
  \int_0^t
  \big\|
    \big(
      e^{ (t-s)A }
      -
      \S_{\fl{s},t}
    \big) \,
    \one_{ 
      \{
        \Vm( \Y{\fl{s}} ) 
        \leq
        M/T
      \}
    } \,
    F( \Y{\fl{s}} )
  \big\|_{\BaMi} \, ds .
\end{split}
\end{align}
Therefore, we obtain
\begin{align}
\begin{split}
 &\big\|
    \YX{t} - \YO{t}
  \big\|_{\BaMi}
\\&\leq
  \int_0^t
  \one_{ 
    \{
      \Vm( \Y{\fl{s}} ) 
      >
      M/T
    \}
  }
  \big\|
    e^{ (t-s)A }
  \big\|_{ L( H, \BaMi ) }
  \big\|
    F( \Y{\fl{s}} )
  \big\|_{ H } \, ds
\\&\quad+
  \int_0^t
  \big\|
    e^{ (t-s)A }
    -
    \S_{\fl{s},t}
  \big\|_{ L( H, \BaMi ) }
  \big\|
    F( \Y{\fl{s}} )
  \big\|_{ H } \, ds
\\&\leq
  \left[ 
    \sup_{s\in(0,T)}
    s^{\sr}
    \big\|
      e^{ sA }
    \big\|_{ L( H, \BaMi ) }
  \right]
  \int_0^t
  \left( t-s \right)^{-\sr }
  \one_{ 
    \{
      |\Vm( \Y{\fl{s}} )|^{\cc} 
      >
      |M/T|^{\cc}
    \}
  }
  \frac{T^{\cc}}{M^{\cc}}
  \big|
    \Vm( \Y{\fl{s}} )
  \big|^{\cc}
  \big\|
    F( \Y{\fl{s}} )
  \big\|_{ H } \, ds
\\&\quad+
  \int_0^t
  \big\|
    e^{ (t-s)A }
    -
    \S_{\fl{s},t}
  \big\|_{ L( H, \BaMi ) }
  \big\|
    F( \Y{\fl{s}} )
  \big\|_{ H } \, ds .
\end{split}
\end{align}
The proof of Lemma~\ref{lem:YX_YO_diff} is thus completed.
\end{proof}

\subsection{Analysis of the error between the 
exact solution and the semilinear
integrated version of the numerical approximation}

\begin{lemma}[Error estimates based on the monotonicity assumption]
\label{lem:XO_YX_diff}
Assume the setting in Section~\ref{sec:pathwise_setting}
and let $ \kappa \in (2,\infty) $,
$ t \in [0,T] $. Then
\begin{align}
\begin{split}
 &\big\|
    \XO{t} - \YX{t}
  \big\|_H^2
  \leq
  \frac{1}{C\left(\kappa - 2\right)}
  \int_{0}^t
  e^{ \kappa C (t-s) }
  \big\|
    F( \YX{s} + O_s )
    -
    F( \Y{\fl{s}} )
  \big\|_{H}^2 \, ds
\\&\leq
  \frac{
    2^{[\qs-1]^+} L
  }{
    C \left( \kappa - 2 \right) 
  }
  \int_0^t
  e^{ \kappa C(t-s) }
  \left(
    1
    +
    \big\| \YX{s} \big\|_{\BaMi}^{\qs}
    + 
    \big\| O_s \big\|_{\BaMi}^{\qs}
    +
    \big\| \Y{\fl{s}} \big\|_{\BaMi}^{\qs}
  \right)
\\&\cdot
  \Big(
    \big\|
      \YX{s}
      -
      \YX{\fl{s}}
    \big\|_{\BaMi}
    +
    \big\|
      \YX{\fl{s}}
      -
      \YO{\fl{s}}
    \big\|_{\BaMi}
    +
    \big\|
      O_s
      -
      O_{\fl{s}}
    \big\|_{\BaMi}
    +
    \big\|
      O_{\fl{s}}
      -
      \O{\fl{s}}
    \big\|_{\BaMi}
  \Big)^2 \, ds .
\end{split}
\end{align}
\end{lemma}
\begin{proof}[Proof of Lemma~\ref{lem:XO_YX_diff}]
W.l.o.g.\ we assume that $ t \in (0,T] $
(otherwise the proof is clear).
Note that the fundamental
theorem of calculus,
Corollary~\ref{cor:XO_in_DA},
and Lemma~\ref{lem:YX_in_DA} imply
that for all 
$ \varepsilon \in (0,t) $
it holds that
\begin{align}
\begin{split}
 &e^{ -\kappa C t }
  \big\|
    \XO{t} - \YX{t}
  \big\|_H^2
  -
  e^{ -\kappa C \varepsilon }
  \big\|
    \XO{\varepsilon} - \YX{\varepsilon}
  \big\|_H^2
\\&=
  2
  \int_{\varepsilon}^t
  e^{ -\kappa C s }
  \big<
    \XO{s} - \YX{s},
    A\XO{s} 
    +
    F( X_s )
    -
    A\YX{s}
    -
    F( \Y{\fl{s}} )
  \big>_{H} \, ds
\\&\quad-
  \kappa C
  \int_{\varepsilon}^t
  e^{ -\kappa C s }
  \big\|
    \XO{s} - \YX{s}
  \big\|_H^2 \, ds
\\&=
  2
  \int_{\varepsilon}^t
  e^{ -\kappa C s }
  \left<
    \XO{s} - \YX{s},
    A\XO{s} 
    +
    F( X_s )
    -
    A\YX{s}
    -
    F( \YX{s} + O_s )
  \right>_H ds
\\&\quad+
  2
  \int_{\varepsilon}^t
  e^{ - \kappa Cs }
  \big<
    \XO{s} - \YX{s},
    F( \YX{s} + O_s )
    -
    F( \Y{\fl{s}} )
  \big>_{H} \, ds
  -
  \kappa C
  \int_{\varepsilon}^t
  e^{ - \kappa C s }
  \big\|
    \XO{s} - \YX{s}
  \big\|_H^2 \, ds .
\end{split}
\end{align}
The assumption that
$
  \forall \,
  x, y \in \BaMi
$
with
$
  x-y \in H_1
  \colon
  \left<
    x - y,
    A(x-y) + F(x)
    - F(y)
  \right>_H
  \leq
  C
  \left\| x - y \right\|_H^2
$
and the Cauchy-Schwarz inequality 
hence prove
that for all 
$ \varepsilon \in (0,t) $
it holds that
\begin{align}
\begin{split}
 &e^{ -\kappa C t }
  \big\|
    \XO{t} - \YX{t}
  \big\|_H^2
  -
  e^{ -\kappa C \varepsilon }
  \big\|
    \XO{\varepsilon} - \YX{\varepsilon}
  \big\|_H^2
\\&=
  2
  \int_{\varepsilon}^t
  e^{ -\kappa C s }
  \big<
    X_s - (\YX{s} + O_s),
    A (X_s - [ \YX{s} + O_s ]) 
    +
    F( X_s )
    -
    F( \YX{s} + O_s )
  \big>_H \, ds
\\&\quad+
  2
  \int_{\varepsilon}^t
  e^{ -\kappa C s }
  \big<
    \XO{s} - \YX{s},
    F( \YX{s} + O_s )
    -
    F( \Y{\fl{s}} )
  \big>_{H} \, ds
  -
  \kappa C
  \int_{\varepsilon}^t
  e^{ -\kappa C s }
  \big\|
    \XO{s} - \YX{s}
  \big\|_H^2 \, ds
\\&\leq
  2 C
  \int_{\varepsilon}^t
  e^{ - \kappa C s }
  \big\|
    X_s - (\YX{s} + O_s)
  \big\|_H^2 \, ds
  -
  \kappa C
  \int_{\varepsilon}^t
  e^{ -\kappa C s }
  \big\|
    \XO{s} - \YX{s}
  \big\|_H^2 \, ds
\\&\quad+
  2
  \int_{\varepsilon}^t
  e^{ -\kappa C s }
  \big\|
    \XO{s} - \YX{s}
  \big\|_{H}
  \big\|
    F( \YX{s} + O_s )
    -
    F( \Y{\fl{s}} )
  \big\|_{H} \, ds .
\end{split}
\end{align}
The fact that 
$ 
  \forall \, x,y \in \mathbb{R} 
  \colon
  xy \leq \frac{x^2}{2} + \frac{y^2}{2} 
$ 
therefore shows
that for all 
$ \varepsilon \in (0,t) $
it holds that
\begin{align}
\begin{split}
 &e^{ -\kappa C t }
  \big\|
    \XO{t} - \YX{t}
  \big\|_H^2
  -
  e^{ -\kappa C \varepsilon }
  \big\|
    \XO{\varepsilon} - \YX{\varepsilon}
  \big\|_H^2
\\&\leq
  C
  \left(2 - \kappa\right)
  \int_{\varepsilon}^t
  e^{ -\kappa C s }
  \big\|
    \XO{s} - \YX{s}
  \big\|_{H}^2 \, ds
\\&\quad+
  2
  \int_{\varepsilon}^t
  e^{ -\kappa C s }
  \left(
    \sqrt{C(\kappa - 2)}
    \big\|
      \XO{s} - \YX{s}
    \big\|_{H}
  \right)
  \left(
    \frac{1}{\sqrt{C(\kappa - 2)}}
    \big\|
      F( \YX{s} + O_s )
      -
      F( \Y{\fl{s}} )
    \big\|_{H} 
  \right) ds
\\&\leq
  C
  \left(2 - \kappa\right)
  \int_{\varepsilon}^t
  e^{ -\kappa C s }
  \big\|
    \XO{s} - \YX{s}
  \big\|_{H}^2 \, ds
\\&\quad+
  \int_{\varepsilon}^t
  e^{ -\kappa C s }
  \left[ 
    C
    \left(\kappa - 2\right)
    \big\|
      \XO{s} - \YX{s}
    \big\|_{H}^2
    +
    \frac{1}{C\left(\kappa - 2\right)}
    \big\|
      F( \YX{s} + O_s )
      -
      F( \Y{\fl{s}} )
    \big\|_{H}^2 
  \right] ds
\\&= 
  \frac{1}{C\left(\kappa - 2\right)}
  \int_{\varepsilon}^t
  e^{ -\kappa C s }
  \big\|
    F( \YX{s} + O_s )
    -
    F( \Y{\fl{s}} )
  \big\|_{H}^2 \, ds .
\end{split}
\end{align}
Corollary~\ref{cor:XO_in_DA}
and Lemma~\ref{lem:YX_in_DA} hence ensure
\begin{align}
\label{eq:XO_YX_diff_1}
\begin{split}
  \big\|
    \XO{t} - \YX{t}
  \big\|_H^2
 &=
  \big\|
    \XO{t} - \YX{t}
  \big\|_H^2
  -
  \lim_{
    \substack{
      \varepsilon \rightarrow 0 \\
      \varepsilon \in (0,t)
    }
  }
  e^{ \kappa C (t-\varepsilon) }
  \big\|
    \XO{\varepsilon} - \YX{\varepsilon}
  \big\|_H^2
\\&\leq
  \frac{1}{C\left(\kappa - 2\right)}
  \lim_{
    \substack{
      \varepsilon \rightarrow 0 \\
      \varepsilon \in (0,t)
    }
  }
  \int_{\varepsilon}^t
  e^{ \kappa C (t-s) }
  \big\|
    F( \YX{s} + O_s )
    -
    F( \Y{\fl{s}} )
  \big\|_{H}^2 \, ds
\\&=
  \frac{1}{C\left(\kappa - 2\right)}
  \int_{0}^t
  e^{ \kappa C (t-s) }
  \big\|
    F( \YX{s} + O_s )
    -
    F( \Y{\fl{s}} )
  \big\|_{H}^2 \, ds .
\end{split}
\end{align}
Next observe that the triangle inequality
implies that for all $ s \in (0,t) $
it holds that
\begin{align}
\label{eq:XO_YX_diff_2}
\begin{split}
 &\big\|
    \YX{s} + O_s
    -
    \Y{\fl{s}}
  \big\|_{\BaMi}
\\&\leq
  \big\|
    \YX{s}
    -
    \YX{\fl{s}}
  \big\|_{\BaMi}
  +
  \big\|
    \YX{\fl{s}} + O_s
    -
    \Y{\fl{s}}
  \big\|_{\BaMi}
\\&\leq
  \big\|
    \YX{s}
    -
    \YX{\fl{s}}
  \big\|_{\BaMi}
  +
  \big\|
    \YX{\fl{s}}
    -
    ( \Y{\fl{s}} - \O{\fl{s}} )
  \big\|_{\BaMi}
  +
  \big\|
    O_s
    -
    \O{\fl{s}}
  \big\|_{\BaMi}
\\&\leq
  \big\|
    \YX{s}
    -
    \YX{\fl{s}}
  \big\|_{\BaMi}
  +
  \big\|
    \YX{\fl{s}}
    -
    \YO{\fl{s}}
  \big\|_{\BaMi}
  +
  \big\|
    O_s
    -
    O_{\fl{s}}
  \big\|_{\BaMi}
  +
  \big\|
    O_{\fl{s}}
    -
    \O{\fl{s}}
  \big\|_{\BaMi} .
\end{split}
\end{align}
Combining the fact that
$ 
  \forall \,
  u, v \in \BaMi
  \colon
  \left\|
    F(u) - F(v)
  \right\|_{H}^2
  \leq
  L
  \left\|
    u - v
  \right\|_{\BaMi}^2
  \left(
    1
    +
    \left\| u \right\|_{\BaMi}^{\qs}
    +
    \left\| v \right\|_{\BaMi}^{\qs}
  \right)
$
and the fact that
$
  \forall \,
  r \in (0,\infty),
  x_1, x_2 \in \R
  \colon
  \left| x_1 + x_2 \right|^r
  \leq
  2^{[r-1]^+}
  \left(
    \left| x_1 \right|^r
    +
    \left| x_2 \right|^r
  \right)
$
hence implies that for all 
$ s \in (0,t) $ it holds that
\begin{align}
\begin{split}
 &\big\|
    F( \YX{s} + O_s )
    -
    F( \Y{\fl{s}} )
  \big\|_{H}^2
\\&\leq 
  L
  \big\|
    \YX{s} + O_s
    -
    \Y{\fl{s}}
  \big\|_{\BaMi}^2
  \left(
    1
    +
    \big\| \YX{s} + O_s \big\|_{\BaMi}^{\qs}
    +
    \big\| \Y{\fl{s}} \big\|_{\BaMi}^{\qs}
  \right)
\\&\leq
  2^{[\qs-1]^+}
  L
  \left(
    1
    +
    \big\| \YX{s} \big\|_{\BaMi}^{\qs}
    + 
    \big\| O_s \big\|_{\BaMi}^{\qs}
    +
    \big\| \Y{\fl{s}} \big\|_{\BaMi}^{\qs}
  \right)
\\&\quad\cdot
  \Big(
    \big\|
      \YX{s}
      -
      \YX{\fl{s}}
    \big\|_{\BaMi}
    +
    \big\|
      \YX{\fl{s}}
      -
      \YO{\fl{s}}
    \big\|_{\BaMi}
    +
    \big\|
      O_s
      -
      O_{\fl{s}}
    \big\|_{\BaMi}
    +
    \big\|
      O_{\fl{s}}
      -
      \O{\fl{s}}
    \big\|_{\BaMi}
  \Big)^2 .
\end{split}
\end{align}
This and \eqref{eq:XO_YX_diff_1}
complete the proof of Lemma~\ref{lem:XO_YX_diff}.
\end{proof}

\section{Strong error estimates}
\label{sec:strong_error}
\subsection{Setting}
\label{sec:strong_setting}
Let $ T, C, \qp \in (0,\infty) $,
$ L \in [0,\infty) $,
$ \sp \in [0,1) $,
let $ ( H, \left< \cdot, \cdot \right>_H, \left\| \cdot \right\|_H ) $
be a separable $ \R $-Hilbert space,
let $ A \colon D(A) \subseteq H \rightarrow H $
be a generator of an analytic semigroup 
with
$ 
  \textup{spectrum}(A) 
  \subseteq 
  \{
    z \in \mathbb{C} 
    \colon 
    \textup{Re}(z) < 0
  \}
$,
let $ ( H_r, \left< \cdot, \cdot \right>_{ H_r }, \left\| \cdot \right\|_{ H_r } ) $, 
$ r \in \R $, be a family of interpolation spaces associated to $ -A $
(cf., e.g., Definition~3.5.26 in~\cite{Jentzen2015SPDElecturenotes}),
let $ (\BaMi, \left\| \cdot \right\|_{\BaMi}) $
be a separable $ \R $-Banach space with 
$ H_{1} \subseteq \BaMi \subseteq H $ 
densely and continuously,
let $ \Vm \in \M\big( \B(\BaMi), \B( [0,\infty) ) \big) $,
$ (\SM)_{M\in\N} \subseteq \M\big( \B(\{ (s, t) \in [0,T]^2 \colon s < t \}), \B(L(H, \BaMi) ) \big) $,
$ F \in \C( \BaMi, H ) $
satisfy for all $ x, y \in \BaMi $ with $ x - y \in H_1 $ that
$
  \left\|
    F(x) - F(y)
  \right\|_{H}^2
  \leq
  L
  \left\|
    x - y
  \right\|_V^2
  \left(
    1
    +
    \left\| x \right\|_{\BaMi}^{\qp}
    +
    \left\| y \right\|_{\BaMi}^{\qp}
  \right)
$
and
$
  \left<
    x - y,
    A(x-y) + F(x)
    - F(y)
  \right>_H
  \leq
  C
  \left\| 
    x - y 
  \right\|_H^2
$,
assume that
$
  \sup_{ t \in (0,T) }
  t^{\sp}
  \|
    e^{tA}
  \|_{ L(H,\BaMi) }
  < \infty
$,
let $ (\Omega, \F, \P ) $ 
be a probability space,
let $ X \colon [0,T] \times \Omega \rightarrow \BaMi $
be a stochastic process with continuous sample paths,
let $ O, \XO{} \colon [0,T] \times \Omega \rightarrow \BaMi $ be stochastic processes
which satisfy for all $ \omega \in \Omega $ that
$
  \limsup_{ r \searrow 0 }
  \sup_{ 0 \leq s < t \leq T }
  \frac{ 
    s \| O_t(\omega) - O_s(\omega) \|_{ \BaMi } 
  }{
    (t-s)^{r}
  }
  <
  \infty
$,
let $ \OM{}, \YM{}, \YOM{}, \YXM{} \colon [0,T] \times \Omega \rightarrow \BaMi $,
$ M \in \N $,
be stochastic processes,
and assume for all 
$ t \in [0,T] $, 
$ M \in \mathbb{N} $ that
$
  X_t
  =
  \int_0^t
  e^{(t-s)A}
  F(X_s) \, ds
  +
  O_t
$,
$ \XO{t} = X_t - O_t $, 
$
  \YXM{t}
  =
  \int_0^t
  e^{(t-s)A}
  F( \YM{\fl{s}} ) \, ds
$,
$ \YOM{t} = \YM{t} - \OM{t} $,
and
$
  \YM{t}
  =
  \int_0^t
  \SM_{\fl{s},t} \,
  \one_{ 
    \{
      \Vm( \YM{\fl{s}} )
      \leq
      M/T
    \}
  } \,
  F( \YM{\fl{s}} ) \, ds
  +
  \OM{t}
$.

\subsection{Analysis of the error between the
numerical approximation and its semilinear
integrated version}
\label{sec:strong_error_ana}

\begin{lemma}
\label{lem:YX_YO_diff_Lp}
Assume the setting in Section~\ref{sec:strong_setting} 
and let $ \cp \in (0,\infty) $, 
$ p \in [1, \infty) $, $ t \in (0,T] $, $ M \in \N $. 
Then
\begin{align}
\begin{split}
 &\big\|
    \YXM{t} - \YOM{t}
  \big\|_{\L^p(\P;\BaMi)}
  \leq
  \left( 
    \int_0^t
    \big\|
      e^{ (t-s)A }
      -
      \SM_{\fl{s},t}
    \big\|_{ L( H, \BaMi ) } \, ds
  \right)
  \left[
    \sup_{ s \in [0,T) }
    \big\|
      F( \YM{s} )
    \big\|_{\L^p(\P;H)}
  \right]
\\&\quad+
  \frac{T^{(1+\cp-\sp)}}{\left(1-\sp\right) M^{\cp}}
  \left[ 
    \sup_{s\in(0,T)}
    s^{\sp}
    \big\|
      e^{ sA }
    \big\|_{ L( H, \BaMi ) }
  \right]
  \left[
    \sup_{s\in[0,T)}
    \big\|
      \Vm( \YM{s} )
    \big\|_{ \L^{2p\cp}( \P; \R ) }^{\cp}
    \big\|
      F( \YM{s} )
    \big\|_{ \L^{2p}( \P; H ) }
  \right] .
\end{split}
\end{align}
\end{lemma}
\begin{proof}[Proof of Lemma~\ref{lem:YX_YO_diff_Lp}]
Observe that Lemma~\ref{lem:YX_YO_diff} 
and H{\"o}lder's inequality
imply
\begin{align}
\begin{split}
 &\big\|
    \YXM{t} - \YOM{t}
  \big\|_{\L^p(\P;\BaMi)}
\\&\leq
  \frac{T^{\cp}}{M^{\cp}}
  \left[ 
    \sup_{s\in(0,T)}
    s^{\sp}
    \big\|
      e^{ sA }
    \big\|_{ L( H, \BaMi ) }
  \right]
  \int_0^t
  \left( t-s \right)^{-\sp }
  \left\|
    \big|
      \Vm( \YM{\fl{s}} ) 
    \big|^{\cp}
    \big\|
      F( \YM{\fl{s}} )
    \big\|_{ H } 
  \right\|_{ \L^p( \P; \R ) } ds
\\&\quad+
  \int_0^t
  \big\|
    e^{ (t-s)A }
    -
    \SM_{\fl{s},t}
  \big\|_{ L( H, \BaMi ) }
  \big\|
    F( \YM{\fl{s}} )
  \big\|_{\L^p(\P;H)} \, ds
\\&\leq
  \frac{T^{\cp}}{M^{\cp}}
  \left[ 
    \sup_{s\in(0,T)}
    s^{\sp}
    \big\|
      e^{ sA }
    \big\|_{ L( H, \BaMi ) }
  \right]
  \int_0^t
  \left( t-s \right)^{-\sp }
  \big\|
    \Vm( \YM{\fl{s}} ) 
  \big\|_{ \L^{2p\cp}( \P; \R ) }^{\cp}
  \big\|
    F( \YM{\fl{s}} )
  \big\|_{ \L^{2p}( \P; H ) } \, ds
\\&\quad+
  \left( 
    \int_0^t
    \big\|
      e^{ (t-s)A }
      -
      \SM_{\fl{s},t}
    \big\|_{ L( H, \BaMi ) } \, ds
  \right)
  \left[
    \sup_{ s \in [0,t) }
    \big\|
      F( \YM{s} )
    \big\|_{\L^p(\P;H)}
  \right] .
\end{split}
\end{align}
The proof of Lemma~\ref{lem:YX_YO_diff_Lp} is thus completed.
\end{proof}

\subsection{Analysis of the error between
the exact solution and the semilinear integrated
version of the numerical approximation}

\begin{lemma}
\label{lem:XO_YX_diff_Lp}
Assume the setting in Section~\ref{sec:strong_setting}
and let $ \kappa \in (2,\infty) $,
$ t \in [0,T] $, $ p \in [2,\infty) $, $ M \in \N $.
Then 
\begin{align}
\begin{split}
 &\big\|
    \XO{t} - \YXM{t}
  \big\|_{ \L^p(\P; H ) }
\\&\leq
  \frac{
    e^{ \kappa C T/2 }
    \sqrt{ 
      2^{[\qp-1]^+} 
      L
    }
  }{
    \sqrt{ C\left(\kappa-2\right) }
  }
  \left[
    1
    +
    \sup_{ s \in (0,T) }
    \big\| \YXM{s} \big\|_{ \L^{p\qp}(\P;\BaMi) }^{\qp/2} 
    +
    \sup_{ s \in (0,T) }
    \big\| O_s \big\|_{ \L^{p\qp}(\P;\BaMi) }^{\qp/2} 
    +
    \sup_{ s \in [0,T) }
    \big\| \YM{s} \big\|_{ \L^{p\qp}(\P;\BaMi) }^{\qp/2} 
  \right]
\\&\quad\cdot
  \Bigg[
    \int_0^t
    \Big(
      \big\|
	\YXM{s}
	-
	\YXM{\fl{s}}
      \big\|_{ \L^{2p}(\P;\BaMi) }
      +
      \big\|
	\YXM{\fl{s}}
	-
	\YOM{\fl{s}}
      \big\|_{ \L^{2p}(\P;\BaMi) }
      +
      \big\|
	O_s
	-
	O_{\fl{s}}
      \big\|_{ \L^{2p}(\P;\BaMi) }
\\&\quad+
      \big\|
	O_{\fl{s}}
	-
	\OM{\fl{s}}
      \big\|_{ \L^{2p}(\P;\BaMi) }
    \Big)^2 ds
  \Bigg]^{\frac{1}{2}} .
\end{split}
\end{align}
\end{lemma}
\begin{proof}[Proof of Lemma~\ref{lem:XO_YX_diff_Lp}]
Note that Lemma~\ref{lem:XO_YX_diff} 
and H{\"o}lder's inequality imply
\begin{align}
\begin{split}
&\big\|
    \XO{t} - \YXM{t}
  \big\|_{ \L^p(\P; H ) }
\\&\leq
  \frac{
    e^{ \kappa C T/2 }
    \sqrt{ 
      2^{[\qp-1]^+} 
      L
    }
  }{
    \sqrt{ C\left(\kappa-2\right) }
  }
  \Bigg[
    \int_0^t
    \bigg\|
      \left(
	1
	+
	\big\| \YXM{s} \big\|_{\BaMi}^{\qp}
	+ 
	\big\| O_s \big\|_{\BaMi}^{\qp}
	+
	\big\| \YM{\fl{s}} \big\|_{\BaMi}^{\qp}
      \right)
      \Big(
	\big\|
	  \YXM{s}
	  -
	  \YXM{\fl{s}}
	\big\|_{\BaMi}
\\&\quad+
	\big\|
	  \YXM{\fl{s}}
	  -
	  \YOM{\fl{s}}
	\big\|_{\BaMi}
	+
	\big\|
	  O_s
	  -
	  O_{\fl{s}}
	\big\|_{\BaMi}
	+
	\big\|
	  O_{\fl{s}}
	  -
	  \OM{\fl{s}}
	\big\|_{\BaMi}
      \Big)^2
    \bigg\|_{ \L^{\nicefrac{p}{2}}(\P;\R) } \, ds
  \Bigg]^{\frac{1}{2}}
\\&\leq
  \frac{
    e^{ \kappa C T/2 }
    \sqrt{ 
      2^{[\qp-1]^+} 
      L
    }
  }{
    \sqrt{ C\left(\kappa-2\right) }
  }
  \Bigg[
    \int_0^t
    \left\|
      1
      +
      \big\| \YXM{s} \big\|_{\BaMi}^{\qp}
      + 
      \big\| O_s \big\|_{\BaMi}^{\qp}
      +
      \big\| \YM{\fl{s}} \big\|_{\BaMi}^{\qp}
    \right\|_{ \L^p(\P;\R) } 
    \Big\|
      \big\|
	\YXM{s}
	-
	\YXM{\fl{s}}
      \big\|_{\BaMi}
\\&\quad+
      \big\|
	\YXM{\fl{s}}
	-
	\YOM{\fl{s}}
      \big\|_{\BaMi}
      +
      \big\|
	O_s
	-
	O_{\fl{s}}
      \big\|_{\BaMi}
      +
      \big\|
	O_{\fl{s}}
	-
	\OM{\fl{s}}
      \big\|_{\BaMi}
    \Big\|_{ \L^{2p}(\P;\R) }^2 ds
  \Bigg]^{\frac{1}{2}} .
\end{split}
\end{align}
This shows
\begin{align}
\begin{split}
&\big\|
    \XO{t} - \YXM{t}
  \big\|_{ \L^p(\P; H ) }
\\&\leq
  \frac{
    e^{ \kappa C T/2 }
    \sqrt{ 
      2^{[\qp-1]^+} 
      L
    }
  }{
    \sqrt{ C\left(\kappa-2\right) }
  }
  \Bigg[
    \int_0^t
    \Big(
      1
      +
      \big\| \YXM{s} \big\|_{ \L^{p\qp}(\P;\BaMi) }^{\qp}
      + 
      \big\| O_s \big\|_{ \L^{p\qp}(\P;\BaMi) }^{\qp}
      +
      \big\| \YM{\fl{s}} \big\|_{ \L^{p\qp}(\P;\BaMi) }^{\qp}
    \Big)
\\&\quad\cdot
    \Big(
      \big\|
	\YXM{s}
	-
	\YXM{\fl{s}}
      \big\|_{ \L^{2p}(\P;\BaMi) }
      +
      \big\|
	\YXM{\fl{s}}
	-
	\YOM{\fl{s}}
      \big\|_{ \L^{2p}(\P;\BaMi) }
      +
      \big\|
	O_s
	-
	O_{\fl{s}}
      \big\|_{ \L^{2p}(\P;\BaMi) }
\\&\quad+
      \big\|
	O_{\fl{s}}
	-
	\OM{\fl{s}}
      \big\|_{ \L^{2p}(\P;\BaMi) }
    \Big)^2 ds
  \Bigg]^{\frac{1}{2}} .
\end{split}
\end{align}
The fact that 
$ 
  \forall \,
  x_1, x_2, x_3, x_4
  \in [0,\infty) 
  \colon
  \sqrt{x_1 + x_2 + x_3 + x_4}
  \leq
  \sqrt{x_1} + \sqrt{x_2} + \sqrt{x_3} + \sqrt{x_4}
$
completes the proof of Lemma~\ref{lem:XO_YX_diff_Lp}.
\end{proof}

\subsection{Analysis of the error between the 
exact solution and the numerical approximation}

\begin{lemma}[Strong temporal regularity 
of the semilinear integrated version of the numerical approximation]
\label{lem:YX_YX_diff_Lp}
Assume the setting in Section~\ref{sec:strong_setting}
and let
$ p \in [1,\infty) $, $ \spp \in [0, 1-\sp) $,
$ t_1, t_2 \in [0,T] $,
$ M \in \N $ with $ t_1 < t_2 $. Then
\begin{align}
\begin{split}
 &\big\|
    \YXM{t_2}
    -
    \YXM{t_1}
  \big\|_{ \L^p( \P; \BaMi ) }
  \leq
  \left(t_2-t_1\right)^{\spp}
  T^{(1-\sp-\spp)}
  \left[
    \sup_{ s \in (0,T) }
    s^{\sp}
    \big\|
      e^{sA}
    \big\|_{ L( H, \BaMi ) }
  \right]
  \left[
    \sup_{ s \in [0,T) } 
    \big\|
      F( \YM{s} )
    \big\|_{ \L^p(\P; H ) }
  \right]
\\&\cdot
  \Bigg(
    \frac{1}{(1-\sp)}
    +
    \frac{ 
      2^{(\sp+\spp)}
    }{
      (1-\sp-\spp)
    }
    \left[
      \sup_{ s \in (0,T) }
      \big\|
	(-sA)^{\varrho}
	e^{sA}
      \big\|_{ L(H) }
    \right]
    \left[
      \sup_{ s \in (0,T) }
      \big\|
	(-sA)^{-\varrho}
	\big(
	  e^{sA}
	  -
	  \Id_H
	\big)
      \big\|_{ L(H) }
    \right]
  \Bigg) .
\end{split}
\end{align}
\end{lemma}
\begin{proof}[Proof of Lemma~\ref{lem:YX_YX_diff_Lp}]
Note that Lemma~\ref{lem:YX_YX_diff} implies
\begin{align}
\begin{split}
 &\big\|
    \YXM{t_2}
    -
    \YXM{t_1}
  \big\|_{ \L^p( \P; \BaMi ) }
  \leq
  \left[
    \sup_{ s \in (0,T) }
    s^{\sp}
    \big\|
      e^{sA}
    \big\|_{ L( H, \BaMi ) }
  \right]
  \int_{t_1}^{t_2}
  \left(
    t_2-s
  \right)^{-\sp}
  \big\|
    F( \YM{\fl{s}} )
  \big\|_{ \L^p(\P; H ) } \, ds
\\&\quad+
  2^{(\sp+\spp)}
  \left( t_2-t_1 \right)^{\spp}
  \left[
    \sup_{ s \in (0,T) }
    s^{\sr}
    \big\|
      e^{sA}
    \big\|_{ L( H, \BaMi ) }
  \right]
  \left[
    \sup_{ s \in (0,T) }
    \big\|
      (-sA)^{\varrho}
      e^{sA}
    \big\|_{ L(H) }
  \right]
\\&\quad\cdot
  \left[
    \sup_{ s \in (0,T) }
    \big\|
      (-sA)^{-\varrho}
      \big(
	e^{sA}
	-
	\Id_H
      \big)
    \big\|_{ L(H) }
  \right]
  \int_{0}^{t_1}
  \left(
    t_1-s
  \right)^{-(\sp+\spp)}
  \big\|
    F( \YM{\fl{s}} )
  \big\|_{ \L^p(\P; H ) } \, ds .
\end{split}
\end{align}
Hence, we obtain
\begin{align}
\begin{split}
 &\big\|
    \YXM{t_2}
    -
    \YXM{t_1}
  \big\|_{ \L^p( \P; \BaMi ) }
  \leq
  \frac{(t_2-t_1)^{(1-\sp)}}{(1-\sp)}
  \left[
    \sup_{ s \in (0,T) }
    s^{\sp}
    \big\|
      e^{sA}
    \big\|_{ L( H, \BaMi ) }
  \right]
  \left[
    \sup_{ s \in [0,T) } 
    \big\|
      F( \YM{s} )
    \big\|_{ \L^p(\P; H ) }
  \right]
\\&\quad+
  \frac{ 
    2^{(\sp+\spp)}
    \left( t_2-t_1 \right)^{\spp}
    \big| t_1 \big|^{(1-\sp-\spp)}
  }{
    (1-\sp-\spp)
  }
  \left[
    \sup_{ s \in (0,T) }
    s^{\sr}
    \big\|
      e^{sA}
    \big\|_{ L( H, \BaMi ) }
  \right]
  \left[
    \sup_{ s \in (0,T) }
    \big\|
      (-sA)^{\varrho}
      e^{sA}
    \big\|_{ L(H) }
  \right]
\\&\quad\cdot
  \left[
    \sup_{ s \in (0,T) }
    \big\|
      (-sA)^{-\varrho}
      \big(
	e^{sA}
	-
	\Id_H
      \big)
    \big\|_{ L(H) }
  \right]
  \left[
    \sup_{ s \in [0,T) } 
    \big\|
      F( \YM{s} )
    \big\|_{ \L^p(\P; H ) }
  \right] .
\end{split}
\end{align}
This completes the proof Lemma~\ref{lem:YX_YX_diff_Lp}.
\end{proof}

\begin{proposition}
\label{prop:X_YM_Lp_diff} 
Assume the setting in 
Section~\ref{sec:strong_setting},
let $ \cp \in (0,\infty) $,
$ \np \in (0,\nicefrac{1}{2}) $,
$ \spp \in [0, 1-\sp) $,
$ p \in [\max\{2,\nicefrac{1}{\qp}\},\infty) $,
and assume that 
\begin{align}
\label{eq:X_YM_diff_assumption}
\begin{split}
 &\sup_{ M \in \N }
  \sup_{ t \in [0,T] }
  \big\|
    \YM{t}
  \big\|_{ \L^{p\max\{1,\qp\}}\!(\P;\BaMi) }
  +
  \sup_{ M \in \N }
  \sup_{ t \in [0,T) }
  \Big\|
    \big|
      \Vm( \YM{t} )
    \big|^{2\cp}
    +
    \big\|
      F( \YM{t} )
    \big\|_H^2
  \Big\|_{ \L^{p\max\{2,\qp\}}\!(\P;\R) }
\\&+
  \sup_{ M \in \N }
  \sup_{ t \in (0,T] }
  M^{\spp}
  \int_0^t
  \big\|
    e^{(t-s)A}
    -
    \SM_{\fl{s},t}
  \big\|_{ L(H,\BaMi) } \, ds
  +
  \sup_{ t \in [0,T) }
  \big\|
    O_t
  \big\|_{ \L^{p\max\{2,\qp\}}\!(\P;\BaMi) }
\\&+
  \sup_{ M \in \N }
  \sup_{ t \in [0,T] }
  M^{\np}
  \left[ 
    \big\|
      O_t
      -
      \OM{t}
    \big\|_{ \L^{p\max\{2,\qp\}}\!(\P;\BaMi) }
    +
    \big|\fl{t}\big|^{\np}
    \big\|
      O_t
      -
      O_{\fl{t}}
    \big\|_{ \L^{2p}(\P;\BaMi) }
  \right]
  <
  \infty .
\end{split}
\end{align}
Then
$
  \sup_{ M \in \N }
  \sup_{ t \in [0,T] }
  \big[
    \|
      X_t
    \|_{ \L^p(\P;H) }
    +
    M^{\min\{\cp, \spp, \np \}}
    \|
      X_t - \YM{t}
    \|_{ \L^p(\P;H) }
  \big]
  <
  \infty
$.
\end{proposition}
\begin{proof}[Proof of Proposition~\ref{prop:X_YM_Lp_diff}]
First of all, observe that
the triangle inequality
and~\eqref{eq:X_YM_diff_assumption} implies 
\begin{align}
\label{eq:strong_err_4}
\begin{split}
 &\sup_{ M \in \N }
  \sup_{ t \in [0,T) }
  \big\|
    \OM{t}
  \big\|_{ \L^{p\max\{2,\qp\}}\!(\P;\BaMi) }
  \leq
  \sup_{ M \in \N }
  \sup_{ t \in [0,T) }
  \Big\|
    \big\|
      \OM{t}
      -
      O_t
    \big\|_{ \BaMi }
    +
    \big\|
      O_t
    \big\|_{ \BaMi }
  \Big\|_{ \L^{p\max\{2,\qp\}}\!(\P;\R) }
  <
  \infty .
\end{split}
\end{align}
This, the triangle inequality,
and~\eqref{eq:X_YM_diff_assumption}
assure that
\begin{align}
\label{eq:strong_err_01}
\begin{split}
  \sup_{ M \in \N }
  \sup_{ t \in [0,T) }
  \big\|
    \YOM{t}
  \big\|_{ \L^{p\qp}(\P;\BaMi) }
 &\leq
  \sup_{ M \in \N }
  \sup_{ t \in [0,T) }
  \left[ 
    \big\|
      \YM{t}
    \big\|_{ \L^{p\max\{1,\qp\}}\!(\P;\BaMi) }
    +
    \big\|
      \OM{t}
    \big\|_{ \L^{p\max\{2,\qp\}}\!(\P;\BaMi) }
  \right]
  <
  \infty .
\end{split}
\end{align}
Next note that
Lemma~\ref{lem:YX_YO_diff_Lp},
\eqref{eq:X_YM_diff_assumption}, 
and the assumption that
$
  \sup_{ t \in (0,T) }
  t^{\sp}
  \|
    e^{tA}
  \|_{ L(H, \BaMi) }
  <
  \infty
$
prove that
\begin{align}
\label{eq:strong_err_1}
\begin{split}
 &\sup_{ M \in \N }
  \sup_{ t \in [0,T] }
  M^{\min\{\cp,\spp\} }
  \big\|
    \YXM{t} - \YOM{t}
  \big\|_{\L^{p\max\{2,\qp\}}\!(\P;\BaMi)}
  =
  \sup_{ M \in \N }
  \sup_{ t \in (0,T] }
  M^{\min\{\cp,\spp\} }
  \big\|
    \YXM{t} - \YOM{t}
  \big\|_{\L^{p\max\{2,\qp\}}\!(\P;\BaMi)}
\\&\leq
  \sup_{ M \in \N }
  \sup_{ t \in (0,T] }
  \left\{
    \left(
      M^{\spp}
      \int_0^t
      \big\|
	e^{ (t-s)A }
	-
	\SM_{\fl{s},t}
      \big\|_{ L( H, \BaMi ) } \, ds
    \right)
    \left[
      \sup_{ s \in [0,T) }
      \big\|
	F( \YM{s} )
      \big\|_{\L^{p\max\{2,\qp\}}\!(\P;H)}
    \right]
  \right\}
\\&+
  \frac{T^{(1+\cp-\sp)}}{(1-\sp)}
  \left[ 
    \sup_{s\in(0,T)}
    s^{\sp}
    \big\|
      e^{ sA }
    \big\|_{ L( H, \BaMi ) }
  \right]
  \left[
    \sup_{ M \in \N }
    \sup_{s\in[0,T)}
    \big\|
      \Vm( \YM{s} )
    \big\|_{ \L^{2\cp p\max\{2,\qp\}}\!( \P; \R ) }^{\cp}
    \big\|
      F( \YM{s} )
    \big\|_{ \L^{2p\max\{2,\qp\}}\!( \P; H ) }
  \right]
\\&<
  \infty .
\end{split}
\end{align}
Combining this, the triangle inequality,
and~\eqref{eq:strong_err_01} ensures
\begin{align}
\label{eq:strong_err_2}
\begin{split}
  \sup_{ M \in \N }
  \sup_{ t \in [0,T) }
  \big\|
    \YXM{t}
  \big\|_{ \L^{p\qp}(\P;\BaMi) }
 &\leq
  \sup_{ M \in \N }
  \sup_{ t \in [0,T) }
  \left[
    \big\|
      \YXM{t} - \YOM{t}
    \big\|_{\L^{p\max\{2,\qp\}}\!(\P;\BaMi)}
    +
    \big\|
      \YOM{t}
    \big\|_{ \L^{p\qp}(\P;\BaMi) }
  \right]
  <
  \infty .
\end{split}
\end{align}
Furthermore, observe that
the fact that
$ A \colon D(A) \subseteq H \rightarrow H $
is a generator of an analytic
semigroup with
$ 
  \text{spectrum}(A)
  \subseteq
  \{ 
    z \in \C
    \colon
    \text{Re}(z) < 0
  \}
$
implies that
$ 
  \sup_{ s \in (0,T) }
  \|
    (-sA)^{-\varrho}
    \big(
      e^{sA}
      -
      \Id_H
    \big)
  \|_{ L(H) }
  +
  \sup_{ s \in (0,T) }
  \|
    (-sA)^{\varrho}
    e^{sA}
  \|_{ L(H) }
  <
  \infty
$.
This, the assumption that
$
  \sup_{ t \in (0,T) }
  t^{\sp}
  \|
    e^{tA}
  \|_{ L(H, \BaMi) }
  <
  \infty
$,
\eqref{eq:X_YM_diff_assumption},
and Lemma~\ref{lem:YX_YX_diff_Lp}
show
\begin{align}
\label{eq:strong_err_3}
\begin{split}
 &\sup_{ M \in \N }
  \sup_{ t \in [0,T] }
  M^{ \spp }
  \big\|
    \YXM{t}
    -
    \YXM{\fl{t}}
  \big\|_{ \L^{2p}( \P; \BaMi ) }
\\&\leq
  \Bigg(
    \frac{1}{(1-\sp)}
    +
    \frac{ 
      2^{(\sp+\spp)}
    }{
      (1-\sp-\spp)
    }
    \left[
      \sup_{ s \in (0,T) }
      \big\|
	(-sA)^{\varrho}
	e^{sA}
      \big\|_{ L(H) }
    \right]
    \left[
      \sup_{ s \in (0,T) }
      \big\|
	(-sA)^{-\varrho}
	\big(
	  e^{sA}
	  -
	  \Id_H
	\big)
      \big\|_{ L(H) }
    \right]
  \Bigg)
\\&\quad\cdot
  T^{(1-\sp)}
  \left[
    \sup_{ s \in (0,T) }
    s^{\sp}
    \big\|
      e^{sA}
    \big\|_{ L( H, \BaMi ) }
  \right]
  \left[
    \sup_{ M \in \N }
    \sup_{ s \in [0,T) } 
    \big\|
      F( \YM{s} )
    \big\|_{ \L^{2p}(\P; H ) }
  \right]
  <
  \infty .
\end{split}
\end{align}
Moreover, note that 
Lemma~\ref{lem:XO_YX_diff_Lp},
the triangle inequality,
and the fact that 
$
  \forall \,
  x,y \in \R
  \colon
  (x+y)^2
  \leq
  2(x^2+y^2)
$
assure for all
$ \kappa \in (2,\infty) $
that
\begin{align}
\label{eq:strong_err_5}
\begin{split}
 &\sup_{ M \in \N }
  \sup_{ t \in [0,T] }
  M^{\min\{\cp, \spp, \np \}}
  \big\|
    \XO{t} - \YXM{t}
  \big\|_{ \L^p(\P; H ) }
\\&\leq
  \frac{
    e^{ \kappa C T/2 }
    \sqrt{ 
      2^{[\qp-1]^+} 
      L
    }
  }{
    \sqrt{ C\left(\kappa-2\right) }
  }
  \left[
    1
    +
    \sup_{ M \in \N }
    \sup_{ s \in (0,T) }
    \big\| \YXM{s} \big\|_{ \L^{p\qp}(\P;\BaMi) }^{\qp/2} 
    +
    \sup_{ s \in (0,T) }
    \big\| O_s \big\|_{ \L^{p\qp}(\P;\BaMi) }^{\qp/2} 
    +
    \sup_{ M \in \N }
    \sup_{ s \in [0,T) }
    \big\| \YM{s} \big\|_{ \L^{p\qp}(\P;\BaMi) }^{\qp/2} 
  \right]
\\&\cdot
  \sup_{ M \in \N }
  M^{\min\{\cp, \spp, \np \}}
  \Bigg\{
    \Bigg[
      \int_0^{\frac{T}{M}}
      \Big(
	\big\|
	  O_s
	\big\|_{ \L^{2p}(\P;\BaMi) }
	+
	\big\|
	  O_{\fl{s}}
	\big\|_{ \L^{2p}(\P;\BaMi) }
      \Big)^2 \, ds
      +
      \int_{\frac{T}{M}}^{T}
      \big\|
	O_s
	-
	O_{\fl{s}}
      \big\|_{ \L^{2p}(\P;\BaMi) }^2 \, ds
    \Bigg]^{\frac{1}{2}}
\\&+
    \Bigg[
      \int_0^T
      \Big(
	\big\|
	  \YXM{s}
	  -
	  \YXM{\fl{s}}
	\big\|_{ \L^{2p}(\P;\BaMi) }
	+
	\big\|
	  \YXM{\fl{s}}
	  -
	  \YOM{\fl{s}}
	\big\|_{ \L^{2p}(\P;\BaMi) }
\\&+
	\big\|
	  O_{\fl{s}}
	  -
	  \OM{\fl{s}}
	\big\|_{ \L^{2p}(\P;\BaMi) }
      \Big)^2 ds
    \Bigg]^{\frac{1}{2}}
  \Bigg\} .
\end{split}
\end{align}
The estimate
\begin{align}
\begin{split}
 &\sup_{ M \in \N }
  \left[ 
    \int_{\frac{T}{M}}^{T}
    \big| \fl{s} \big|^{-2\np} \, ds
  \right]
  =
  \sup_{ M \in \N }
  \left[ 
    \sum_{ l=1 }^{ M-1 }
    \int_{\frac{lT}{M}}^{\frac{(l+1)T}{M}}
    \frac{ M^{2\np} }{(lT)^{2\np}} \, ds
  \right]
  =
  \sup_{ M \in \N }
  \left[ 
    \frac{T^{(1-2\np)}}{M^{(1-2\np)}}
    \sum_{ l=1 }^{ M-1 }
    \frac{ 1 }{l^{2\np}}
  \right]
\\&\leq
  \sup_{ M \in \N }
  \left[ 
    \frac{T^{(1-2\np)}}{M^{(1-2\np)}}
    \left(
      1
      +
      \int_1^M
      \frac{ 1 }{s^{2\np}} \, ds 
    \right)
  \right]
  =
  \sup_{ M \in \N }
  \left[ 
    \frac{T^{(1-2\np)}}{M^{(1-2\np)}}
    \left(
      1
      +
      \frac{ (M^{(1-2\np)} - 1) }{(1-\np)}
    \right)
  \right]
  \leq
  \frac{ T^{(1-2\np)} }{ (1-\np) } ,
\end{split}
\end{align}
inequality~\eqref{eq:X_YM_diff_assumption}, 
inequalities~\eqref{eq:strong_err_1}--\eqref{eq:strong_err_3},
and the fact that
$
  \forall \,
  x,y \in [0,\infty)
  \colon
  \sqrt{x+y}
  \leq
  \sqrt{x}
  +
  \sqrt{y}
$
therefore imply that
\begin{align}
\label{eq:strong_err_6}
\begin{split}
 &\sup_{ M \in \N }
  \sup_{ t \in [0,T] }
  M^{\min\{\cp, \spp, \np \}}
  \big\|
    \XO{t} - \YXM{t}
  \big\|_{ \L^p(\P; H ) }
\\&\leq
  \frac{
    e^{ \kappa C T/2 }
    \sqrt{ 
      2^{[\qp-1]^+} 
      L
    }
  }{
    \sqrt{ C\left(\kappa-2\right) }
  }
  \left[
    1
    +
    \sup_{ M \in \N }
    \sup_{ s \in (0,T) }
    \big\| \YXM{s} \big\|_{ \L^{p\qp}(\P;\BaMi) }^{\qp/2} 
    +
    \sup_{ s \in (0,T) }
    \big\| O_s \big\|_{ \L^{p\qp}(\P;\BaMi) }^{\qp/2} 
    +
    \sup_{ M \in \N }
    \sup_{ s \in [0,T) }
    \big\| \YM{s} \big\|_{ \L^{p\qp}(\P;\BaMi) }^{\qp/2} 
  \right]
\\&\cdot
  \sqrt{(1+T)}
  \sup_{ M \in \N }
  \Bigg[
    2 \!
    \sup_{ s \in [0,T) }
    \big\|
      O_s
    \big\|_{ \L^{2p}(\P;\BaMi) }
    + \!\!
    \sup_{ s \in (0,T) }
    \big|M\fl{s}\big|^{\np}  
    \big\|
      O_s
      -
      O_{\fl{s}}
    \big\|_{ \L^{2p}(\P;\BaMi) } \!
    \left(
      \smallint_{\frac{T}{M}}^{T}
      \big|\fl{s}\big|^{-2\np} ds
    \right)^{\!\frac{1}{2}}
\\&+
    \sup_{ s \in [0,T] }
    \left(
      M^{\spp}
      \big\|
	\YXM{s}
	-
	\YXM{\fl{s}}
      \big\|_{ \L^{2p}(\P;\BaMi) }
      +
      M^{\min\{\cp, \spp\}}
      \big\|
	\YXM{s}
	-
	\YOM{s}
      \big\|_{ \L^{2p}(\P;\BaMi) }
      +
      M^{\np}
      \big\|
	O_{s}
	-
	\OM{s}
      \big\|_{ \L^{2p}(\P;\BaMi) }
    \right)
  \bigg]
\\&<
  \infty .
\end{split}
\end{align}
Next note that the triangle inequality,
the assumption that
$
  \BaMi \subseteq H
$ continuously,
\eqref{eq:X_YM_diff_assumption}, 
\eqref{eq:strong_err_1},
and~\eqref{eq:strong_err_6} prove
\begin{align}
\begin{split}
 &\sup_{ M \in \N }
  \sup_{ t \in [0,T] }
  M^{\min\{\cp, \spp, \np \}}
  \big\|
    X_t - \YM{t}
  \big\|_{ \L^p(\P;H) }
\\&\leq
  \sup_{ M \in \N }
  \sup_{ t \in [0,T] }
  M^{\min\{\cp, \spp, \np \}}
  \left[
    \big\|
      \XO{t} - \YOM{t}
    \big\|_{ \L^p(\P;H) }
    +
    \big\|
      O_t - \OM{t}
    \big\|_{ \L^p(\P;H) }
  \right]
\\&\leq
  \sup_{ M \in \N }
  \sup_{ t \in [0,T] } 
  M^{\min\{\cp, \spp, \np \}}
  \big\|
    \XO{t} - \YXM{t}
  \big\|_{ \L^p(\P;H) }
\\&\quad+
  \left[
    \sup_{ v \in \BaMi \backslash \{0\} }
    \frac{ \| v \|_H }{ \| v \|_{\BaMi} }
  \right] 
  \sup_{ M \in \N }
  \sup_{ t \in [0,T] } 
  M^{\min\{\cp, \spp \}}
  \big\|
    \YXM{t} - \YOM{t}
  \big\|_{\L^{p\max\{2,\qp\}}\!(\P;\BaMi)}
\\&\quad+
  \left[
    \sup_{ v \in \BaMi \backslash \{0\} }
    \frac{ \| v \|_H }{ \| v \|_{\BaMi} }
  \right] 
  \sup_{ M \in \N }
  \sup_{ t \in [0,T] } 
  M^{ \np }
  \big\|
    O_t - \OM{t}
  \big\|_{\L^{p\max\{2,\qp\}}\!(\P;\BaMi)}
  <
  \infty .
\end{split}
\end{align}
This,
the triangle inequality,
the assumption that
$
  \BaMi \subseteq H
$ continuously,
and \eqref{eq:X_YM_diff_assumption}
assure that
\begin{align}
\begin{split}
  \sup_{ t \in [0,T] }
  \big\|
    X_t
  \big\|_{ \L^p(\P;H) }
 &\leq
  \sup_{ M \in \N }
  \sup_{ t \in [0,T] }
  \left[
    \big\|
      X_t - \YM{t}
    \big\|_{ \L^p(\P;H) }
    +
    \left[
      \sup_{ v \in \BaMi \backslash \{0\} }
      \frac{ \| v \|_H }{ \| v \|_{\BaMi} }
    \right] 
    \big\|
      \YM{t}
    \big\|_{ \L^{p\max\{1,\qp\}}\!(\P;\BaMi) }
  \right]
  <
  \infty .
\end{split}
\end{align}
The proof of Proposition~\ref{prop:X_YM_Lp_diff} is thus completed.
\end{proof}

\section{Main result}
\label{sec:main_result}

\subsection{Setting}
\label{sec:main_result_setting}
Let $ T, c, \qmm, \qm, \cm \in (0,\infty) $,
let $ ( H, \left< \cdot, \cdot \right>_H, \left\| \cdot \right\|_H ) $
be a separable $ \R $-Hilbert space,
let $ A \colon D(A) \subseteq H \rightarrow H $
be a generator of an analytic semigroup 
with 
$ 
  \textup{spectrum}(A) 
  \subseteq 
  \{
    z \in \mathbb{C} 
    \colon 
    \textup{Re}(z) < 0
  \}
$,
let $ ( H_r, \left< \cdot, \cdot \right>_{ H_r }, \left\| \cdot \right\|_{ H_r } ) $, 
$ r \in \R $, be a family of interpolation spaces associated to $ -A $
(cf., e.g., Definition~3.5.26 in~\cite{Jentzen2015SPDElecturenotes}),
let $ (\BaMi, \left\| \cdot \right\|_{\BaMi}) $
be a separable $ \R $-Banach space with 
$ H_{1} \subseteq \BaMi \subseteq H $ 
densely and continuously,
let $ F \in \C( \BaMi, H ) $,
$ \Um, \Rm \in \M\big( \B(\BaMi), \B([0,\infty)) \big) $,
$ (\SM)_{M\in\N} \subseteq \M\big( \B(\{ (s, t) \in [0,T]^2 \colon s < t \}), \B(L(H, \BaMi) ) \big) $
satisfy for all $ u, v, x, y \in \BaMi $,
$ M \in \N $,
$ r_1,r_2,r_3 \in [0,T] $,
$ s \in \{ 0, \frac{T}{M}, \frac{2T}{M}, \ldots, \frac{(M-1)T}{M} \} $,
$ t \in (s,s+\frac{T}{M} ] $ with
$ x-y \in H_1 $ and
$ r_1 < r_2 < r_3 $ 
that
$ 
  \left\| F(u + v) \right\|_{ H } 
  \leq 
  c\,( 1 + |\Um(u)|^{\qm} + |\Um(v)|^{\qm})
$,
$
  \left\|
    F(x) - F(y)
  \right\|_{H}^2
  \leq
  c
  \left\|
    x - y
  \right\|_V^2
  \left(
    1
    +
    \left\| x \right\|_{\BaMi}^{\qmm}
    +
    \left\| y \right\|_{\BaMi}^{\qmm}
  \right)
$,
$
  \left<
    x - y,
    A(x-y) + F(x)
    - F(y)
  \right>_H
  \leq
  c
  \left\| 
    x - y 
  \right\|_H^2
$,
$ 
  \SM_{r_1,r_3}
  =
  \SM_{r_2,r_3}
  \SM_{r_1,r_2}
$,
and
\begin{align}
\label{eq:assume_u_main}
\begin{split}
 &\Um\!\left(
    \SM_{s,t}
    \left[
      u
      +
      (t-s) \,
      \one_{ 
	[0, (M/T)^{\cm} ]
      }
      ( \| u+v \|_{ \BaMi } ) \,
      F( u + v )
    \right]
  \right)
  \leq
  e^{c(t-s)}
  \left[
    \Um(u)
    +
    (t-s)
    \Rm(v)
  \right], 
\end{split}
\end{align}
let $ (\Omega, \F, \P ) $ 
be a probability space,
let $ O, X \colon [0,T] \times \Omega \rightarrow \BaMi $ 
be stochastic processes with continuous sample paths,
let $ \OM{}, \YM{} \colon [0,T] \times \Omega \rightarrow \BaMi $,
$ M \in \N $, be stochastic processes,
and assume for all 
$ t \in [0,T] $,
$ M \in \mathbb{N} $ that
$
  \P\big(
    \limsup_{ r \searrow 0 }
    \sup_{ 0 \leq s < u \leq T }
    \frac{ 
      s \| O_u - O_s \|_{ \BaMi } 
    }{
      (u-s)^{r}
    }
    <
    \infty
  \big)
  =
  \P\big(
    X_t
    =
    \int_0^t
    e^{(t-s)A}
    F(X_s) \, ds
    +
    O_t
  \big)
  =
  \P\big(
    \YM{t}
    =
    \int_0^t
    \SM_{\fl{s},t} \,
    \one_{ 
      \{
	\| \YM{\fl{s}} \|_{\BaMi}
	\leq
	(M/T)^{\cm}
      \}
    } \,
    F( \YM{\fl{s}} ) \, ds
    +
    \OM{t}
  \big)
  =
  1
$.

\subsection{A priori bounds}

\begin{corollary}
\label{cor:YM_bound}
Assume the setting in Section~\ref{sec:main_result_setting},
let $ p \in [\max\{1,\nicefrac{1}{\qm}\}, \infty) $, 
$ \sm \in [0,1) $, and assume that
$
  \sup_{ M \in \N }
  \sup_{ 0 \leq s < t < T }
  \left(t-s\right)^{\sm}
  \|
    \SM_{\fl{s},t}
  \|_{ L(H,\BaMi) }
  +
  \sup_{ M \in \N }
  \sup_{ t \in [0,T) }
  \|
    \Um( \OM{t} )
    +
    \Rm( \OM{t} )
  \|_{ \L^{p\qm}(\P;\R) }
$
$
  +
  \sup_{ M \in \N }
  \sup_{ t \in [0,T] }
  \|
    \OM{t}
  \|_{ \L^{p}(\P;\BaMi) }
  <
  \infty
$. Then 
$
  \sup_{M \in \N}
  \sup_{t\in[0,T]}
  \|
    \YM{t}
  \|_{ \L^p(\P;\BaMi) }
  <
  \infty
$.
\end{corollary}
\begin{proof}[Proof of Corollary~\ref{cor:YM_bound}]
Throughout this proof let
$ \YMM{} \colon [0,T] \times \Omega \rightarrow \BaMi $,
$ M \in \N $,
be the stochastic processes which satisfy
for all $ t \in [0,T] $, $ M \in \N $ that
\begin{align}
\begin{split}
 &\YMM{t}
  =
  \int_0^t
  \SM_{\fl{s},t} \,
  \one_{ 
    \{
      \| \YMM{\fl{s}} \|_{\BaMi}
      \leq
      (M/T)^{\cm}
    \}
  } \,
  F( \YMM{\fl{s}} ) \, ds
  +
  \OM{t} .
\end{split}
\end{align}
Next note that the triangle inequality,
the fact that
$
  \forall \, 
  t \in [0,T]
$,
$ 
  M \in \N
  \colon
  \P(
    \YM{t} = \YMM{t}
  )
  = 
  1
$,
and Corollary~\ref{cor:a_priori_Lp}
(with $ \Yb{t} = \YMM{t} - \OM{t} $ for $ t \in (0,T], M \in \N $
in the notation of Corollary~\ref{cor:a_priori_Lp}) prove that
\begin{align}
\begin{split}
 &\sup_{ M \in \N }
  \sup_{ t \in [0,T] }
  \big\|
    \YM{t}
  \big\|_{ \L^{p}(\P;\BaMi) }
  =
  \sup_{ M \in \N }
  \sup_{ t \in [0,T] }
  \big\|
    \YMM{t}
  \big\|_{ \L^{p}(\P;\BaMi) }
\\&\leq
  \sup_{ M \in \N }
  \sup_{ t \in [0,T] }
  \big\|
    \OM{t}
  \big\|_{ \L^{p}(\P;\BaMi) }
  +
  \sup_{ M \in \N }
  \sup_{ t \in (0,T] }
  \big\|
    \YMM{t}
    -
    \OM{t}
  \big\|_{ \L^{p}(\P;\BaMi) }
\\&\leq
  \sup_{ M \in \N }
  \sup_{ t \in [0,T] }
  \big\|
    \OM{t}
  \big\|_{ \L^{p}(\P;\BaMi) }
  +
  \frac{c T^{(1-\sm)}}{(1-\sm)}
  \left[  
    \sup_{ M \in \N }
    \sup_{ 0 \leq s < t < T }
    \left(t-s\right)^{\sm}
    \big\|
      \SM_{\fl{s},t}
    \big\|_{ L(H,\BaMi) }
  \right]
\\&\cdot
  \left(
    1
    +
    \sup_{ M \in \N }
    \sup_{ t \in [0,T) }
    \big\|
      \Um( \OM{t} )
    \big\|_{ \L^{p\qm}(\P;\R) }^{\qm}
    +
    2^{[\qm-1]^+}
    e^{c\qm T}
    \left[
      \big|
	\Um(0)
      \big|^{\qm}
      +
      T^{\qm}
      \sup_{ M \in \N }
      \sup_{ t \in [0,T) }
      \big\|
	\Rm( \OM{t} )
      \big\|_{ \L^{p\qm}(\P;\R) }^{\qm}
    \right]
  \right)
  <
  \infty .
\end{split}
\end{align}
The proof of Corollary~\ref{cor:YM_bound} is thus completed.
\end{proof}

\begin{corollary}
\label{cor:F1_bound}
Assume the setting in Section~\ref{sec:main_result_setting},
let $ p \in [\max\{1,\nicefrac{1}{\qm}\}, \infty) $, and assume that
$
  \sup_{ M \in \N }
  \sup_{ t \in [0,T) }
  \|
    \Um( \OM{t} )
    +
    \Rm( \OM{t} )
  \|_{ \L^{p\qm}(\P;\R) }
  <
  \infty
$. Then 
$
  \sup_{M \in \N}
  \sup_{t\in[0,T)}
  \|
    F( \YM{t} )
  \|_{ \L^p(\P;H) }
  <
  \infty
$. 
\end{corollary}
\begin{proof}[Proof of Corollary~\ref{cor:F1_bound}]
Throughout this proof let
$ \YMM{} \colon [0,T] \times \Omega \rightarrow \BaMi $,
$ M \in \N $,
be the stochastic processes which satisfy
for all $ t \in [0,T] $, $ M \in \N $ that
\begin{align}
\begin{split}
 &\YMM{t}
  =
  \int_0^t
  \SM_{\fl{s},t} \,
  \one_{ 
    \{
      \| \YMM{\fl{s}} \|_{\BaMi}
      \leq
      (M/T)^{\cm}
    \}
  } \,
  F( \YMM{\fl{s}} ) \, ds
  +
  \OM{t} .
\end{split}
\end{align}
In the next step observe that
Lemma~\ref{lem:bootstrap_F} 
(with $ \Yb{t} = \YMM{t} - \OM{t} $ for $ t \in [0,T], M \in \N $
in the notation of Lemma~\ref{lem:bootstrap_F}) ensures 
that for all $ t \in [0,T] $, $ M \in \N $
it holds that
\begin{align}
\begin{split}
  \big\|
    F( \YMM{t} )
  \big\|_H
 &\leq
  c
  \left( 
    1 
    + 
    2^{[\qm-1]^+}
    e^{c \qm t}
    \left[ 
      \left|
	\Um( 0 )
      \right|^{\qm}
      +
      \left|
	\smallint_0^t
	\Rm( \OM{\fl{s}} ) \, ds
      \right|^{\qm}
    \right] 
    + 
    |\Um(\OM{t})|^{\qm}
  \right) .
\end{split}
\end{align}
This and
the fact that
$
  \forall \, 
  t \in [0,T]
$,
$ 
  M \in \N
  \colon
  \P(
    \YM{t} = \YMM{t}
  )
  = 
  1
$
imply that
\begin{align}
\label{eq:F_Lp_bound}
\begin{split}
 &\sup_{ M \in \N }
  \sup_{ t \in [0,T) }
  \big\|
    F( \YM{t} )
  \big\|_{ \L^{p}(\P;H) }
  =
  \sup_{ M \in \N }
  \sup_{ t \in [0,T) }
  \big\|
    F( \YMM{t} )
  \big\|_{ \L^{p}(\P;H) }
\\&\leq
  \sup_{ M \in \N }
  \sup_{ t \in [0,T) }
  \left\{ 
    c
    \left( 
      1
      + 
      2^{[\qm-1]^+}
      e^{c \qm t}
      \left[ 
	\left|
	  \Um( 0 )
	\right|^{\qm}
	+
	t^{\qm}
	\sup_{ s \in [0,t) }
	\big\|
	  \Rm( \OM{\fl{s}} )
	\big\|_{ \L^{p\qm}(\P;\R) }^{\qm}
      \right]
      +
      \left\|
	\Um( \OM{t} )
      \right\|_{ \L^{p\qm}(\P;\R) }^{\qm}
    \right)
  \right\}
\\&\leq
  c
  \left( 
    1 
    + 
    2^{[\qm-1]^+}
    e^{c \qm T}
    \left[ 
      \left|
	\Um( 0 )
      \right|^{\qm}
      +
      T^{\qm}
      \sup_{ M \in \N }
      \sup_{ t \in [0,T) }
      \big\|
	\Rm( \OM{\fl{t}} )
      \big\|_{ \L^{p\qm}(\P;\R) }^{\qm}
    \right]
    +
    \sup_{ M \in \N }
    \sup_{ t \in [0,T) }
    \left\|
      \Um( \OM{t} )
    \right\|_{ \L^{p\qm}(\P;\R) }^{\qm}
  \right)
\\&<
  \infty .
\end{split}
\end{align}
The proof of Corollary~\ref{cor:F1_bound} is thus completed.
\end{proof}

\subsection{Main result}
\label{sec:main_result_theorem}

\begin{theorem}
\label{thm:main}
Assume the setting in 
Section~\ref{sec:main_result_setting},
let $ \nmm \in (0,\infty) $,
$ \nm \in (0,\nicefrac{1}{2}) $,
$ \sm \in [0,1) $,
$ \smm \in [0, 1-\sm) $,
$ 
  p 
  \in 
  [
    \max\{
      2,
      \frac{1}{\qmm},
      \frac{1}{2\qm}
      \min\{
	1,
	\frac{2}{\qmm},
	\frac{1}{2\nmm},
	\frac{1}{\nmm\qmm}
      \}
    \}
    ,\infty
  ) 
$,
and assume that 
\begin{align}
\label{eq:main_assumption}
\begin{split}
 &\sup_{ M \in \N }
  \sup_{ 0 \leq s < t \leq T }
  \left[
    t^{\sm}
    \big\|
      e^{tA}
    \big\|_{ L(H,\BaMi) }
    +
    \left(t-s\right)^{\sm}
    \big\|
      \SM_{\fl{s},t}
    \big\|_{ L(H,\BaMi) }
    +
    M^{\smm}
    \int_0^t
    \big\|
      e^{(t-u)A}
      -
      \SM_{\fl{u},t}
    \big\|_{ L(H,\BaMi) } \, du
  \right]
\\&+ 
  \sup_{ M \in \N }
  \sup_{ t \in [0,T] }
  \left[
    \left\|
      \Um( \OM{t} )
      +
      \Rm( \OM{t} )
    \right\|_{ \L^{2p\qm\max\{2,\qmm, 2\nmm, \nmm \qmm\}}\!(\P;\R) }
    +
    \big|M \fl{t}\big|^{\nm}
    \big\|
      O_t
      -
      O_{\fl{t}}
    \big\|_{ \L^{2p}(\P;\BaMi) }
  \right]
\\&+ 
  \sup_{ M \in \N }
  \sup_{ t \in [0,T] }
  \left\|
    \|  O_{t} \|_{ \BaMi }
    +
    M^{\nm}
    \|
      O_{t}
      -
      \OM{t}
    \|_{ \BaMi }
  \right\|_{ \L^{p\max\{2,\qmm, 4 \nmm,2\nmm \qmm\}}\!(\P;\R) }
  <
  \infty .
\end{split}
\end{align}
Then
\begin{align}
\label{eq:main_assert}
  \sup_{ M \in \N }
  \sup_{ t \in [0,T] }
  \left[
    \big\|
      \YM{t}
    \big\|_{ \L^{p\max\{2,\qmm, 4\nmm,2\nmm\qmm\}}\!(\P;\BaMi) }
    +
    \big\|
      X_t
    \big\|_{ \L^p(\P;H) }
    +
    M^{\min\{\nmm\cm, \smm, \nm \}}
    \big\|
      X_t - \YM{t}
    \big\|_{ \L^p(\P;H) }
  \right]
  <
  \infty .
\end{align}
\end{theorem}
\begin{proof}[Proof of Theorem~\ref{thm:main}]
Throughout this proof assume w.l.o.g.\ that
$ H \neq \{0\} $, 
let
$ \YMM{} \colon [0,T] \times \Omega \rightarrow \BaMi $,
$ M \in \N $,
be the stochastic processes which satisfy
for all $ t \in [0,T] $, $ M \in \N $ that
$
  \YMM{t}
  =
  \int_0^t
  \SM_{\fl{s},t} \,
  \one_{ 
    \{
      \| \YMM{\fl{s}} \|_{\BaMi}
      \leq
      (M/T)^{\cm}
    \}
  } \,
  F( \YMM{\fl{s}} ) \, ds
  +
  \OM{t}
$,
let
$ \tilde{\Omega} \subseteq \Omega $
be the set given by
$
  \tilde{\Omega}
  =
  \big\{
    \omega \in \Omega
    \colon
    \big(
      \forall \,
      t \in [0,T]
      \colon
      X_t(\omega)
      =
      \smallint_0^t
      e^{(t-s)A}
      F(X_s(\omega)) \, ds
      +
      O_t(\omega)
    \big)
  \big\}
  \cap 
  \big\{
    \omega \in \Omega
    \colon
    \limsup_{ r \searrow 0 }
    \sup_{ 0 \leq s < u \leq T }
    \frac{ 
      s \| O_u(\omega) - O_s(\omega) \|_V
    }{
      (u-s)^r
    }
    <
    \infty
  \big\}
$,
let $ \tilde{O}, \tilde{X} \colon [0,T] \times \Omega \rightarrow V $
be the stochastic processes which satisfy
for all $ t \in [0,T] $, $ \omega \in \tilde{\Omega} $ that
$ \tilde{O}_t(\omega) = O_t(\omega) $,
$ \tilde{X}_t(\omega) = X_t(\omega) $
and which satisfy for all $ t \in [0,T] $, 
$ \omega \in \Omega \backslash \tilde{\Omega} $ that
$ \tilde{O}_t(\omega) = -\int_0^t e^{(t-s)A} F(0) \, ds $,
$ \tilde{X}_t(\omega) = 0 $,
let $ \kappa \in (0,1-\sm) $
be a real number, and let
$ K \in [0,\infty) $
be the real number with the property that 
$ 
  K 
  = 
  \| F(0) \|_H 
  + 
  \sup_{t \in (0,T]}
  \big[
    t^{\sm}
    \| e^{tA} \|_{ L(H,\BaMi) }
    +
    \| (-tA)^{\kappa} e^{tA} \|_{L(H)}
    +
    \| (-tA)^{-\kappa} (e^{tA}-\Id_H) \|_{L(H)}
  \big]
$.
The assumption that 
$
  \sup_{t \in (0,T]}
  t^{\sm}
  \| e^{tA} \|_{ L(H,\BaMi) }
  <
  \infty
$
and the assumption that
$ A \colon D(A) \subseteq H \rightarrow H $
is a generator of an analytic semigroup 
with 
$ 
  \textup{spectrum}(A) 
  \subseteq 
  \{
    z \in \mathbb{C} 
    \colon 
    \textup{Re}(z) < 0
  \}
$
ensure that such a real number $ K $ does
indeed exist.
Next note that for all
$ s,t \in [0,T] $,
$ \omega \in \Omega \backslash \tilde{\Omega} $
with $ s < t $ it holds that
\begin{align}
\begin{split}
 &\big\|
    \tilde{O}_t(\omega)
    -
    \tilde{O}_s(\omega)
  \big\|_{\BaMi}
  =
  \left\|
    \int_0^t
    e^{(t-u)A}
    F(0) \, du
    -
    \int_0^s
    e^{(s-u)A}
    F(0) \, du
  \right\|_{\BaMi}
\\&\leq
  \left\|
    \int_s^t
    e^{(t-u)A}
    F(0) \, du
  \right\|_{\BaMi}
  +
  \left\|
    \int_0^s
    \left(
      e^{(t-u)A}
      -
      e^{(s-u)A}
    \right)
    F(0) \, du
  \right\|_{\BaMi}
\\&\leq
  \int_s^t
  \left\|
    e^{(t-u)A}
  \right\|_{ L(H,\BaMi) }
  \left\|
    F(0)
  \right\|_H du
  +
  \int_0^s
  \big\|
    e^{\frac{1}{2}(s-u)A}
  \big\|_{ L(H,\BaMi) }
  \big\|
    e^{\frac{1}{2}(s-u)A}
    \left(
      e^{(t-s)}
      -
      \Id_H
    \right)
    F(0)
  \big\|_H \, du
\\&\leq
  \left[
    \sup_{ u \in (0,t-s) }
    u^{\sm}
    \left\|
      e^{uA}
    \right\|_{ L(H,\BaMi) }
  \right] 
  \int_s^t
  \left(
    t-u
  \right)^{-\sm}
  \left\|
    F(0)
  \right\|_H du
  +
  2^{\sm}
  \left[
    \sup_{ u \in (0,s) }
    u^{\sm}
    \left\|
      e^{uA}
    \right\|_{ L(H,\BaMi) }
  \right] 
\\&\quad\cdot
  \int_0^s
  \left( s-u \right)^{-\sm}
  \big\| \!
    \left(-A\right)^{\kappa}
    e^{\frac{1}{2}(s-u)A}
  \big\|_{L(H)}
  \big\| \!
    \left(-A\right)^{-\kappa}
    \left(
      e^{(t-s)A}
      -
      \Id_H
    \right) \!
  \big\|_{L(H)}
  \big\|
    F(0)
  \big\|_H \, du .
\end{split}
\end{align}
This proves
that for all
$ s,t \in [0,T] $,
$ \omega \in \Omega \backslash \tilde{\Omega} $
with $ s < t $ it holds that
\begin{align}
\label{eq:main_cont_1}
\begin{split}
 &\big\|
    \tilde{O}_t(\omega)
    -
    \tilde{O}_s(\omega)
  \big\|_{\BaMi}
\\&\leq
  \frac{ 
    K^2 
    \left(t-s\right)^{(1-\sm)}
  }{ 
    \left(1-\sm\right)
  }
  +
  2^{(\sm+\kappa)}
  K^2
  \left( t-s \right)^{\kappa}
  \left[
    \sup_{ u \in (0,s) }
    \left\|
      \left(-uA\right)^{\kappa}
      e^{uA}
    \right\|_{ L(H) }
  \right]
\\&\quad\cdot
  \left[
    \sup_{ u \in (0,T] }
    \left\|
      \left(-uA\right)^{-\kappa}
      \left(
        e^{uA}
        -
        \Id_H
      \right)
    \right\|_{ L(H) }
  \right]
  \int_0^s
  \left( s-u \right)^{-(\sm+\kappa)} du
\\&\leq
  \frac{ 
    K^2 
    \left(t-s\right)^{(1-\sm)}
  }{ 
    \left(1-\sm\right)
  }
  +
  \frac{
    2^{(\sm+\kappa)}
    K^4 
    \left( t-s \right)^{\kappa}
    s^{(1-\sm-\kappa)}
  }{ 
    \left(1-\sm-\kappa\right)
  }
  \leq
  \left( t-s \right)^{\kappa}
  T^{(1-\sm-\kappa)}
  \left[
    \frac{K^2}{\left(1-\sm\right)}
    +
    \frac{
      2^{(\sm+\kappa)}
      K^4
    }{\left(1-\sm-\kappa\right)}
  \right].
\end{split}
\end{align}
Hence, we obtain
that for all
$ \omega \in \Omega \backslash \tilde{\Omega} $,
$ r \in [0,\kappa] $ it holds that
\begin{align}
\begin{split}
  \sup_{ 0 \leq s < t \leq T }
  \frac{ 
    s
    \big\|
      \tilde{O}_t(\omega)
      -
      \tilde{O}_s(\omega)
    \big\|_{\BaMi}
  }{ 
    \left(t-s\right)^r
  }
 &\leq
  \sup_{ 0 \leq s < t \leq T }
  \left(
    s
    \left( t-s \right)^{(\kappa-r)}
    T^{(1-\sm-\kappa)}
    \left[
      \frac{K^2}{\left(1-\sm\right)}
      +
      \frac{
	2^{(\sm+\kappa)}
	K^4
      }{\left(1-\sm-\kappa\right)}
    \right]
  \right)
\\&\leq
  T^{(2-\sm-r)}
  \left[
    \frac{K^2}{\left(1-\sm\right)}
    +
    \frac{
      2^{(\sm+\kappa)}
      K^4
    }{\left(1-\sm-\kappa\right)}
  \right]
  <
  \infty .
\end{split}
\end{align}
Combining this with the
fact that
$
  \forall \,
  \omega \in \tilde{\Omega}
  \colon
  \limsup_{ r \searrow 0 }
  \sup_{ 0 \leq s < t \leq T }
  \frac{ 
    s \| \tilde{O}_t(\omega) - \tilde{O}_s(\omega) \|_V
  }{
    (t-s)^r
  }
  <
  \infty
$
ensures
that for all $ \omega \in \Omega $ it
holds that
\begin{align}
\label{eq:main_O_hoelder_like}
\begin{split}
  \limsup_{ r \searrow 0 }
  \sup_{ 0 \leq s < t \leq T }
  \frac{ 
    s \| \tilde{O}_t(\omega) - \tilde{O}_s(\omega) \|_V
  }{
    (t-s)^r
  }
  <
  \infty .
\end{split}
\end{align}
In the next step note 
that for all $ t \in [0,T] $
it holds that
$ \P( \tilde{O}_t = O_t ) \geq \P( \tilde{\Omega} ) = 1 $.
This and~\eqref{eq:main_assumption}
show that
\begin{align}
\label{eq:main_O_assumption}
\begin{split}
 &\sup_{ M \in \N }
  \sup_{ t \in [0,T] }
  \left[
    \big|M \fl{t}\big|^{\nm} 
    \big\|
      \tilde{O}_t
      -
      \tilde{O}_{\fl{t}}
    \big\|_{ \L^{2p}(\P;\BaMi) }
    +
    \big\|
      \|  \tilde{O}_{t} \|_{ \BaMi }
      +
      M^{\nm} 
      \|
	\tilde{O}_{t}
	-
	\OM{t}
      \|_{ \BaMi }
    \big\|_{ \L^{p\max\{2,\qmm, 4 \nmm,2\nmm \qmm\}}\!(\P;\R) } 
  \right]
\\&<
  \infty .
\end{split}
\end{align}
The triangle inequality 
hence implies that
\begin{align}
\label{eq:main_O}
\begin{split}
 &\sup_{ M \in \N }
  \sup_{ t \in [0,T] }
  \big\|
    \OM{t}
  \big\|_{ \L^{p\max\{2,\qmm, 4 \nmm,2\nmm \qmm\}}\!(\P;\BaMi) }
  \leq
  \sup_{ M \in \N }
  \sup_{ t \in [0,T] }
  \big\|
    \|
      \tilde{O}_{t}
      -
      \OM{t}
    \|_{ \BaMi }
    +
    \|  \tilde{O}_{t} \|_{ \BaMi }
  \big\|_{ \L^{p\max\{2,\qmm, 4 \nmm,2\nmm \qmm\}}\!(\P;\R) }
\\&\leq
  \sup_{ M \in \N }
  \sup_{ t \in [0,T] }
  \big\|
    \|  \tilde{O}_{t} \|_{ \BaMi }
    +
    M^{\nm}
    \|
      \tilde{O}_{t}
      -
      \OM{t}
    \|_{ \BaMi }
  \big\|_{ \L^{p\max\{2,\qmm, 4 \nmm,2\nmm \qmm\}}\!(\P;\R) }
  <
  \infty .
\end{split}
\end{align}
In addition,
observe that the assumption that
$
  p
  \geq
  \max\{
    2,
    \frac{1}{\qmm},
    \frac{1}{2\qm}
    \min\{
      1,
      \frac{2}{\qmm},
      \frac{1}{2\nmm},
      \frac{1}{\nmm\qmm}
    \}
  \}
$
ensures that
$
  p\max\{2,\qmm, 4 \nmm,2\nmm \qmm\}
  \geq
  \nicefrac{1}{\qm}
$.
Combining~\eqref{eq:main_assumption},
\eqref{eq:main_O},
and the fact that
$
  \forall \,
  t \in [0,T]
$,
$ 
  M \in \N
  \colon
  \P( 
    \YMM{t}
    =
    \YM{t}
  )
  =
  1
$
with Corollary~\ref{cor:YM_bound}
hence proves that
\begin{align}
\label{eq:main_bound}
\begin{split}
  \sup_{ M \in \N }
  \sup_{ t \in [0,T] }
  \big\|
    \YMM{t}
  \big\|_{ \L^{p\max\{2,\qmm, 4 \nmm,2\nmm \qmm\}}\!(\P;\BaMi) }
  =
  \sup_{ M \in \N }
  \sup_{ t \in [0,T] }
  \big\|
    \YM{t}
  \big\|_{ \L^{p\max\{2,\qmm, 4 \nmm,2\nmm \qmm\}}\!(\P;\BaMi) }
  <
  \infty.
\end{split}
\end{align}
In the next step we combine
\eqref{eq:main_assumption},
the fact that
$
  2p\max\{2,\qmm, 2\nmm, \nmm \qmm\}
  \geq
  p\max\{2,\qmm, 4 \nmm,2\nmm \qmm\}
  \geq
  \nicefrac{1}{\qm}
$,
and the fact that
$
  \forall \,
  t \in [0,T]
$,
$ 
  M \in \N
  \colon
  \P( 
    \YMM{t}
    =
    \YM{t}
  )
  =
  1
$
with Corollary~\ref{cor:F1_bound}
to obtain that 
\begin{align}
\begin{split}
  \sup_{ M \in \N }
  \sup_{ t \in [0,T) }
  \big\|
    F( \YMM{t} )
  \big\|_{ \L^{2p\max\{2,\qmm, 2\nmm, \nmm \qmm\}}\!(\P;H) }
  =
  \sup_{ M \in \N }
  \sup_{ t \in [0,T) }
  \big\|
    F( \YM{t} )
  \big\|_{ \L^{2p\max\{2,\qmm, 2\nmm, \nmm \qmm\}}\!(\P;H) }
  <
  \infty .
\end{split}
\end{align}
This and \eqref{eq:main_bound} imply that
\begin{align}
\begin{split}
 &\sup_{ M \in \N }
  \sup_{ t \in [0,T] } 
  \big\|
    \YMM{t}
  \big\|_{ \L^{p\max\{1,\qmm\}}\!(\P;\BaMi) }
  +
  \sup_{ M \in \N }
  \sup_{ t \in [0,T) }
  \left\|
    \big\|
      \YMM{t}
    \big\|_{ \BaMi }^{ 2 \nmm }
    +
    \big\|
      F( \YMM{t} )
    \big\|_H^2
  \right\|_{ \L^{p\max\{2,\qmm\}}\!(\P;\R) }
  <
  \infty .
\end{split}
\end{align}
Combining this,
\eqref{eq:main_assumption},
\eqref{eq:main_O_hoelder_like},
and~\eqref{eq:main_O_assumption}
with Proposition~\ref{prop:X_YM_Lp_diff}
(with $ C = c $, $ L = c $, $ \alpha = \nmm\cm $, $ \Vm(v) = \| v \|_{\BaMi}^{\nicefrac{1}{\cm}} $,
$ X_t = \tilde{X}_t $,
$ O_t = \tilde{O}_t $,
$ \XO{t} = \tilde{X}_t - \tilde{O}_t $,
$ \OM{t} = \OM{t} $,
$ \YM{t} = \YMM{t} $,
$ \YOM{t} = \YMM{t} - \OM{t} $,
and 
$ 
  \YXM{t} 
  = 
  \int_0^t 
  e^{(t-s)A} 
  F( \YMM{\fl{s}} ) \, ds
$
for $ v \in \BaMi $, $ t \in [0,T] $, $ M \in \N $ in the 
notation of Proposition~\ref{prop:X_YM_Lp_diff})
ensures that
$
  \sup_{ M \in \N }
  \sup_{ t \in [0,T] }
  \big[
    \|
      \tilde{X}_t
    \|_{ \L^p(\P;H) }
    +
    M^{\min\{\nmm\cm, \smm, \nm \}}
    \|
      \tilde{X}_t - \YMM{t}
    \|_{ \L^p(\P;H) }
  \big]
  <
  \infty
$.
This, 
the fact that 
$
  \forall \,
  t \in [0,T] 
  \colon
  \P(
    X_t
    =
    \tilde{X}_t
  )
  \geq
  \P( \tilde{\Omega} )
  =
  1
$,
the fact that
$
  \forall \,
  t \in [0,T]
$,
$
  M \in \N
  \colon
  \P(
    \YM{t}
    =
    \YMM{t}
  )
  =
  1
$,
and~\eqref{eq:main_bound}
establish~\eqref{eq:main_assert}.
The proof of Theorem~\ref{thm:main}
is thus completed.
\end{proof}

\section{Examples}
\label{sec:examples}

\subsection{Setting}
\label{sec:example_setting}
Let 
$ T, \nu \in (0, \infty) $,
$ d \in \N $,
$ n \in \{ 1, 3, 5, \ldots \} $,
$ a_0,a_1,\ldots,a_{n-1}, b_1 \in \R $,
$ a_n \in (-\infty, 0) $,
$ b_2 \in (b_1, \infty) $,
$ \D = (b_1,b_2)^d \subseteq \mathbb{R}^d $,
$ (H, \left<\cdot,\cdot\right>_H, \left\| \cdot \right\|_H )
= ( L^2(\lambda_{\D};\R), \left< \cdot, \cdot \right>_{L^2(\lambda_{\D};\R)},$ 
$ \left\| \cdot \right\|_{ L^2(\lambda_{\D};\R) } ) $, 
$ ( \BaMi, \left\| \cdot \right\|_{ \BaMi } )
= ( L^{2n^2}(\lambda_{\D}; \R),$ $ \left\| \cdot \right\|_{ L^{2n^2}(\lambda_{\D};\R) } ) $,
$ F \in \C( \BaMi, H ) $,
$ \Ff \in \C ( \L^{2n^2}(\lambda_{D}; \R) , \L^{2}(\lambda_{D}; \R) ) $
satisfy for all $ v \in \BaMi $, 
$ w \in \L^{2n^2}(\lambda_{D}; \R) $
that 
$ 
  F(v) 
  = 
  \sum_{k=0}^n 
  a_k v^k
$
and
$ 
  \Ff(w) 
  = 
  \sum_{k=0}^n 
  a_k w^k
$,
let $ A \colon D(A) \subseteq H \rightarrow H $
be the Laplacian with Dirichlet boundary conditions on $ H $
times the real number $ \nu $,
let $ ( H_r, \left< \cdot, \cdot \right>_{ H_r },$ $ \left\| \cdot \right\|_{ H_r } ) $, 
$ r \in \R $, be a family of interpolation spaces associated to $ -A $
(see, e.g., Definition~3.5.26 in~\cite{Jentzen2015SPDElecturenotes}),
let $ (\Omega, \F, \P ) $ 
be a probability space
with a normal filtration
$ (\F_t)_{t\in[0,T]} $,
and let
$ 
  (\underline{\cdot})
  \colon
  \big\{
    [v]_{\lambda_{D},\B(\R)} \in H
    \colon
    \big( v \colon \D \rightarrow \R
    \text{ is a uniformly continuous function} \big)
  \big\}
  \rightarrow
  \C( \D, \R)
$
be the function with the property that
for all uniformly continuous functions
$ v \colon \D \rightarrow \R $ it holds
that $ \underline{[v]_{\lambda_{D}, \B(\R)}} = v $.

\subsection{Auxiliary lemmas}
\label{sec:example_aux}

In this subsection we establish
a few elementary and partially well-known
auxiliary lemmas that we need to apply
Theorem~\ref{thm:main} above
(see Corollary~\ref{cor:exp_euler_convergence}
and Corollary~\ref{cor:li_euler_convergence}
below for details).

\begin{lemma}
\label{lem:U_lem1}
Assume the setting in Section~\ref{sec:example_setting}
and let $ q \in \{2,4,6,\ldots\} $,
$ v \in L^q(\lambda_{\D};\R) $, $ t \in [0,T] $.
Then
$
  \max\!\big\{ 
    \|
      e^{tA} v
    \|_{ L^q(\lambda_{\D};\R) },
    \|
      ( \Id_H - tA)^{-1}
      v
    \|_{ L^q(\lambda_{\D};\R) }
  \big\}
  \leq
  \|
    v
  \|_{ L^q(\lambda_{\D};\R) }
$.
\end{lemma}
\begin{proof}[Proof of Lemma~\ref{lem:U_lem1}]
Throughout this proof assume w.l.o.g.\ that $ t \in (0,T] $ (otherwise the proof is clear).
Note that the fundamental theorem of calculus implies
for all $ \varepsilon \in (0,t) $ that
\begin{align}
\label{eq:semi_contraction_h1}
\begin{split}
 &\big\|
    e^{tA} v
  \big\|_{ L^q(\lambda_{\D};\R) }^q
  =
  \big\|
    e^{\varepsilon A}
    v
  \big\|_{ L^q(\lambda_{\D};\R) }^q
  +
  \int_{\varepsilon}^t
  \left(
    \tfrac{\partial}{\partial s}
    \big\|
      e^{sA} v
    \big\|_{ L^q(\lambda_{\D};\R) }^q
  \right) ds .
\end{split}
\end{align}
Furthermore, integration by parts
shows that for all 
$ s \in (0,t) $
it holds that
\begin{align}
\label{eq:by_parts}
\begin{split}
 &\tfrac{\partial}{\partial s}
  \big\|
    e^{sA} v
  \big\|_{ L^q(\lambda_{\D};\R) }^q
  =
  \int_{\D}
  \tfrac{\partial}{\partial s}
  \left[
    \left(
      \underline{e^{sA} v}
    \right)\!(x)
  \right]^q \lambda_{\R^d}(dx)
  =
  q
  \int_{\D}
  \left[
    \left(
      \underline{e^{sA} v}
    \right)\!(x)
  \right]^{(q-1)}
  \left[
    \tfrac{\partial}{\partial s}
    \left(
      \underline{e^{sA} v}
    \right)\!(x)
  \right] \lambda_{\R^d}(dx)
\\&=
  \nu
  q
  \sum_{ k=1 }^d
  \int_{\D}
  \left[
    \left(
      \underline{e^{sA} v}
    \right)\!(x_1,\ldots,x_d)
  \right]^{(q-1)}
  \left[ 
    \tfrac{\partial^2}{\partial x_k^2}
    \left(
      \underline{e^{sA} v}
    \right)\!(x_1,\ldots,x_d) 
  \right] \lambda_{\R^d}(dx_1, \ldots, dx_d)
\\&=
  -
  \nu
  q \,
  (q-1)
  \sum_{ k=1 }^d
  \int_{\D}
  \left[
    \left(
      \underline{e^{sA} v}
    \right)\!(x_1,\ldots,x_d)
  \right]^{(q-2)}
  \left[ 
    \tfrac{\partial}{\partial x_k}
    \left(
      \underline{e^{sA} v}
    \right)\!(x_1,\ldots,x_d) 
  \right]^2 \lambda_{\R^d}(dx_1, \ldots, dx_d)
  \leq
  0 .
\end{split}
\end{align}
Combining~\eqref{eq:semi_contraction_h1}
and~\eqref{eq:by_parts}
proves for all $ \varepsilon \in (0,t) $ that
$
  \|
    e^{tA} v
  \|_{ L^q(\lambda_{\D};\R) }
  \leq
  \|
    e^{\varepsilon A}
    v
  \|_{ L^q(\lambda_{\D};\R) }
$.
Therefore, we obtain that
$
  \|
    e^{tA} v
  \|_{ L^q(\lambda_{\D};\R) }
  \leq
  \|
    v
  \|_{ L^q(\lambda_{\D};\R) }
$.
This and the Hille-Yosida theorem
(see, e.g., Theorem~1.3.1 in Pazy \cite{p83}) imply
that
$
  \|
    ( \Id_H - tA)^{-1}
    v
  \|_{ L^q(\lambda_{\D};\R) }
  \leq
  \|
    v
  \|_{ L^q(\lambda_{\D};\R) }
$.
The proof of Lemma~\ref{lem:U_lem1} is thus completed.
\end{proof}

\begin{lemma}
\label{lem:U_lem2}
Assume the setting in Section~\ref{sec:example_setting},
let $ \eta \colon \N_0 \times \N_0 \rightarrow (0,\infty) $ be 
a function,
and let $ q \in \{2n, 2n+2, 2n+4, \ldots \} $, $ h \in (0,T] $,
$ t \in [0,h] $, $ u, v \in \L^{nq}(\lambda_{\D}; \R) $,
$ \chi \in [0, \frac{1}{2n}] $.
Then
\begin{align}
\begin{split} 
\label{eq:U_lem2_assert}
 &\big\|
    u
    +
    t \,
    \one_{ 
      [0, h^{-\chi}]
    }
    ( \left\| u + v \right\|_{ \L^{nq}(\lambda_{\D};\R) } ) \,
    \Ff( u + v ) 
  \big\|_{ \L^q(\lambda_{\D};\R) }^q
\\&\leq
  q t
  \left(
    a_n
    +
    \textstyle\sum_{ k = 0 }^{ n }
    \textstyle\sum_{ j = 0 }^{ \min\{k,n-1\} }
    \tbinom{k}{j} \,
    | a_k | \,
    |\eta_{(k,j)}|^{\frac{(q+n-1)}{(q+j-1)}}
  \right)
  \| u \|_{\L^{(q+n-1)}(\lambda_{\D};\R)}^{(q+n-1)}
\\&\quad+
  e^{t} \,
  \|
    u
  \|_{ \L^q(\lambda_{\D};\R) }^q
  + 
  t
  \max\!\left\{ 1, \lambda_{\R^d}(\D) \right\}
  \max\!\left\{
    1,
    \| v \|_{\L^{(q+n-1)}(\lambda_{\D};\R)}^{(q+n-1)}
  \right\}
\\&\quad\cdot
  \left(
    2
    q
    \textstyle\sum_{ k = 0 }^{ n }
    \textstyle\sum_{ j = 0 }^{ \min\{k,n-1\} }
    \tbinom{k}{j} \,
    | a_k | \,
    | \eta_{(k,j)} |^{-\frac{(q+n-1)}{(n-j)}}
    +
    \left[ 
      q
      \left(n+1\right)
      \max\{1,T\}
      \max_{ j \in \{0,1,\ldots,n\} }
      |a_j|
    \right]^q
  \right) .
\end{split}
\end{align}
\end{lemma}
\begin{proof}[Proof of Lemma~\ref{lem:U_lem2}]
First of all, note that
\begin{align}
\label{eq:U_lem2_1}
\begin{split}
 &\big\|
    u
    +
    t \,
    \one_{ 
      [0, h^{-\chi}]
    }
    ( \left\| u + v \right\|_{ \L^{nq}(\lambda_{\D};\R) } ) \,
    \Ff( u + v ) 
  \big\|_{ \L^q(\lambda_{\D};\R) }^q
\\&=
  \int_{\D}
  \left|
    u(x)
    +
    t \,
    \one_{ 
      [0, h^{-\chi}]
    }
    ( \| u+v \|_{ \L^{nq}(\lambda_{\D};\R)} ) \,
    \left( 
      \textstyle\sum_{ k=0 }^{ n }
      a_k
      \left[
        u(x)
        +
        v(x)
      \right]^k
    \right)
  \right|^q \lambda_{\R^d}(dx)
\\&=
  \left\| u \right\|_{ \L^q(\lambda_{\D};\R) }^q
  +
  q t \,
  \one_{ 
    [0, h^{-\chi}]
  }
  ( \| u+v \|_{ \L^{nq}(\lambda_{\D};\R)} ) \,
  \int_{\D}
  \left[ u(x) \right]^{(q-1)}
  \Big( 
    \textstyle\sum_{ k=0 }^{ n }
    a_k
    \left[ u(x) + v(x) \right]^k
  \Big) \, \lambda_{\R^d}(dx)
\\&\quad+
  \sum_{ l=2 }^q
  \binom{q}{l} \,
  t^l \,
  \one_{ 
    [0, h^{-\chi}]
  }
  ( \| u + v \|_{ \L^{nq}(\lambda_{\D};\R) } )
  \int_{ \D }
  \left[ u(x) \right]^{(q-l)}
  \Big[
    \textstyle\sum_{ k=0 }^n
    a_k
    \left[ u(x) + v(x) \right]^k
  \Big]^l \, \lambda_{\R^d}(dx) .
\end{split}
\end{align}
Furthermore, observe that
Young's inequality assures that
\begin{align}
\begin{split}
 &q t \,
  \one_{ 
    [0, h^{-\chi}]
  }
  ( \| u+v \|_{\L^{nq}(\lambda_{\D};\R)} ) \,
  \int_{\D}
  [ u(x) ]^{(q-1)}
  \Big( 
    \textstyle\sum_{ k=0 }^{ n }
    a_k
    \left[ u(x) + v(x) \right]^k
  \Big) \, \lambda_{\R^d}(dx)
\\&=
  q t \,
  \one_{ 
    [0, h^{-\chi}]
  }
  ( \| u+v \|_{\L^{nq}(\lambda_{\D};\R)} ) 
  \sum_{ k = 0 }^{ n }
  \sum_{ j = 0 }^{ k }
  \binom{k}{j} \,
  a_k
  \int_{\D}
  \left[ u(x) \right]^{(q+j-1)}
  \left[ v(x) \right]^{(k-j)} \lambda_{\R^d}(dx)
\\&\leq
  q t
  a_n
  \left\| u \right\|_{\L^{(q+n-1)}(\lambda_{\D};\R)}^{(q+n-1)}
  +
  q t
  \sum_{ k = 0 }^{ n }
  \sum_{ j = 0 }^{ \min\{k,n-1\} }
  \binom{k}{j} \,
  | a_k |
  \int_{\D}
  \eta_{(k,j)}
  \left| u(x) \right|^{(q+j-1)}
  \frac{ 
    | v(x) |^{(k-j)}
  }{
    \eta_{(k,j)}
  } \, \lambda_{\R^d}(dx)
\\&\leq
  q t
  a_n
  \left\| u \right\|_{\L^{(q+n-1)}(\lambda_{\D};\R)}^{(q+n-1)}
\\&+
  q t
  \sum_{ k = 0 }^{ n }
  \sum_{ j = 0 }^{ \min\{k,n-1\} } \!
  \binom{k}{j} \,
  | a_k |
  \int_{\D}
  \left[ 
    | \eta_{(k,j)} |^{\frac{(q+n-1)}{(q+j-1)}}
    \left| u(x) \right|^{(q+n-1)}
    +
    \left| \eta_{(k,j)} \right|^{-\frac{(q+n-1)}{(n-j)}}
    | v(x) |^{ \frac{(k-j)(q+n-1)}{(n-j)} } 
  \right] \lambda_{\R^d}(dx) .
\end{split}
\end{align}
This 
implies that
\begin{align}
\label{eq:U_lem2_2}
\begin{split}
 &q t \,
  \one_{ 
    [0, h^{-\chi}]
  }
  ( \| u+v \|_{\L^{nq}(\lambda_{\D};\R)} ) \,
  \int_{\D}
  [ u(x) ]^{(q-1)}
  \Big( 
    \textstyle\sum_{ k=0 }^{ n }
    a_k
    \left[ u(x) + v(x) \right]^k
  \Big) \, \lambda_{\R^d}(dx)
\\&\leq
  q t
  \left[
    a_n
    +
    \sum_{ k = 0 }^{ n }
    \sum_{ j = 0 }^{ \min\{k,n-1\} }
    \binom{k}{j} \,
    | a_k | \,
    | \eta_{(k,j)} |^{\frac{(q+n-1)}{(q+j-1)}}
  \right]
  \left\| u \right\|_{\L^{(q+n-1)}(\lambda_{\D};\R)}^{(q+n-1)}
\\&\quad+
  q t
  \sum_{ k = 0 }^{ n }
  \sum_{ j = 0 }^{ \min\{k,n-1\} }
  \binom{k}{j} \,
  | a_k | \,
  | \eta_{(k,j)} |^{-\frac{(q+n-1)}{(n-j)}}
  \int_{\D}
  \max\!\left\{ 1, \left| v(x) \right|^{ (q+n-1) } \right\}
  \lambda_{\R^d}(dx) 
\\&\leq
  q t
  \left[
    a_n
    +
    \sum_{ k = 0 }^{ n }
    \sum_{ j = 0 }^{ \min\{k,n-1\} }
    \binom{k}{j} \,
    | a_k | \,
    | \eta_{(k,j)} |^{\frac{(q+n-1)}{(q+j-1)}}
  \right]
  \left\| u \right\|_{\L^{(q+n-1)}(\lambda_{\D};\R)}^{(q+n-1)}
\\&\quad+
  q t
  \sum_{ k = 0 }^{ n }
  \sum_{ j = 0 }^{ \min\{k,n-1\} }
  \binom{k}{j} \,
  | a_k | \,
  | \eta_{(k,j)} |^{-\frac{(q+n-1)}{(n-j)}}
  \left( 
    \lambda_{\R^d}(\D)
    +
    \left\| v \right\|_{\L^{(q+n-1)}(\lambda_{\D};\R)}^{(q+n-1)}
  \right) .
\end{split}
\end{align}
Moreover, note that
the fact that
$ \forall \, m \in \N $, 
$ (z_l)_{ l\in\{0,1,\ldots,m\} } \subseteq \R $,
$ r \in [0,\infty) 
  \colon 
  | z_0 + z_1 + \ldots + z_m|^r
  \leq
  (m+1)^{[r-1]^+}
  \left(
    |z_0|^r
    +
    |z_1|^r
    +
    \ldots
    +
    |z_m|^r  
  \right)
$ 
and
H{\"o}lder's inequality 
imply that
\begin{align}
\begin{split}
 &\sum_{ l=2 }^q
  \binom{q}{l} \,
  t^l \,
  \one_{ 
    [0, h^{-\chi}]
  }
  ( \| u + v \|_{ \L^{nq}(\lambda_{\D};\R) } )
  \int_{ \D }
  \left[ u(x) \right]^{(q-l)}
  \Big[
    \textstyle\sum_{ k=0 }^n
    a_k
    \left[ u(x) + v(x) \right]^k
  \Big]^l \, \lambda_{\R^d}(dx)
\\&\leq
  \sum_{ l=2 }^q
  \sum_{ k=0 }^n
  \binom{q}{l}
  \left(n+1\right)^{(l-1)}
  t^l \,
  |a_k|^l \,
  \one_{ 
    [0, h^{-\chi}]
  }
  ( \| u + v \|_{ \L^{nq}(\lambda_{\D};\R) } )
  \int_{ \D }
  \left| u(x) \right|^{(q-l)}
  \left| u(x) + v(x) \right|^{kl} \lambda_{\R^d}(dx)
\\&\leq
  \sum_{ l=2 }^q
  \binom{q}{l}
  \left(n+1\right)^{(l-1)}
  t^l \,
  |a_0|^l \,
  \one_{ 
    [0, h^{-\chi}]
  }
  ( \| u + v \|_{ \L^{nq}(\lambda_{\D};\R) } )
  \left\| u \right\|_{ \L^q(\lambda_{\D};\R) }^{(q-l)}
  \left| \lambda_{\R^d}(\D) \right|^{\frac{l}{q}}
\\&\quad+
  \sum_{ l=2 }^q
  \sum_{ k=1 }^n
  \binom{q}{l}
  \left(n+1\right)^{(l-1)}
  t^l \,
  |a_k|^l \,
  \one_{ 
    [0, h^{-\chi}]
  }
  ( \| u + v \|_{ \L^{nq}(\lambda_{\D};\R) } )
  \left\| u \right\|_{ \L^q(\lambda_{\D};\R) }^{(q-l)}
  \left\| u + v \right\|_{ \L^{kq}(\lambda_{\D};\R) }^{kl} .
\end{split}
\end{align}
Again H{\"o}lder's inequality hence shows that
\begin{align}
\begin{split}
 &\sum_{ l=2 }^q
  \binom{q}{l} \,
  t^l \,
  \one_{ 
    [0, h^{-\chi}]
  }
  ( \| u + v \|_{ \L^{nq}(\lambda_{\D};\R) } )
  \int_{ \D }
  \left[ u(x) \right]^{(q-l)}
  \Big[
    \textstyle\sum_{ k=0 }^n
    a_k
    \left[ u(x) + v(x) \right]^k
  \Big]^l \, \lambda_{\R^d}(dx)
\\&\leq
  \sum_{ l=2 }^q
  \binom{q}{l}
  \left(n+1\right)^{(l-1)}
  t^l \,
  |a_0|^l 
  \left\| u \right\|_{ \L^q(\lambda_{\D};\R) }^{(q-l)}
  \left| \lambda_{\R^d}(\D) \right|^{\frac{l}{q}}
\\&+
  \sum_{ l=2 }^q
  \sum_{ k=1 }^n
  \binom{q}{l}
  \left(n+1\right)^{(l-1)}
  t^l \,
  |a_k|^l \,
  \one_{ 
    [0, h^{-\chi}]
  }
  ( \| u + v \|_{ \L^{nq}(\lambda_{\D};\R) } )
  \left\| u \right\|_{ \L^q(\lambda_{\D};\R) }^{(q-l)}
  \left\| u + v \right\|_{ \L^{nq}(\lambda_{\D};\R) }^{kl}
  \left| \lambda_{\R^d}(\D) \right|^{\frac{l(n-k)}{nq}}
\\&\leq
  \sum_{ l=2 }^q
  \tbinom{q}{l}
  \left(n+1\right)^{(l-1)}
  t^l
  \left[ 
    \max_{ j \in \{0,\ldots,n\} }
    |a_j|^l
  \right] \!
  \left[ 
    \left\| u \right\|_{ \L^q(\lambda_{\D};\R) }^{(q-l)}
    | \lambda_{\R^d}(\D) |^{\frac{l}{q}}
    + 
    \sum_{ k=1 }^n
    h^{-kl\chi}
    \left\| u \right\|_{ \L^q(\lambda_{\D};\R) }^{(q-l)}
    | \lambda_{\R^d}(\D) |^{\frac{l(n-k)}{nq}}
  \right]
\\&\leq
  t
  \sum_{ l=2 }^q
  \binom{q}{l}
  \left(n+1\right)^{(l-1)}
  \left[ 
    \max_{ j \in \{0,\ldots,n\} }
    |a_j|^l
  \right]
  \sum_{ k=0 }^n
  h^{(l-1-kl\chi)}
  \left\| u \right\|_{ \L^q(\lambda_{\D};\R) }^{(q-l)}
  | \lambda_{\R^d}(\D) |^{\frac{l(n-k)}{nq}} .
\end{split}
\end{align}
Young's inequality therefore
ensures that
\begin{align}
\label{eq:U_lem2_3}
\begin{split}
 &\sum_{ l=2 }^q
  \binom{q}{l} \,
  t^l \,
  \one_{ 
    [0, h^{-\chi}]
  }
  ( \| u + v \|_{ \L^{nq}(\lambda_{\D};\R) } )
  \int_{ \D }
  \left[ u(x) \right]^{(q-l)}
  \Big[
    \textstyle\sum_{ k=0 }^n
    a_k
    \left[ u(x) + v(x) \right]^k
  \Big]^l \, \lambda_{\R^d}(dx)
\\&\leq
  t
  \sum_{ l=2 }^q
  \binom{q}{l} \,
  \frac{ 
    q^l \left(n+1\right)^l
  }{ 
    q^l \left(n+1\right)
  }
  \left[ 
    \max_{ j \in \{0,\ldots,n\} }
    |a_j|^l
  \right]
  \sum_{ k=0 }^n
  \max\!\left\{ 1, T^{(l-1)} \right\}
  \left\| u \right\|_{ \L^q(\lambda_{\D};\R) }^{(q-l)}
  \left| \lambda_{\R^d}(\D) \right|^{\frac{l(n-k)}{nq}}
\\&\leq
  t
  \sum_{ l=2 }^q
  \tbinom{q}{l} \,
  \tfrac{1}{q^l\left(n+1\right)}
  \sum_{ k=0 }^n
  \left(
    \left\| u \right\|_{ \L^q(\lambda_{\D};\R) }^{q}
    +
    q^q
    \left(n+1\right)^q
    \max\!\left\{ 1, T^{\frac{q(l-1)}{l}} \right\}
    \left[ 
      \max_{ j \in \{0,1,\ldots,n\} }
      |a_j|^q
    \right]
    | \lambda_{\R^d}(\D) |^{\frac{(n-k)}{n}}
  \right)
\\&\leq
  t
  \sum_{ l=2 }^q
  \frac{1}{l!}
  \left(
    \left\| u \right\|_{ \L^q(\lambda_{\D};\R) }^{q}
    +
    \left[ 
      q
      \left(n+1\right)
      \max\{1,T\}
      \max_{ j \in \{0,1,\ldots,n\} }
      |a_j|
    \right]^q
    \max\!\left\{ 1, \lambda_{\R^d}(\D) \right\}
  \right) .
\end{split}
\end{align}
Combining~\eqref{eq:U_lem2_1},
\eqref{eq:U_lem2_2}, and~\eqref{eq:U_lem2_3}
yields that
\begin{align}
\begin{split}
 &\big\|
    u
    +
    t \,
    \one_{ 
      [0, h^{-\chi}]
    }
    ( \left\| u + v \right\|_{ \L^{nq}(\lambda_{\D};\R) } ) \,
    \Ff( u + v ) 
  \big\|_{ \L^q(\lambda_{\D};\R) }^q
\\&\leq
  \left\| u \right\|_{ \L^q(\lambda_{\D};\R) }^q
  +
  q t
  \left[
    a_n
    +
    \sum_{ k = 0 }^{ n }
    \sum_{ j = 0 }^{ \min\{k,n-1\} }
    \binom{k}{j} \,
    | a_k | \,
    | \eta_{(k,j)} |^{\frac{(q+n-1)}{(q+j-1)}}
  \right]
  \left\| u \right\|_{\L^{(q+n-1)}(\lambda_{\D};\R)}^{(q+n-1)}
\\&\quad+
  q t
  \sum_{ k = 0 }^{ n }
  \sum_{ j = 0 }^{ \min\{k,n-1\} }
  \binom{k}{j} \,
  | a_k | \,
  | \eta_{(k,j)} |^{-\frac{(q+n-1)}{(n-j)}}
  \left( 
    \lambda_{\R^d}(\D)
    +
    \left\| v \right\|_{\L^{(q+n-1)}(\lambda_{\D};\R)}^{(q+n-1)}
  \right)
\\&\quad+
  t
  \sum_{ l=2 }^q
  \frac{1}{l!}
  \left(
    \left\| u \right\|_{ \L^q(\lambda_{\D};\R) }^{q}
    +
    \left[ 
      q
      \left(n+1\right)
      \max\{1,T\}
      \max_{ j \in \{0,1,\ldots,n\} }
      |a_j|
    \right]^q
    \max\!\left\{ 1, \lambda_{\R^d}(\D) \right\}
  \right) .
\end{split}
\end{align}
The fact that
$
  \sum_{ l=2 }^{q} \frac{ 1 }{ l! }
  =
  \left(
    \sum_{ l=0 }^q \frac{ 1 }{ l! }
  \right)
  -
  2
  \leq
  e
  -
  2
  <
  1
$
hence assures that
\begin{align}
\begin{split}
 &\big\|
    u
    +
    t \,
    \one_{ 
      [0, h^{-\chi}]
    }
    ( \left\| u + v \right\|_{ \L^{nq}(\lambda_{\D};\R) } ) \,
    \Ff( u + v ) 
  \big\|_{ \L^q(\lambda_{\D};\R) }^q
\\&\leq
  \left(
    1
    +
    t
  \right)
  \left\| u \right\|_{ \L^q(\lambda_{\D};\R) }^q
  +
  q t
  \left[
    a_n
    +
    \sum_{ k = 0 }^{ n }
    \sum_{ j = 0 }^{ \min\{k,n-1\} }
    \binom{k}{j} \,
    | a_k | \,
    |\eta_{(k,j)} |^{\frac{(q+n-1)}{(q+j-1)}}
  \right]
  \| u \|_{\L^{(q+n-1)}(\lambda_{\D};\R)}^{(q+n-1)}
\\&\quad+
  t
  \max\!\left\{ 
    1, 
    \lambda_{\R^d}(\D),
    \| v \|_{\L^{(q+n-1)}(\lambda_{\D};\R)}^{(q+n-1)}
  \right\}
\\&\quad\cdot
  \left(
    2
    q
    \sum_{ k = 0 }^{ n }
    \sum_{ j = 0 }^{ \min\{k,n-1\} }
    \binom{k}{j} \,
    | a_k | \,
    | \eta_{(k,j)} |^{-\frac{(q+n-1)}{(n-j)}}
    +
    \left[ 
      q
      \left(n+1\right)
      \max\{1,T\}
      \max_{ j \in \{0,1,\ldots,n\} }
      |a_j|
    \right]^q
  \right) .
\end{split}
\end{align}
This and the fact that
$ 
  1 + t \leq e^t
$
complete the proof of Lemma~\ref{lem:U_lem2}.
\end{proof}

\begin{corollary}
\label{cor:U_cor1}
Assume the setting in Section~\ref{sec:example_setting}
and let $ q \in \{2n, 2n+2, 2n+4, \ldots \} $.
Then there exists a real number $ K \in (0,\infty) $
such that for all $ h \in (0,T] $,
$ t \in [0,h] $, $ u, v \in L^{nq}(\lambda_{\D}; \R) $,
$ \chi \in [0, \frac{1}{2n}] $
it holds that
\begin{align}
\label{eq:U_cor1_1}
\begin{split} 
 &\max\!\Big\{
    \big\|
      e^{tA}
      \big[
	u
	+
	t \,
	\one_{ 
	  [0, h^{-\chi}]
	}
	( \left\| u + v \right\|_{ L^{nq}(\lambda_{\D};\R) } ) \,
	F( u + v )
      \big]
    \big\|_{ L^q(\lambda_{\D};\R) }^q,
\\&\quad
    \big\|
      \big( \Id_H - tA\big)^{-1}
      \big[
	u
	+
	t \,
	\one_{ 
	  [0, h^{-\chi}]
	}
	( \left\| u + v \right\|_{ L^{nq}(\lambda_{\D};\R) } ) \,
	F( u + v )
      \big]
    \big\|_{ L^q(\lambda_{\D};\R) }^q
  \Big\}
\\&
  \leq
  e^{t}
  \left[
    \big\|
      u
    \big\|_{ L^q(\lambda_{\D};\R) }^q
    + 
    t
    K
    \max\!\left\{
      1,
      \| v \|_{L^{(q+n-1)}(\lambda_{\D};\R)}^{(q+n-1)}
    \right\}
  \right] .
\end{split}
\end{align}
\end{corollary}
\begin{proof}[Proof of Corollary~\ref{cor:U_cor1}]
Combining Lemma~\ref{lem:U_lem1} and Lemma~\ref{lem:U_lem2} 
implies~\eqref{eq:U_cor1_1}. The proof of 
Corollary~\ref{cor:U_cor1} is thus completed.
\end{proof}

\begin{lemma}
\label{lem:limplicit_bound}
Assume the setting in Section~\ref{sec:example_setting}
and let $ t \in (0,\infty) $, $ \kappa \in [0,2] $,
$ j \in \{ 2, 3, \ldots \} $.
Then
$
  \|
    ( - j t A )^{ \kappa } \,
    ( \Id_H - t A )^{ -j }
  \|_{ L( H ) }
  \leq
  4
$.
\end{lemma}
\begin{proof}[Proof of Lemma~\ref{lem:limplicit_bound}]
Note that the assumption that
$ A \colon D(A) \subseteq H \rightarrow H $
is the Laplacian with Dirichlet boundary
conditions on $ H $ ensures that there exists
an orthonormal basis $ \mathbb{B} \subseteq H $ of $ H $
and a mapping
$ \lambda \colon \mathbb{B} \rightarrow (0,\infty) $
such that
$
  D(A)
  =
  \left\{
    v \in H
    \colon
    \sum_{ b \in \mathbb{B} }
    \left| 
      \lambda_b
      \left< b,v \right>_H 
    \right|^2
    <
    \infty
  \right\}
$
and such that
for all $ v \in D(A) $ it holds that
$ 
  Av
  =
  \sum_{ b \in \mathbb{B} }
  -
  \lambda_b
  \left<b,v\right>_H b
$.
Next observe that for all 
$ x \in [0,\infty) $ it holds that
\begin{align}
  \left( 1 + x \right)^j
  \geq
  1
  +
  \frac{ j ( j - 1 ) }{ 2 }
  x^2
  =
  1
  +
  \frac{ j }{ 2 }
  \left[
    \left(
      \frac{ j }{ 2 }
      +
      \left(
        \frac{ j }{ 2 }
        -
        1
      \right)
    \right)
    x^2
  \right]
  \geq
  1
  +
  \frac{ j^2 x^2 }{ 4 }
  \geq
  \frac{ \left( j x \right)^{ \kappa } }{ 4 } .
\end{align}
This implies that
\begin{align}
  \left\|
    \left( - j t A \right)^{ \kappa }
    \left( \Id_H - t A \right)^{ -j }
  \right\|_{ L( H ) }
  =
  \sup_{ b \in \mathbb{B} }
  \left[
    \frac{ \left( j t \lambda_b \right)^{ \kappa } }{ \left( 1 + t \lambda_b \right)^j }
  \right]
  \leq
  \sup_{ x \in [0,\infty) }
  \left(
    \frac{ \left( j x \right)^{ \kappa } }{ \left( 1 + x \right)^j }
  \right)
  \leq
  4 .
\end{align}
The proof
of Lemma~\ref{lem:limplicit_bound} is thus completed.
\end{proof}

\begin{lemma}
\label{lem:semigroup_bound}
Assume the setting in Section~\ref{sec:example_setting}
and let $ q \in [2, \infty) $,
$ \rho \in [ \nicefrac{d}{4} - \nicefrac{d}{2q}, \infty ) $
satisfy $ \nicefrac{d}{4} - \nicefrac{d}{2q} < 1 $. Then
$
  \sup_{ M \in \N }
  \sup_{ 0\leq s<t \leq T }
  \big[
    t^{\rho}
    \|
      e^{tA}
    \|_{ L( H, L^q(\lambda_{\D}; \R) ) }
    +
    \left(t-s\right)^{\rho}
    \|
      e^{(t-\fl{s})A}
    \|_{ L( H, L^q(\lambda_{\D}; \R) ) }
  \big]
  <
  \infty
$
and
\begin{align}
\begin{split}
  \sup_{ M \in \N }
  \sup_{ 0\leq s<t \leq T }
  \left(t-s\right)^{\rho}
  \big\|
    \big( \Id_H - (t-\fl{t})A \big)^{-1}
    \big( \Id_H - \tfrac{T}{M}A\big)^{(\fl{s}-\fl{t})\frac{M}{T}}
  \big\|_{ L( H, L^q(\lambda_{\D}; \R) ) }
 &<
  \infty .
\end{split}
\end{align}
\end{lemma}
\begin{proof}[Proof of Lemma~\ref{lem:semigroup_bound}]
Note that the fact that 
for all
$ r \in [0,\infty) $
it holds that
$
  H_r \subseteq W^{2r,2}(\D,\R)
$
continuously
(cf., e.g., (A.46) in Da Prato \& Zabczyk \cite{dz92} and Lunardi \cite{l09})
and the fact that for all $ r \in [\nicefrac{d}{4} - \nicefrac{d}{2q},\infty) $ 
it holds that 
$ 
  W^{2r,2}(\D,\R) 
  \subseteq 
  W^{0,q}(\D,\R) 
  = 
  L^q(\lambda_{\D};\R) 
$
continuously imply that
there exists a real number
$ C \in (0,\infty) $ 
such that for all $ v \in H_{(\nicefrac{d}{4} - \nicefrac{d}{2q})} $
it holds that 
$ \| v \|_{ L^q(\lambda_{\D};\R) } \leq C \| v \|_{ H_{(\nicefrac{d}{4} - \nicefrac{d}{2q})} } $.
This and the fact that 
$ \forall \, s \in [0,\infty) $, 
$ \kappa \in [0,1] 
  \colon 
  \| 
    (-sA)^{\kappa} e^{sA} 
  \|_{ L(H) } \leq 1 
$
prove that for all $ t \in (0,T] $ it holds that
\begin{align}
\begin{split}
 &\sup_{ v \in H \backslash \{0\} }
  \left[ 
    \frac{ 
      t^{\rho} \,
      \big\|
	e^{tA}
	v
      \big\|_{ L^q(\lambda_{\D};\R) }
    }{
      \left\| v \right\|_H
    }
  \right]
\\&\leq
  \sup_{ v \in H \backslash \{0\} }
  \left[ 
    \frac{ 
      C \,
      t^{\rho} \,
      \big\|
	e^{tA}
	v
      \big\|_{ H_{(\nicefrac{d}{4} - \nicefrac{d}{2q})} }
    }{
      \left\| v \right\|_H
    }
  \right]
  =
  \sup_{ v \in H \backslash \{0\} }
  \left[ 
    \frac{ 
      C \,
      t^{\rho} \,
      \big\| \!
	\left( -A \right)^{(\nicefrac{d}{4} - \nicefrac{d}{2q})}
	e^{tA}
	v
      \big\|_{ H }
    }{
      \left\| v \right\|_H
    }
  \right]
\\&\leq
  C \,
  t^{( \rho - [\nicefrac{d}{4} - \nicefrac{d}{2q}]) } \,
  \big\| \!
    \left( -tA \right)^{(\nicefrac{d}{4} - \nicefrac{d}{2q})}
    e^{tA}
  \big\|_{ L( H ) }
  \leq
  C \, 
  t^{( \rho - [\nicefrac{d}{4} - \nicefrac{d}{2q}]) } .
\end{split}
\end{align}
Hence, we obtain that
$
  \sup_{ t \in (0,T] }
  t^{\rho}
  \|
    e^{tA}
  \|_{ L( H, L^q(\lambda_{\D}; \R) ) }
  \leq
  C \,
  T^{( \rho - [\nicefrac{d}{4} - \nicefrac{d}{2q}]) }
  <
  \infty
$.
This shows that
\begin{align}
\begin{split}
 &\sup_{ M \in \N }
  \sup_{ 0\leq s<t \leq T }
  \left(t-s\right)^{\rho}
  \big\|
    e^{(t-\fl{s})A}
  \big\|_{ L( H, L^q(\lambda_{\D}; \R) ) }
\\&\leq
  \sup_{ M \in \N }
  \sup_{ 0\leq s<t \leq T }
  \big(t-\fl{s}\big)^{\rho} \,
  \big\|
    e^{(t-\fl{s})A}
  \big\|_{ L( H, L^q(\lambda_{\D}; \R) ) }
  \leq
  \sup_{ t \in (0,T] }
  t^{\rho} \,
  \big\|
    e^{tA}
  \big\|_{ L( H, L^q(\lambda_{\D}; \R) ) }
  <
  \infty .
\end{split}
\end{align}
In addition, note that
the fact that
$ 
  \forall \,
  x,y \in \R
$,
$
  r \in (0, \infty) 
  \colon
  | x + y |^r 
  \leq 
  2^{[r-1]^+}
  \left(
    |x|^r
    +
    |y|^r
  \right)
$ implies
that for all $ s,t \in [0,T] $,
$ M \in \N $ with $ s<t $ 
it holds that
\begin{align}
\begin{split}
 &\sup_{ v \in H \backslash \{0\} }
  \left[ 
    \frac{ 
      \left(t-s\right)^{\rho}
    }{
      \left\| v \right\|_H
    }
    \big\|
      \big( \Id_H - (t-\fl{t})A \big)^{-1}
      \big( \Id_H - \tfrac{T}{M}A\big)^{(\fl{s}-\fl{t})\frac{M}{T}}
      v
    \big\|_{ L^q(\lambda_{\D}; \R) }
  \right]
\\&\leq
  \sup_{ v \in H \backslash \{0\} }
  \left[ 
    \frac{
      C
      \left(t-s\right)^{\rho}
    }{
      \left\| v \right\|_H
    }
    \big\|
      \big( \Id_H - (t-\fl{t})A \big)^{-1}
      \big( \Id_H - \tfrac{T}{M}A\big)^{(\fl{s}-\fl{t})\frac{M}{T}}
      v
    \big\|_{ H_{(\nicefrac{d}{4} - \nicefrac{d}{2q})} }
  \right]
\\&=
  \sup_{ v \in H \backslash \{0\} }
  \left[ 
    \frac{
      C
      \left(t-s\right)^{\rho}
    }{
      \left\| v \right\|_H
    }
    \big\|
      ( -A )^{(\nicefrac{d}{4} - \nicefrac{d}{2q})}
      \big( \Id_H - (t-\fl{t})A \big)^{-1}
      \big( \Id_H - \tfrac{T}{M}A\big)^{(\fl{s}-\fl{t})\frac{M}{T}}
      v
    \big\|_{ H }
  \right]
\\&\leq
  C
  \left(t-\fl{s}\right)^{\rho}
  \big\|
    ( -A )^{(\nicefrac{d}{4} - \nicefrac{d}{2q})}
    \big( \Id_H - (t-\fl{t})A \big)^{-1}
    \big( \Id_H - \tfrac{T}{M}A\big)^{(\fl{s}-\fl{t})\frac{M}{T}}
  \big\|_{ L(H) }
\\&\leq 
  2^{[\rho-1]^+}
  C
  \left[
    \left(t-\fl{t}\right)^{\rho}
    +
    \left(\fl{t}-\fl{s}\right)^{\rho}
  \right]
\\&\quad\cdot
  \big\|
    ( -A )^{(\nicefrac{d}{4} - \nicefrac{d}{2q})}
    \big( \Id_H - (t-\fl{t})A \big)^{-1}
    \big( \Id_H - \tfrac{T}{M}A\big)^{(\fl{s}-\fl{t})\frac{M}{T}}
  \big\|_{ L( H ) } .
\end{split}
\end{align}
This assures that
\begin{align}
\begin{split}
 &\sup_{ M \in \N }
  \sup_{ 0\leq s<t \leq T }
  \left(t-s\right)^{\rho}
  \big\|
    \big( \Id_H - (t-\fl{t})A \big)^{-1}
    \big( \Id_H - \tfrac{T}{M}A\big)^{(\fl{s}-\fl{t})\frac{M}{T}}
  \big\|_{ L( H, L^q(\lambda_{\D}; \R) ) }
\\&\leq
  \sup_{ M \in \N }
  \sup_{ 0 \leq s<t \leq T }
  2^{[\rho-1]^+}
  C
  \left(t-\fl{t}\right)^{\rho} 
  \big\|
    ( -A )^{(\nicefrac{d}{4} - \nicefrac{d}{2q})}
    \big( \Id_H - (t-\fl{t})A \big)^{-1}
  \big\|_{ L( H ) }
\\&\quad\cdot
  \big\|
    \big( \Id_H - \tfrac{T}{M}A\big)^{-1}
  \big\|_{ L( H ) }^{(\fl{t}-\fl{s})\frac{M}{T}}
\\&\quad+
  \sup_{ M \in \N }
  \sup_{ t \in [\frac{T}{M}, T] }
  \sup_{ s \in [0,\fl{t}) }
  2^{[\rho-1]^+}
  C 
  \left(\fl{t}-\fl{s}\right)^{\rho}
  \big\|
    \big( \Id_H - (t-\fl{t})A \big)^{-1}
  \big\|_{ L( H ) }
\\&\quad\cdot
  \big\|
    ( -A )^{(\nicefrac{d}{4} - \nicefrac{d}{2q})}
    \big( \Id_H - \tfrac{T}{M}A\big)^{(\fl{s}-\fl{t})\frac{M}{T}}
  \big\|_{ L( H ) } .
\end{split}
\end{align}
The fact that
$ \forall \, s \in [0,\infty) $, 
$ \kappa \in [0,1] 
  \colon
  \| 
    (-sA)^{\kappa} 
    ( \Id_H - sA)^{-1} 
  \|_{ L(H) } 
  \leq 1 
$
therefore yields that
\begin{align}
\begin{split}
 &\sup_{ M \in \N }
  \sup_{ 0 \leq s<t \leq T }
  \left(t-s\right)^{\rho}
  \big\|
    \big( \Id_H - (t-\fl{t})A \big)^{-1}
    \big( \Id_H - \tfrac{T}{M}A\big)^{(\fl{s}-\fl{t})\frac{M}{T}}
  \big\|_{ L( H, L^q(\lambda_{\D}; \R) ) }
\\&\leq
  \sup_{ M \in \N }
  \sup_{ t \in (0,T] }
  2^{[\rho-1]^+}
  C
  \left(t-\fl{t}\right)^{\rho}
  \big\|
    ( -A )^{(\nicefrac{d}{4} - \nicefrac{d}{2q})}
    \big( \Id_H - (t-\fl{t})A \big)^{-1}
  \big\|_{ L( H ) }
\\&+
  \sup_{ M \in \N }
  \sup_{ l \in \{ 1,2,\ldots,M\} }
  \sup_{ k \in \{ 0,1,\ldots,l-1 \} }
  2^{[\rho-1]^+}
  C
  \left(\tfrac{T}{M}\left(l-k\right)\right)^{\rho}
  \big\|
    ( -A )^{(\nicefrac{d}{4} - \nicefrac{d}{2q})}
    \big( \Id_H - \tfrac{T}{M}A\big)^{-(l-k)}
  \big\|_{ L( H ) }
\\&\leq
  2^{[\rho-1]^+}
  C
  \sup_{ M \in \N }
  \left[ 
    \sup_{ t \in (0,T] }
    \left(t-\fl{t}\right)^{(\rho-[\nicefrac{d}{4} - \nicefrac{d}{2q}])}
    +
    \left(\tfrac{T}{M}\right)^{\rho}
    \big\|
      ( -A )^{(\nicefrac{d}{4} - \nicefrac{d}{2q})}
      \big( \Id_H - \tfrac{T}{M}A\big)^{-1}
    \big\|_{ L( H ) }
  \right]
\\&+
  2^{[\rho-1]^+}
  C
  \sup_{ M \in \{ 2,3,\ldots\} }
  \sup_{ l \in \{ 2,3,\ldots,M\} }
  \sup_{ k \in \{ 0,1,\ldots,l-2 \} }
  \left(\tfrac{T}{M}\left(l-k\right)\right)^{\rho} 
  \big\|
    ( -A )^{(\nicefrac{d}{4} - \nicefrac{d}{2q})}
    \big( \Id_H - \tfrac{T}{M}A\big)^{-(l-k)}
  \big\|_{ L( H ) } .
\end{split}
\end{align}
This, Lemma~\ref{lem:limplicit_bound},
and again the fact that 
$ \forall \, s \in [0,\infty) $, 
$ \kappa \in [0,1] 
  \colon
  \| 
    (-sA)^{\kappa} 
    ( \Id_H - sA )^{-1} 
  \|_{ L(H) } 
  \leq 1 
$
prove that
\begin{align}
\begin{split}
 &\sup_{ M \in \N }
  \sup_{ 0 \leq s<t \leq T }
  \left(t-s\right)^{\rho}
  \big\|
    \big( \Id_H - (t-\fl{t})A \big)^{-1}
    \big( \Id_H - \tfrac{T}{M}A\big)^{(\fl{s}-\fl{t})\frac{M}{T}}
  \big\|_{ L( H, L^q(\lambda_{\D}; \R) ) }
\\&\leq
  2^{[\rho-1]^+}
  C
  \left[ 
    \sup_{ M \in \N }
    2
    \left(\tfrac{T}{M}\right)^{(\rho-[\nicefrac{d}{4} - \nicefrac{d}{2q}])}
    + 
    \sup_{ M \in \{ 2,3,\ldots\} }
    \sup_{ l \in \{ 2,3,\ldots,M\} }
    \sup_{ k \in \{ 0,1,\ldots,l-2 \} }
    4
    \left(\tfrac{T}{M}\left(l-k\right)\right)^{(\rho-[\nicefrac{d}{4} - \nicefrac{d}{2q}])}
  \right]
\\&\leq
  6\cdot 
  2^{[\rho-1]^+}
  C \,
  T^{(\rho-[\nicefrac{d}{4} - \nicefrac{d}{2q}])}
  <
  \infty .
\end{split}
\end{align}
The proof of Lemma~\ref{lem:semigroup_bound} 
is thus completed. 
\end{proof}

\begin{lemma}
\label{lem:F_lem1}
Assume the setting in Section~\ref{sec:example_setting}
and let $ q \in \{2n,2n+2,2n+4, \ldots\} $,
$ v, w \in L^{nq}(\lambda_{\D};\R) $.
Then
\begin{equation}
  \big\|
    F(v+w)
  \big\|_{ H }
  \leq
  2^{(n+1)}
  \sqrt{ 
    \max\{1,\lambda_{ \R^d }(\D)\}
  }
  \left[
    \max_{ k\in\{0,1,\ldots,n\} }
    | a_k |
  \right]
  \left(
    1
    +
    \left\|
      v
    \right\|_{ L^{q}(\lambda_{\D};\R) }^{\nicefrac{q}{2}}
    +
    \left\|
      w
    \right\|_{ L^{q}(\lambda_{\D};\R) }^{\nicefrac{q}{2}}
  \right) .
\end{equation}
\end{lemma}
\begin{proof}[Proof of Lemma~\ref{lem:F_lem1}]
Note that
H{\"o}lder's inequality
ensures that
\begin{align}
\begin{split}
  \big\|
    F(v+w)
  \big\|_{ H }
 &\leq
  \sum_{k=0}^n
  \left| a_k \right|
  \left\|
    \left[
      v + w
    \right]^k
  \right\|_H
  =
  \left|
    a_0
  \right|
  \sqrt{ 
    \lambda_{ \R^d }(\D)
  }
  +
  \sum_{k=1}^n
  \left| a_k \right|
  \left\|
      v + w
  \right\|_{ L^{2k}(\lambda_{\D}; \R) }^k
\\&\leq 
  \left|
    a_0
  \right|
  \sqrt{ 
    \lambda_{ \R^d }(\D)
  }
  +
  \sum_{k=1}^n
  \left| a_k \right|
  \left[ 
    \left|
      \lambda_{ \R^d }(\D)
    \right|^{(\frac{1}{2k}-\frac{1}{q})}
    \left\|
      v + w
    \right\|_{ L^{q}(\lambda_{\D}; \R) }
  \right]^k
\\&\leq
  \left|
    a_0
  \right|
  \sqrt{ 
    \lambda_{ \R^d }(\D)
  }
  +
  \sum_{k=1}^n
  \left| a_k \right|
  \left|
    \lambda_{ \R^d }(\D)
  \right|^{(\frac{1}{2}-\frac{k}{q})}
  \left[ 
    \left\|
      v
    \right\|_{ L^{q}(\lambda_{\D}; \R) }
    +
    \left\|
      w
    \right\|_{ L^{q}(\lambda_{\D}; \R) }
  \right]^k
\\&=
  \sum_{k=0}^n
  \left| a_k \right|
  \left|
    \lambda_{ \R^d }(\D)
  \right|^{(\frac{1}{2}-\frac{k}{q})}
  \left[ 
    \left\|
      v
    \right\|_{ L^{q}(\lambda_{\D}; \R) }
    +
    \left\|
      w
    \right\|_{ L^{q}(\lambda_{\D}; \R) }
  \right]^k .
\end{split}
\end{align}
The fact that
$ \forall \, x,y,r \in [0,\infty) 
  \colon
  ( x + y )^{r}
  \leq
  2^{[r-1]^+}
  (
    x^r
    +
    y^r
  )
$
hence
implies that
\begin{align}
\begin{split}
  \big\|
    F(v+w)
  \big\|_{ H }
 &\leq
  \sqrt{ 
    \max\{1,\lambda_{ \R^d }(\D)\}
  }
  \sum_{k=0}^n
  \left| a_k \right|
  2^{(k-1)}
  \left[ 
    \left\|
      v
    \right\|_{ L^{q}(\lambda_{\D}; \R) }^k
    +
    \left\|
      w
    \right\|_{ L^{q}(\lambda_{\D}; \R) }^k
  \right]
\\&\leq
  \sqrt{ 
    \max\{1,\lambda_{ \R^d }(\D)\}
  }
  \sum_{k=0}^n
  \left| a_k \right|
  2^{k}
  \left[
    1
    +
    \left\|
      v
    \right\|_{ L^{q}(\lambda_{\D}; \R) }^{\nicefrac{q}{2}}
    +
    \left\|
      w
    \right\|_{ L^{q}(\lambda_{\D}; \R) }^{\nicefrac{q}{2}}
  \right]
\\&\leq
  \sqrt{ 
    \max\{1,\lambda_{ \R^d }(\D)\}
  }
  \left[
    \max_{ k\in\{0,1,\ldots,n\} }
    | a_k |
  \right]
  \left[ 
    \sum_{k=0}^n
    2^{k}
  \right]
  \left[
    1
    +
    \left\|
      v
    \right\|_{ L^{q}(\lambda_{\D}; \R) }^{\nicefrac{q}{2}}
    +
    \left\|
      w
    \right\|_{ L^{q}(\lambda_{\D}; \R) }^{\nicefrac{q}{2}}
  \right] .
\end{split}
\end{align}
This completes the proof of Lemma~\ref{lem:F_lem1}.
\end{proof}

\begin{lemma}
\label{lem:F_lem2}
Assume the setting in Section~\ref{sec:example_setting}
and let
$ v, w \in \BaMi $
with $ v-w \in H_1 $.
Then
\begin{align}
\begin{split}
 &\left<
    v-w,
    A(v-w) + F(v) - F(w)
  \right>_{ H }
\\&\leq
  \left(n-1\right)
  \max\!\left\{ 1, \left|a_n\right|^{(2-n)} \right\}
  \max\!\left\{ 
    1, \max_{ k \in \{ 0,1,\ldots,n-1 \} } \big[ k \left|a_k\right| \big]^{(n-1)}
  \right\}
  \left\|
    v-w
  \right\|_{ H }^2 .
\end{split}
\end{align}
\end{lemma}
\begin{proof}[Proof of Lemma~\ref{lem:F_lem2}]
Throughout this proof
assume w.l.o.g.\ that $ n>1 $
(otherwise the proof is clear)
and let
$ f \colon \R \rightarrow \R $
be the function with the
property that for all
$ x \in \R $ it holds that
$ f(x) = \sum_{k=0}^n a_k x^k $.
Note that the fundamental theorem of 
calculus implies that
for all $ x,y \in \R $ it holds that
\begin{align}
\label{eq:F_lem2_1}
\begin{split}
  \left(x-y\right)
  \left( f(x) - f(y) \right)
 &=
  \left(x-y\right)
  \int_y^x f'(z) \, dz
  \leq
  \left[
    \sup_{z\in \R}
    f'(z)
  \right]
  \left(x-y\right)^2 .
\end{split}
\end{align}
Next note that 
Young's inequality
shows
that for all
$ x \in \R $ it holds that
\begin{align}
\label{eq:F_lem2_2}
\begin{split}
 &f'(x)
  =
  \sum_{k=1}^n
  k a_k x^{(k-1)}
  \leq
  n a_n x^{(n-1)}
  +
  |a_1|
  +
  \sum_{k=2}^{n-1}
  k \left|a_k\right|
  \left|x\right|^{(k-1)}
\\&=
  n a_n x^{(n-1)}
  +
  |a_1|
  +
  \sum_{k=2}^{n-1}
  k \left|a_k\right|
  \left[ 
    \frac{ 
      \left( n-1 \right)
      \left| a_n \right|
    }{
      \left( k-1 \right)
      \max\{ 1, k \left|a_k\right| \}
    }
  \right]^{\frac{(k-1)}{(n-1)}} 
  \left|x\right|^{(k-1)}
  \left[ 
    \frac{ 
      \left( n-1 \right)
      \left| a_n \right|
    }{
      \left( k-1 \right)
      \max\{ 1, k \left|a_k\right| \}
    }
  \right]^{-\frac{(k-1)}{(n-1)}}
\\&\leq
  n a_n x^{(n-1)}
  +
  |a_1|
  +
  \sum_{k=2}^{n-1}
  k \left|a_k\right|
  \left(
    \frac{ 
      \left|a_n\right|
      x^{(n-1)}
    }{
      \max\{ 1, k \left|a_k\right| \}
    }
    +
    \frac{(n-k)}{(n-1)}
    \left[
      \frac{
        \left( n-1 \right)
        \left|a_n\right|
      }{
        \left(k-1\right)
        \max\{ 1, k \left|a_k\right| \}
      }
    \right]^{-\frac{(k-1)}{(n-k)}}
  \right)
\\&\leq
  \big[
    n a_n
    +
    \left(n-2\right)
    \left|a_n\right|
  \big]
  x^{(n-1)}
  +
  |a_1|
  +
  \sum_{k=2}^{n-1}
  k \left|a_k\right|
  \left[
    \frac{
      \max\{ 1, k \left|a_k\right| \}
    }{
      \left|a_n\right|
    }
  \right]^{\frac{(k-1)}{(n-k)}}
\\&\leq
  \sum_{k=1}^{n-1}
  \max\!\left\{ 1, \left|a_n\right|^{-(n-2)} \right\}
  \big[ \max\!\left\{ 1, k \left|a_k\right| \right\} \big]^{(n-1)}
\\&\leq
  \left(n-1\right)
  \max\!\left\{ 1, \left|a_n\right|^{-(n-2)} \right\}
  \max\!\left\{ 
    1, \max_{ k \in \{ 0,1,\ldots,n-1 \} } \big[ k \left|a_k\right| \big]^{(n-1)}
  \right\} .
\end{split}
\end{align}
In addition, observe 
that~\eqref{eq:F_lem2_1}--\eqref{eq:F_lem2_2}
prove that
\begin{align}
\label{eq:F_lem2_3}
\begin{split}
 &\left<
    v-w,
    F(v) - F(w)
  \right>_{ H }
  =
  \int_{\D}
  \left(
    v(x) - w(x)
  \right)
  \big[
    f(v(x))
    -
    f(w(x))
  \big] \, \lambda_{\R^d}(dx)
\\&\leq
  \left(n-1\right)
  \max\!\left\{ 1, \left|a_n\right|^{-(n-2)} \right\}
  \max\!\left\{ 
    1, \max_{ k \in \{ 0,1,\ldots,n-1 \} } \big[ k \left|a_k\right| \big]^{(n-1)}
  \right\}
  \int_{\D}
  \left(
    v(x) - w(x)
  \right)^2 \lambda_{\R^d}(dx)
\\&=
  \left(n-1\right)
  \max\!\left\{ 1, \left|a_n\right|^{-(n-2)} \right\}
  \max\!\left\{ 
    1, \max_{ k \in \{ 0,1,\ldots,n-1 \} } \big[ k \left|a_k\right| \big]^{(n-1)}
  \right\}
  \left\|
    v-w
  \right\|_{ H }^2 .
\end{split}
\end{align}
Furthermore, note that
\begin{align}
\begin{split}
  \left<
    v-w,
    A(v-w)
  \right>_{ H }
 &=
  -
  \big<
    v-w,
    (-A)^{\nicefrac{1}{2}}
    (-A)^{\nicefrac{1}{2}}
    (v-w) 
  \big>_{ H } 
\\&=
  -
  \big< 
    (-A)^{\nicefrac{1}{2}}
    (v-w),
    (-A)^{\nicefrac{1}{2}}
    (v-w)
  \big>_{ H }
  =
  -
  \big\| 
    (-A)^{\nicefrac{1}{2}}
    (v-w) 
  \big\|_{ H }
  \leq
  0 .
\end{split}
\end{align}
Combining this with \eqref{eq:F_lem2_3}
completes the proof of
Lemma~\ref{lem:F_lem2}.
\end{proof}

\begin{lemma}
\label{lem:F_lem3}
Assume the setting in Section~\ref{sec:example_setting}
and let
$ v, w \in \BaMi $.
Then 
\begin{align}
\begin{split}
 &\left\|
    F(v) - F(w)
  \right\|_{ H }^2
\\&\leq
  \max\!\left\{ 1, \lambda_{\R^d}(\D) \right\}
  \left[ 
    \tfrac{n\left(n+1\right)}{2}
    \max_{j\in\{1,2,\ldots,n\}}
    \left| a_j \right|
  \right]^2
  \left\|
    v-w
  \right\|_{ \BaMi }^2
  \left(
    1
    +
    \left\| v \right\|_{ \BaMi }^{\max\{1,2(n-1)\}}
    +
    \left\| w \right\|_{ \BaMi }^{\max\{1,2(n-1)\}}
  \right).
\end{split}
\end{align}
\end{lemma}
\begin{proof}[Proof of Lemma~\ref{lem:F_lem3}]
Throughout this proof 
assume w.l.o.g.\ that
$ n>1 $ (otherwise the proof is clear).
Next note that the 
fundamental theorem of
calculus implies that for all 
$ k \in \N $, $ x,y \in \R $
it holds that
\begin{align}
\begin{split}
  \left|
    x^k - y^k 
  \right| 
 &=
  \left|
    \int_0^1
    k
    \left( y + r\left( x-y \right) \right)^{(k-1)}
    \left( x-y \right) dr
  \right|
  \leq
  k
  \left| x-y \right|
  \int_0^1
  \left|
    rx
    +
    \left(1-r\right)y
  \right|^{(k-1)} dr
\\&\leq
  k
  \left| x-y \right|
  \int_0^1
  \left(
    r
    \left| x \right|^{(k-1)}
    +
    \left(1-r\right)
    \left| y \right|^{(k-1)}
  \right) dr .
\end{split}
\end{align}
This and H{\"o}lder's inequality
prove that
\begin{align}
\begin{split}
 &\left\|
    F(v) - F(w)
  \right\|_H
  =
  \left\|
    \textstyle\sum_{ k=0 }^n
    a_k
    \left( v^k - w^k \right)
  \right\|_H
  \leq 
  \sum_{ k=1 }^{ n }
  \left| a_k \right|
  \left\| v^k - w^k \right\|_H
\\&\leq
  \sum_{ k=1 }^n
  k \left| a_k \right|
  \int_0^1
  \left\|
    \left| v-w \right|
    \left(
      r \left| v \right|^{(k-1)}
      +
      \left(1-r\right)
      \left| w \right|^{(k-1)}
    \right)
  \right\|_H dr
\\&\leq
  \sum_{ k=1 }^n
  k \left| a_k \right|
  \int_0^1
  \left\|
    v-w
  \right\|_V
  \left\|
    r \left| v \right|^{(k-1)}
    +
    \left(1-r\right)
    \left| w \right|^{(k-1)}
  \right\|_{ L^{\nicefrac{2n^2\!}{(n^2-1)}}(\lambda_{\D};\R) } dr
\\&\leq
  \left\|
    v-w
  \right\|_V
  \Bigg[
    \left| a_1 \right|
    \sqrt{ \max\!\left\{ 1, \lambda_{\R^d}(\D) \right\} }
\\&\quad+
    \sum_{ k=2 }^n
    k \left| a_k \right|
    \int_0^1
    \left(
      r \,
      \big\| v \big\|_{ L^{\nicefrac{2n^2(k-1)}{(n^2-1)}}\!(\lambda_{\D};\R) }^{(k-1)}
      +
      \left(1-r\right)
      \big\| w \big\|_{ L^{\nicefrac{2n^2(k-1)}{(n^2-1)}}\!(\lambda_{\D};\R) }^{(k-1)}
    \right) dr
  \Bigg] .
\end{split}
\end{align}
Again H{\"o}lder's inequality hence ensures that
\begin{align}
\begin{split}
 &\left\|
    F(v) - F(w)
  \right\|_H
\\&\leq
  \left\|
    v-w
  \right\|_V
  \Bigg[
    \left| a_1 \right|
    \sqrt{ \max\!\left\{ 1, \lambda_{\R^d}(\D) \right\} }
\\&\quad+
    \frac{1}{2}
    \sum_{ k=2 }^n
    k \left| a_k \right|
    \left(
      \big\| v \big\|_{ L^{\nicefrac{2n^2(k-1)}{(n^2-1)}}\!(\lambda_{\D};\R) }^{(k-1)}
      +
      \big\| w \big\|_{ L^{\nicefrac{2n^2(k-1)}{(n^2-1)}}\!(\lambda_{\D};\R) }^{(k-1)}
    \right)
  \Bigg]
\\&\leq
  \left\|
    v-w
  \right\|_V
  \Bigg[
    \left| a_1 \right|
    \sqrt{ \max\!\left\{ 1, \lambda_{\R^d}(\D) \right\} }
    +
    \frac{1}{2}
    \sum_{ k=2 }^n
    k \left| a_k \right|
    \left| \lambda_{\R^d}(\D) \right|^{\frac{(n^2-k)}{2n^2}}
    \left(
      \left\| v \right\|_V^{(k-1)}
      +
      \left\| w \right\|_V^{(k-1)}
    \right)
  \Bigg]
\\&\leq
  \frac{1}{2}
  \left\|
    v-w
  \right\|_V
  \sqrt{ \max\!\left\{ 1, \lambda_{\R^d}(\D) \right\} }
  \left[
    \max_{ j \in \{1,2,\ldots,n\} }
    \left| a_j \right|
  \right]
  \sum_{ k=1 }^n
  k
  \left(
    \left\| v \right\|_V^{(k-1)}
    +
    \left\| w \right\|_V^{(k-1)}
  \right)
\\&\leq
  \tfrac{n\left(n+1\right)}{4}
  \left\|
    v-w
  \right\|_V
  \sqrt{ \max\!\left\{ 1, \lambda_{\R^d}(\D) \right\} }
  \left[
    \max_{ j \in \{1,2,\ldots,n\} }
    \left| a_j \right|
  \right]
  \left(
    \max\!\left\{ 
      1, \left\| v \right\|_V^{(n-1)}
    \right\}
    +
    \max\!\left\{ 
      1, \left\| w \right\|_V^{(n-1)}
    \right\}
  \right) .
\end{split}
\end{align}
Combining this with the fact that
$ 
  \forall \,
  x,y \in \R \colon
  | x + y |^2
  \leq 
  2 \, ( x^2 + y^2 )
$
completes the proof of Lemma~\ref{lem:F_lem3}.
\end{proof}

\begin{lemma}
\label{lem:exponential_order}
Assume the setting in Section~\ref{sec:example_setting}
and let $ q \in [2, \infty) $,
$ \varrho \in [0,\infty) \cap (-\infty , 1+\nicefrac{d}{2q}-\nicefrac{d}{4} ) $. 
Then
$
  \sup_{ M \in \N }
  \sup_{ t \in (0,T] }
  M^{ \varrho }
  \int_0^t
  \|
    e^{(t-s)A}
    -
    e^{(t-\fl{s})A}
  \|_{ L( H, L^q(\lambda_{\D}; \R) ) } \, ds
  <
  \infty
$.
\end{lemma}
\begin{proof}[Proof of Lemma~\ref{lem:exponential_order}]
Observe that for all $ t \in (0,T] $,
$ M \in \N $ it holds that
\begin{align}
\begin{split}
 &\int_0^t
  \left\|
    e^{(t-s)A}
    -
    e^{(t-\fl{s})A}
  \right\|_{ L( H, L^q(\lambda_{\D};\R) ) } ds
\\&=
  \int_0^t
  \left\|
    e^{(t-s)A}
    \left(
      \Id_H
      -
      e^{(s-\fl{s})A}
    \right)
  \right\|_{ L( H, L^q(\lambda_{\D};\R) ) } ds
\\&\leq
  \int_0^t
  \left\|
    e^{\frac{1}{2}(t-s)A}
  \right\|_{ L( H, L^q(\lambda_{\D};\R) ) }
  \left\|
    e^{\frac{1}{2}(t-s)A}
    \left(
      \Id_H
      -
      e^{(s-\fl{s})A}
    \right)
  \right\|_{ L( H ) } ds
\\&\leq
  \left[
    \sup_{ s \in (0,T] }
    s^{(\nicefrac{d}{4}-\nicefrac{d}{2q})}
    \left\|
      e^{sA}
    \right\|_{ L( H, L^q(\lambda_{\D};\R) ) }
  \right]
  \int_0^t
  \left[\tfrac{1}{2}(t-s)\right]^{-(\varrho+\nicefrac{d}{4}-\nicefrac{d}{2q})}
  \left(s-\fl{s}\right)^{\varrho}
\\&\quad\cdot
  \left\|
    \left(-\tfrac{1}{2}(t-s)A\right)^{\varrho}
    e^{\frac{1}{2}(t-s)A}
  \right\|_{ L(H) }
  \left\|
    \left(-(s-\fl{s})A\right)^{-\varrho}
    \left(
      \Id_H
      -
      e^{(s-\fl{s})A}
    \right)
  \right\|_{ L(H) } ds .
\end{split}
\end{align}
This and the facts that
$ \forall \, s \in [0,\infty) $, 
$ r \in [0,1] 
  \colon
  \|
    (-sA)^r
    e^{sA}
  \|_{ L(H) }
  \leq 1
$
and
$ \forall \, s \in (0,\infty) $, 
$ r \in [0,1] 
  \colon
  \|
    (-sA)^{-r}
    (
      \Id_H
      -
      e^{sA}
    )
  \|_{ L(H) }
  \leq 1
$
show that
for all $ t \in (0,T] $,
$ M \in \N $
it holds that
\begin{align}
\begin{split}
 &\int_0^t
  \left\|
    e^{(t-s)A}
    -
    e^{(t-\fl{s})A}
  \right\|_{ L( H, L^q(\lambda_{\D};\R) ) } ds
\\&\leq
  \left[
    \sup_{ s \in (0,T] }
    s^{(\nicefrac{d}{4}-\nicefrac{d}{2q})}
    \left\|
      e^{sA}
    \right\|_{ L( H, L^q(\D;\R) ) }
  \right]
  \int_0^t
  \left[\tfrac{1}{2}(t-s)\right]^{-(\varrho+\nicefrac{d}{4}-\nicefrac{d}{2q})}
  \frac{T^{\varrho}}{M^{\varrho}} \, ds
\\&\leq
  \frac{
    2 \, T^{(1+\nicefrac{d}{2q}-\nicefrac{d}{4})}
  }{
    \left(1-\varrho-\nicefrac{d}{4}+\nicefrac{d}{2q}\right)M^{\varrho}
  }
  \left[
    \sup_{ s \in (0,T] }
    s^{(\nicefrac{d}{4}-\nicefrac{d}{2q})}
    \left\|
      e^{sA}
    \right\|_{ L( H, L^q(\lambda_{\D};\R) ) }
  \right] .
\end{split}
\end{align}
Combining this with Lemma~\ref{lem:semigroup_bound} 
completes the proof of Lemma~\ref{lem:exponential_order}.
\end{proof}

\begin{lemma}
\label{lem:limplicit_order_1}
Assume the setting in Section~\ref{sec:example_setting}
and let $ t \in (0,\infty) $, $ \kappa \in [0,2] $.
Then
\begin{align}
\begin{split}
 &\big\|
    \big( -tA \big)^{-\kappa}
    \big(
      e^{tA}
      -
      \big( \Id_H - tA \big)^{-1}
    \big)
  \big\|_{ L( H ) }
  \leq
  1 .
\end{split}
\end{align}
\end{lemma}
\begin{proof}[Proof of Lemma~\ref{lem:limplicit_order_1}]
Observe that
the fundamental theorem of calculus shows that
for all $ v \in H $,
$ r \in [0,1] $
it holds that
\begin{align}
\begin{split}
 &\left\|
    \big( -tA \big)^{-r}
    \big(
      e^{tA}
      -
      \big( \Id_H - tA \big)^{-1}
    \big) v
  \right\|_{ H }
\\&=
  \left\|
    \big( -tA \big)^{-r}
    \big( \Id_H - tA \big)^{-1}
    \big(
      e^{tA}
      -
      tA e^{tA}
      -
      \Id_H
    \big) \, v
  \right\|_{ H }
\\&=
  \left\|
    \big( -tA \big)^{-r}
    \big( \Id_H - tA \big)^{-1}
    \left(
      \Id_H
      +
      \int_0^t
      A e^{uA} \, du
      -
      tA e^{tA}
      -
      \Id_H
    \right) v
  \right\|_{ H }
\\&=
  \left\|
    \big( -tA \big)^{-r}
    \big( \Id_H - tA \big)^{-1}
    \left(
      \int_0^t
      A 
      \big( e^{uA} - e^{tA} \big) \, du
    \right) v
  \right\|_{ H } .
\end{split}
\end{align}
This and the facts that
$ \forall \, s \in (0,\infty) $, 
$ r \in [0,1] 
  \colon
  \|
    (-sA)^{-r}
    (
      \Id_H
      -
      e^{sA}
    )
  \|_{ L(H) }
  \leq 1
$,
$ \forall \, s \in [0,\infty) $, 
$ r \in [0,1] 
  \colon
  \|
    (-sA)^r
    e^{sA}
  \|_{ L(H) }
  \leq 1
$,
and
$ \forall \, s \in [0,\infty) $, 
$ r \in [0,1] 
  \colon
  \|
    (-sA)^{r}
    (
      \Id_H
      -
      sA
    )^{-1}
  \|_{ L(H) }
  \leq 1
$
imply that
for all
$ v \in H $,
$ r \in [0,1] $
it holds that
\begin{align}
\label{eq:limplicit_order_1_1}
\begin{split}
 &\left\|
    \big( -tA \big)^{-r}
    \big(
      e^{tA}
      -
      \big( \Id_H - tA \big)^{-1}
    \big) v
  \right\|_{ H }
\\&\leq 
  \left\|
    \big( -tA \big)^{(1-r)}
    \big( \Id_H - tA \big)^{-1}
  \right\|_{ L( H ) }
  \left\|
    t^{-1}
    \left(
      \int_0^t
      e^{uA}
      \big( \Id_H - e^{(t-u)A} \big) \, du
    \right) v
  \right\|_{ H }
\\&\leq
  t^{-1}
  \int_0^t
  \left\|
    e^{uA}
  \right\|_{ L( H ) }
  \left\|
    \Id_H
    -
    e^{(t-u)A}
  \right\|_{ L( H ) }
  \left\|
    v
  \right\|_{ H } du
  \leq
  \left\|
    v
  \right\|_{ H } .
\end{split}
\end{align}
Moreover, note that
a Taylor expansion proves
that for all
$ v \in H $, $ r \in (1,2] $ it holds that
\begin{align}
\begin{split}
 &\left\|
    \big( -tA \big)^{-r}
    \big(
      e^{tA}
      -
      \big( \Id_H - tA \big)^{-1}
    \big) v
  \right\|_{ H }
\\&=
  \left\|
    \big( -tA \big)^{-r}
    \big( \Id_H - tA \big)^{-1}
    \big(
      e^{tA}
      -
      tA e^{tA}
      -
      \Id_H
    \big) \, v
  \right\|_{ H }
\\&=
  \left\|
    \big( -tA \big)^{-r}
    \big( \Id_H - tA \big)^{-1}
    \left(
      tA
      +
      \int_0^t
      A^2 e^{uA}
      \left(t-u\right) du
      -
      tA 
      \left(
        \Id_H
        +
        \int_0^t
        A e^{uA} \, du
      \right)
    \right) \, v
  \right\|_{ H }
\\&=
  \left\|
    \big( -tA \big)^{-r}
    \big( \Id_H - tA \big)^{-1}
    \left(
      -
      \int_0^t
      A^2 e^{uA}
      u \, du
    \right) v
  \right\|_{ H } .
\end{split}
\end{align}
The facts that
$ \forall \, s \in [0,\infty) $, 
$ r \in [0,1] 
  \colon
  \|
    (-sA)^{r}
    (
      \Id_H
      -
      sA
    )^{-1}
  \|_{ L(H) }
  \leq 1
$
and
$ \forall \, s \in [0,\infty) $, 
$ r \in [0,1] 
  \colon
  \|
    (-sA)^r
    e^{sA}
  \|_{ L(H) }
  \leq 1
$
hence yield that
for all
$ v \in H $,
$ r \in (1,2] $
it holds that
\begin{align}
\label{eq:limplicit_order_1_2}
\begin{split}
 &\left\|
    \big( -tA \big)^{-r}
    \big(
      e^{tA}
      -
      \big( \Id_H - tA \big)^{-1}
    \big) v
  \right\|_{ H }
\\&\leq
  \left\|
    \big( -tA \big)^{(2-r)}
    \big( \Id_H - tA \big)^{-1}
  \right\|_{ L( H ) }
  \left\|
    t^{-2}
    \left(
      \int_0^t
      e^{uA}
      u \, du
    \right) v
  \right\|_{ H }
\\&\leq
  t^{-2}
  \int_0^t
  u
  \left\|
    e^{uA}
  \right\|_{ L( H ) }
  \left\|
    v
  \right\|_{ H } du
  \leq
  \frac{1}{2}
  \left\|
    v
  \right\|_{ H } .
\end{split}
\end{align}
Combining~\eqref{eq:limplicit_order_1_1}
and~\eqref{eq:limplicit_order_1_2}
completes the proof of Lemma~\ref{lem:limplicit_order_1}.
\end{proof}

\begin{lemma}
\label{lem:limplicit_order_3}
Assume the setting in Section~\ref{sec:example_setting}
and let $ t \in (0,\infty) $, $ \kappa \in [0,1] $,
$ m \in \N $.
Then
\begin{equation}
  \left\|
    \left( - t A \right)^{ \kappa }
    \left(
      e^{ mtA }
      -
      \left( \Id_H - t A \right)^{ -m }
    \right)
  \right\|_{ L( H ) }
  \leq
  \frac{ 16 }{ m } .
\end{equation}
\end{lemma}
\begin{proof}[Proof of Lemma~\ref{lem:limplicit_order_3}]
First of all, note that the triangle inequality
and the facts that
$ \forall \, s \in [0,\infty) $, 
$ r \in [0,1]
  \colon
  \|
    (-sA)^r
    e^{sA}
  \|_{ L(H) }
  \leq 1
$
and
$ \forall \, s \in [0,\infty) $, 
$ r \in [0,1]
  \colon
  \|
    (-sA)^{r}
    (
      \Id_H
      -
      sA
    )^{-1}
  \|_{ L(H) }
  \leq 1
$
imply that
\begin{align}
\label{eq:limplicit_order_3_1}
\begin{split}
 &\left\|
    \left( - t A \right)^{ \kappa }
    \left(
      e^{ tA }
      -
      \left( \Id_H - t A \right)^{ -1 }
    \right)
  \right\|_{ L( H ) }
  \leq
  \left\| 
    \left( -t A \right)^{ \kappa }
    e^{ tA } 
  \right\|_{ L( H ) }
  +
  \left\| 
    \left( -t A \right)^{ \kappa }
    \left( \Id_H - t A \right)^{ -1 } 
  \right\|_{ L( H ) }
  \leq
  2 .
\end{split}
\end{align}
and
\begin{align}
\label{eq:limplicit_order_3_11}
\begin{split}
 &\left\|
    \left( - t A \right)^{ \kappa }
    \left(
      e^{ tA }
      +
      \left( \Id_H - t A \right)^{ -1 }
    \right)
  \right\|_{ L( H ) }
  \leq
  \left\| 
    \left( -t A \right)^{ \kappa }
    e^{ tA } 
  \right\|_{ L( H ) }
  +
  \left\| 
    \left( -t A \right)^{ \kappa }
    \left( \Id_H - t A \right)^{ -1 } 
  \right\|_{ L( H ) }
  \leq
  2 .
\end{split}
\end{align}
Next observe that
Lemma~\ref{lem:limplicit_order_1}
and~\eqref{eq:limplicit_order_3_11}
prove that
\begin{align}
\label{eq:limplicit_order_3_2}
\begin{split}
 &\left\|
    \left( - t A \right)^{ \kappa }
    \left(
      e^{ 2tA }
      -
      \left( \Id_H - t A \right)^{ -2 }
    \right)
  \right\|_{ L( H ) }
\\&=
  \left\|
    \left( - t A \right)^{ \kappa }
    \left(
      e^{ tA }
      +
      \left( \Id_H - t A \right)^{ - 1 }
    \right)
    \left(
      e^{ tA }
      -
      \left( \Id_H - t A \right)^{ -1 }
    \right)
  \right\|_{ L( H ) }
\\&\leq
  \left\|
    \left( -t A \right)^{ \kappa }
    \left(
      e^{ tA }
      +
      \left( \Id_H - t A \right)^{ -1 }
    \right)
  \right\|_{ L( H ) }
  \left\|
    e^{ tA }
    -
    \left( \Id_H - t A \right)^{ -1 }
  \right\|_{ L( H ) }
  \leq
  2 .
\end{split}
\end{align}
Moreover, note that 
Lemma~\ref{lem:limplicit_order_1}
ensures that for all $ j \in \{ 3,4,\ldots\} $
it holds that
\begin{align}
\begin{split}
 &\left\|
    \left( - t A \right)^{ \kappa }
    \left(
      e^{ jtA }
      -
      \left( \Id_H - t A \right)^{ -j }
    \right)
  \right\|_{ L( H ) }
\\&=
  \left\|
    \sum_{ l = 0 }^{ j - 1 }
    \left( - t A \right)^{ \kappa }
    e^{ ltA }
    \left( \Id_H - t A \right)^{ ( l - j + 1 ) }
    \left(
      e^{ tA }
      -
      \left( \Id_H - t A \right)^{ -1 }
    \right)
  \right\|_{ L( H ) }
\\&\leq
  t^{ \kappa }
  \sum_{ l = 0 }^{ j - 1 }
  \left\|
    \left( - A \right)^2
    e^{ ltA }
    \left( \Id_H - t A \right)^{ ( l - j + 1 ) }
  \right\|_{ L( H ) }
  \left\|
    \left( - A \right)^{ -(2-\kappa) }
    \left(
      e^{ tA }
      -
      \left( \Id_H - t A \right)^{ -1 }
    \right)
  \right\|_{ L( H ) }
\\&\leq
  t^2
  \sum_{ l = 0 }^{ j - 1 }
  \left\|
    \left( - A \right)^{ 2 }
    e^{ ltA }
    \left( \Id_H - t A \right)^{ ( l - j + 1 ) }
  \right\|_{ L( H ) } .
\end{split}
\end{align}
Lemma~\ref{lem:limplicit_bound}
and the fact that
$ \forall \, s \in [0,\infty) $, 
$ r \in [0,2]
  \colon
  \|
    (-sA)^r
    e^{sA}
  \|_{ L(H) }
  \leq 1
$
therefore yield
that for all
$ j \in \{3,4,\ldots\} $
it holds that
\begin{align}
\label{eq:limplicit_order_3_3}
\begin{split}
  &\left\|
    \left( - t A \right)^{ \kappa }
    \left(
      e^{ jtA }
      -
      \left( \Id_H - t A \right)^{ -j }
    \right)
  \right\|_{ L( H ) }
\\&\leq
  t^2
  \left[
    \sum_{ l = 0 }^{ \fl[1]{ \nicefrac{ j }{ 2 } } - 1 }
    \left\|
      \left( - A \right)^2
      \left( \Id_H - t A \right)^{ ( l - j + 1 ) }
    \right\|_{ L( H ) }
    +
    \sum_{ l = \fl[1]{ \nicefrac{ j }{ 2 } } }^{ j - 1 }
    \left\|
      \left( - A \right)^2
      e^{ ltA }
    \right\|_{ L( H ) }
  \right]
\\&\leq
  \sum_{ l = 0 }^{ \fl[1]{ \nicefrac{ j }{ 2 } } - 1 }
  \frac{ 4 t^2 }{ t^2 \left( j - l - 1 \right)^2 }
  +
  \sum_{ l = \fl[1]{ \nicefrac{ j }{ 2 } } }^{ j - 1 }
  \frac{ t^2 }{ t^2 l^2 } 
  \leq
  \sum_{ l = 0 }^{ \fl[1]{ \nicefrac{ j }{ 2 } } - 1 }
  \frac{ 4 }{ 
    \left( 
      j 
      - 
      \fl[1]{ \nicefrac{ j }{ 2 } } 
    \right)^2
  }
  +
  \sum_{ l = \fl[1]{ \nicefrac{ j }{ 2 } } }^{ j - 1 }
  \frac{ 1 }{ \left| \fl[1]{ \nicefrac{ j }{ 2 } } \right|^2 }
\\&=
  \frac{ 
    4  
    \fl[1]{ \nicefrac{ j }{ 2 } }
  }{ 
    \left| \cl[1]{ \nicefrac{ j }{ 2 } } \right|^2
  }
  +
  \frac{ 
    \left( 
      j
      - 
      \fl[1]{ \nicefrac{ j }{ 2 } }
    \right) 
  }{ 
    \left| \fl[1]{ \nicefrac{ j }{ 2 } } \right|^2
  }
  \leq
  \frac{ 4 }{ 
    \cl[1]{ \nicefrac{ j }{ 2 } }
  }
  +
  \frac{ 4 }{ 
    \cl[1]{ \nicefrac{ j }{ 2 } }
  }
  \leq
  \frac{ 16 }{ j } .
\end{split}
\end{align}
Combining \eqref{eq:limplicit_order_3_1},
\eqref{eq:limplicit_order_3_2},
and~\eqref{eq:limplicit_order_3_3} completes
the proof of Lemma~\ref{lem:limplicit_order_3}.
\end{proof}

\begin{lemma}
\label{lem:limplicit_order_4}
Assume the setting in Section~\ref{sec:example_setting}
and let $ q \in [2, \infty) $, 
$ \varrho \in [0,\infty) \cap (-\infty, 1+\nicefrac{d}{2q}-\nicefrac{d}{4} ) $. 
Then
\begin{multline}
  \sup_{M \in \N}
  \sup_{t \in (0,T] }
  M^{ \varrho }
  \int_0^t
  \left\|
    e^{(t-s)A}
    -
    \big(
      \Id_H - (t-\fl{t})A
    \big)^{-1}
    \big(
      \Id_H - \tfrac{T}{M}A
    \big)^{ (\fl{s}-\fl{t})\frac{M}{T} }
  \right\|_{ L( H, L^q(\lambda_{\D};\R) ) } ds
\\
  <
  \infty .
\end{multline}
\end{lemma}
\begin{proof}[Proof of Lemma~\ref{lem:limplicit_order_4}]
Throughout this proof let
$ R \colon [0,T] \rightarrow L(H) $
be the function with the property
that for all $ t \in [0,T] $
it holds that
$ R_t = (\Id_H - tA)^{-1} $.
Next note
that the triangle inequality
implies that 
for all $ M \in \N $, 
$ t \in (0,T] $
it holds that
\begin{align}
\label{eq:limplicit_order_4_h1}
\begin{split}
 &\int_0^t
  \left\|
    e^{(t-s)A}
    -
    R_{(t-\fl{t})}
    \left[ R_{\,\nicefrac{T}{M}} \right]^{ (\fl{t}-\fl{s})\frac{M}{T} }
  \right\|_{ L( H, L^q(\lambda_{\D};\R) ) } ds
\\&\leq
  \int_0^t
  \left\|
    e^{(t-s)A}
    -
    e^{(t-\fl{s})A}
  \right\|_{ L( H, L^q(\lambda_{\D};\R) ) } ds
\\&+
  \int_0^t
  \left\|
    e^{(t-\fl{s})A}
    -
    R_{(t-\fl{t})}
    \left[ R_{\,\nicefrac{T}{M}} \right]^{ (\fl{t}-\fl{s})\frac{M}{T} }
  \right\|_{ L( H, L^q(\lambda_{\D};\R) ) } ds
\\&\leq
  \int_0^t
  \left\|
    e^{(t-s)A}
    -
    e^{(t-\fl{s})A}
  \right\|_{ L( H, L^q(\lambda_{\D};\R) ) } ds
\\&+
  \int_{\max\{0,\fl{t}-\frac{T}{M}\}}^t
  \left\|
    e^{(t-\fl{s})A}
    -
    R_{(t-\fl{t})}
    \left[ R_{\,\nicefrac{T}{M}} \right]^{ (\fl{t}-\fl{s})\frac{M}{T} }
  \right\|_{ L( H, L^q(\lambda_{\D};\R) ) } ds
\\&+
  \int_0^{\max\{0,\fl{t}-\frac{T}{M}\}}
  \left\|
    e^{(\fl{t}-\fl{s})A}
    \left(
      e^{(t-\fl{t})A}
      -
      R_{(t-\fl{t})}
    \right)
  \right\|_{ L( H, L^q(\lambda_{\D};\R) ) } ds
\\&+
  \int_0^{\max\{0,\fl{t}-\frac{T}{M}\}}
  \left\|
    R_{(t-\fl{t})}
    \left(
      e^{(\fl{t}-\fl{s})A}
      -
      \left[ R_{\,\nicefrac{T}{M}} \right]^{ (\fl{t}-\fl{s})\frac{M}{T} }
    \right)
  \right\|_{ L( H, L^q(\lambda_{\D};\R) ) } ds .
\end{split}
\end{align}
Moreover, the triangle inequality
and Lemma~\ref{lem:semigroup_bound} prove
that
\begin{align}
\label{eq:limplicit_order_4_h2}
\begin{split}
 &\sup_{ M \in \N }
  \sup_{ t \in (0, T] }
  M^{\varrho}
  \int_{\max\{0,\fl{t}-\frac{T}{M}\}}^t
  \left\|
    e^{(t-\fl{s})A}
    -
    R_{(t-\fl{t})}
     \left[ R_{\,\nicefrac{T}{M}} \right]^{ (\fl{t}-\fl{s})\frac{M}{T} }
  \right\|_{ L( H, L^q(\lambda_{\D};\R) ) } ds
\\&\leq
  \left[ 
    \sup_{ M \in \N }
    \sup_{ t \in (0, T] }
    M^{\varrho}
    \int_{\max\{0,\fl{t}-\frac{T}{M}\}}^t
    \left(t-s\right)^{-( \nicefrac{d}{4} - \nicefrac{d}{2q} ) } ds
  \right]
\\&\quad\cdot
  \Bigg[
    \sup_{ M \in \N }
    \sup_{ 0<s<t\leq T }
    \left(t-s\right)^{ ( \nicefrac{d}{4} - \nicefrac{d}{2q} ) }
    \left\|
      e^{(t-\fl{s})A}
      -
      R_{(t-\fl{t})}
      \left[ R_{\,\nicefrac{T}{M}} \right]^{ (\fl{t}-\fl{s})\frac{M}{T} }
    \right\|_{ L( H, L^q(\lambda_{\D};\R) ) }
  \Bigg]
\\&\leq
  \Bigg[ 
    \sup_{ M \in \N }
    \frac{ 
      2 \, T^{(1+\nicefrac{d}{2q}-\nicefrac{d}{4} ) }  M^{\varrho} 
    }{
      \left( 1 + \nicefrac{d}{2q} - \nicefrac{d}{4} \right) M^{ (1 + \nicefrac{d}{2q} -  \nicefrac{d}{4}) }
    }
  \Bigg]
  \Bigg[
    \sup_{ s \in (0,T] }
    s^{ ( \nicefrac{d}{4} - \nicefrac{d}{2q} ) }
    \left\|
      e^{sA}
    \right\|_{ L( H, L^q(\lambda_{\D};\R) ) }
\\&\quad+
    \sup_{ M \in \N }
    \sup_{ 0<s<t\leq T }
    \left(t-s\right)^{ ( \nicefrac{d}{4} - \nicefrac{d}{2q} ) }
    \left\|
      R_{(t-\fl{t})}
      \left[ R_{\,\nicefrac{T}{M}} \right]^{ (\fl{t}-\fl{s})\frac{M}{T} }
    \right\|_{ L( H, L^q(\lambda_{\D};\R) ) }
  \Bigg]
  <
  \infty .
\end{split}
\end{align}
In the next step observe that 
the fact that 
$ \forall \, s \in (0,\infty) $, 
$ r \in [0,1] 
  \colon
  \|
    (-sA)^r
    e^{sA}
  \|_{ L(H) }
  \leq 1
$
and Lemma~\ref{lem:limplicit_order_1}
assure that for
all $ t \in (0,T] $, $ M \in \N $
it holds that
\begin{align}
\begin{split}
 &\int_{0}^{\max\{0,\fl{t}-\frac{T}{M}\}}
  \left\|
    e^{(\fl{t}-\fl{s})A}
    \left(
      e^{(t-\fl{t})A}
      -
      R_{(t-\fl{t})}
    \right)
  \right\|_{ L( H, L^q(\lambda_{\D};\R) ) } ds
\\&\leq
  \int_{0}^{\max\{0,\fl{t}-\frac{T}{M}\}}
  \left\|
    e^{\frac{1}{2}(\fl{t}-\fl{s})A}
  \right\|_{ L( H, L^q(\lambda_{\D};\R) ) }
\\&\quad\cdot
  \left\|
    e^{\frac{1}{2}(\fl{t}-\fl{s})A}
    \left(
      e^{(t-\fl{t})A}
      -
      R_{(t-\fl{t})}
    \right)
  \right\|_{ L( H ) } ds
\\&\leq
  \left[
    \sup_{ s \in (0,T] }
    s^{(\nicefrac{d}{4} - \nicefrac{d}{2q})}
    \left\|
      e^{sA}
    \right\|_{ L( H, L^q(\lambda_{\D};\R) ) }  
  \right]
  \int_{0}^{\max\{0,\fl{t}-\frac{T}{M}\}}
  \left[
    \tfrac{1}{2}
    (\fl{t}-\fl{s})
  \right]^{-(\nicefrac{d}{4} - \nicefrac{d}{2q})}
\\&\quad\cdot
  \left\|
    \left( -A \right)
    e^{\frac{1}{2}(\fl{t}-\fl{s})A}
  \right\|_{ L(H) }
  \left\|
    \left(
      -A
    \right)^{-1}
    \left(
      e^{(t-\fl{t})A}
      -
      R_{(t-\fl{t})}
    \right)
  \right\|_{ L( H ) } ds
\\&\leq
  \frac{T}{M}
  \left[
    \sup_{ s \in (0,T] }
    s^{(\nicefrac{d}{4} - \nicefrac{d}{2q})}
    \left\|
      e^{sA}
    \right\|_{ L( H, L^q(\lambda_{\D};\R) ) }  
  \right]
  \int_{0}^{\max\{0,\fl{t}-\frac{T}{M}\}}
  \left[
    \tfrac{1}{2}
    (\fl{t}-\fl{s})
  \right]^{-(1+\nicefrac{d}{4} - \nicefrac{d}{2q})} ds .
\end{split}
\end{align}
Lemma~\ref{lem:semigroup_bound}
and the fact that 
$
  \forall \, M \in \N
$, 
$  
  \varepsilon \in (0,\infty)
  \colon
  \sum_{ l=2 }^M
  l^{-(1+\nicefrac{d}{4} - \nicefrac{d}{2q} + \varepsilon)}
  \leq
  \int_1^M
  s^{-(1+\nicefrac{d}{4} - \nicefrac{d}{2q} + \varepsilon)} \, ds
  =
  \frac{ 
    1 - M^{-(\nicefrac{d}{4} - \nicefrac{d}{2q} + \varepsilon)} 
  }{ 
    \nicefrac{d}{4} - \nicefrac{d}{2q}  + \varepsilon 
  }
$
hence
imply that for all
$ \varepsilon \in (0, 1 +\nicefrac{d}{2q}-\nicefrac{d}{4}-\varrho] $
it holds that
\begin{align}
\label{eq:limplicit_order_4_h3}
\begin{split}
 &\sup_{ M \in \N }
  \sup_{ t \in (0,T] }
  M^{\varrho}
  \int_{0}^{\max\{0,\fl{t}-\frac{T}{M}\}}
  \left\|
    e^{(\fl{t}-\fl{s})A}
    \left(
      e^{(t-\fl{t})A}
      -
      R_{(t-\fl{t})}
    \right)
  \right\|_{ L( H, L^q(\lambda_{\D};\R) ) } ds
\\&\leq
  \left[
    \sup_{ s \in (0,T] }
    s^{(\nicefrac{d}{4} - \nicefrac{d}{2q})}
    \left\|
      e^{sA}
    \right\|_{ L( H, L^q(\lambda_{\D};\R) ) }  
  \right]
\\&\cdot
  \left[
    \sup_{ M \in \N }
    \sup_{ t \in (0,T] }
    \frac{ 4 \, T }{ M^{(1-\varrho)} }
    \sum_{ l=0 }^{ \fl{t}\tfrac{M}{T}-2}
    \int_{ \frac{lT}{M} }^{ \frac{(l+1)T}{M} }
    \left[
      \big(
	\fl{t} \tfrac{M}{T}
	-
	l
      \big) \tfrac{T}{M}
    \right]^{-(1+\nicefrac{d}{4} - \nicefrac{d}{2q})} ds
  \right] 
\\&\leq
  \left[
    \sup_{ s \in (0,T] }
    s^{(\nicefrac{d}{4} - \nicefrac{d}{2q})}
    \left\|
      e^{sA}
    \right\|_{ L( H, L^q(\lambda_{\D};\R) ) }  
  \right]
  \left[
    \sup_{ M \in \N }
    \sup_{ t \in (0,T] }
    \frac{ 
      4 \, 
      T^{(1-\nicefrac{d}{4} + \nicefrac{d}{2q})} 
    }{ 
      M^{(1-\varrho-\nicefrac{d}{4} + \nicefrac{d}{2q})} 
    }
    \sum_{ l=2 }^{ \fl{t}\tfrac{M}{T} }
    \frac{
      l^{\varepsilon}
    }{
      l^{(1+\nicefrac{d}{4} - \nicefrac{d}{2q} + \varepsilon) }
    }
  \right]
\\&\leq
  \left[
    \sup_{ s \in (0,T] }
    s^{(\nicefrac{d}{4} - \nicefrac{d}{2q})}
    \left\|
      e^{sA}
    \right\|_{ L( H, L^q(\lambda_{\D};\R) ) }  
  \right]
  \left[
    \sup_{ M \in \N }
    \frac{ 
      4 \, 
      T^{(1-\nicefrac{d}{4} + \nicefrac{d}{2q})} \,
      (1 - M^{-(\nicefrac{d}{4} - \nicefrac{d}{2q} + \varepsilon)})
    }{ 
      M^{(1-\varrho-\nicefrac{d}{4} + \nicefrac{d}{2q}-\varepsilon)} \,
      (\nicefrac{d}{4} - \nicefrac{d}{2q} + \varepsilon)
    }
  \right]
  <
  \infty .
\end{split}
\end{align}
Furthermore,
observe that the fact that 
for all
$ r \in [0,\infty) $
it holds that
$
  H_r \subseteq W^{2r,2}(\D,\R)
$
continuously
(cf., e.g., (A.46) in Da Prato \& Zabczyk \cite{dz92} and Lunardi \cite{l09})
and the fact that for all $ r \in [\nicefrac{d}{4} - \nicefrac{d}{2q},\infty) $ 
it holds that 
$ 
  W^{2r,2}(\D,\R) 
  \subseteq 
  W^{0,q}(\D,\R) 
  = 
  L^q(\lambda_{\D};\R) 
$
continuously imply that
there exists a real number
$ C \in (0,\infty) $ 
such that for all $ v \in H_{(\nicefrac{d}{4} - \nicefrac{d}{2q})} $
it holds that 
$ \| v \|_{ L^q(\lambda_{\D};\R) } \leq C \| v \|_{ H_{(\nicefrac{d}{4} - \nicefrac{d}{2q})} } $.
This and the fact that
$ \forall \, s \in (0,\infty) $, 
$ r \in [0,1]  
  \colon
  \|
    (-sA)^r
    R_s
  \|_{ L(H) }
  \leq 1
$
ensure that for all
$ M \in \N $, $ s,t \in (0,T] $
with $ s < \fl{t}-\nicefrac{T}{M} $
it holds that
\begin{align}
\begin{split}
 &\sup_{ v \in H \backslash \{0\} }
  \left[ 
    \frac{1}{ \left\| v \right\|_H }
    \left\|
      R_{(t-\fl{t})}
      \left(
	e^{(\fl{t}-\fl{s})A}
	-
	\left[ R_{\,\nicefrac{T}{M}} \right]^{ (\fl{t}-\fl{s})\frac{M}{T} }
      \right)
      v
    \right\|_{ L^q(\lambda_{\D};\R) }
  \right]
\\&\leq
  \sup_{ v \in H \backslash \{0\} }
  \left[ 
    \frac{C}{ \left\| v \right\|_H }
    \left\|
      R_{(t-\fl{t})}
      \left(
	e^{(\fl{t}-\fl{s})A}
	-
	\left[ R_{\,\nicefrac{T}{M}} \right]^{ (\fl{t}-\fl{s})\frac{M}{T} }
      \right)
      v
    \right\|_{ H_{ (\nicefrac{d}{4} - \nicefrac{d}{2q}) } }
  \right]
\\&\leq
  \sup_{ v \in H \backslash \{0\} }
  \left[ 
    \frac{
	C \,
      \big\|
	R_{(t-\fl{t})}
      \big\|_{ L(H) }
    }{ \left\| v \right\|_H }
    \left\|
      \left(-A\right)^{(\nicefrac{d}{4} - \nicefrac{d}{2q})}
      \left(
	e^{(\fl{t}-\fl{s})A}
	-
	\left[ R_{\,\nicefrac{T}{M}} \right]^{ (\fl{t}-\fl{s})\frac{M}{T} }
      \right)
      v
    \right\|_{ H }
  \right]
\\&\leq
  C
  \left\|
    \left(-A\right)^{(\nicefrac{d}{4} - \nicefrac{d}{2q})}
    \left(
      e^{(\fl{t}-\fl{s})A}
      -
      \left[ R_{\,\nicefrac{T}{M}} \right]^{ (\fl{t}-\fl{s})\frac{M}{T} }
    \right)
  \right\|_{ L( H ) } .
\end{split}
\end{align}
Lemma~\ref{lem:limplicit_order_3} 
hence shows that
for all $ \varepsilon \in (0, 1 +\nicefrac{d}{2q}-\nicefrac{d}{4}-\varrho] $
it holds that
\begin{align}
\label{eq:limplicit_order_4_h4}
\begin{split}
 &\sup_{ M \in \N }
  \sup_{ t \in (0, T] }
  M^{\varrho} \!\!\!
  \int\limits_0^{\max\{0,\fl{t}-\frac{T}{M}\}} 
  \left\|
    R_{(t-\fl{t})}
    \left(
      e^{(\fl{t}-\fl{s})A}
      -
      \left[ R_{\,\nicefrac{T}{M}} \right]^{ (\fl{t}-\fl{s})\frac{M}{T} }
    \right)
  \right\|_{ L( H, L^q(\lambda_{\D};\R) ) } ds
\\&\leq
  \sup_{ M \in \N }
  \sup_{ l \in \{1,2,\ldots,M\} }
  M^{\varrho}
  \sum_{ k=0 }^{ l-2 } 
  \int_{\frac{kT}{M}}^{\frac{(k+1)T}{M}}
  C
  \left\|
    \left(-A\right)^{(\nicefrac{d}{4} - \nicefrac{d}{2q})}
    \left(
      e^{\frac{(l-k)T}{M}A}
      -
      \left[ R_{\,\nicefrac{T}{M}} \right]^{ (l-k) }
    \right)
  \right\|_{ L( H ) } ds
\\&\leq
  \sup_{ M \in \N }
  \sup_{ l \in \{1,2,\ldots,M\} }
  \left[ 
    M^{\varrho}
    \sum_{ k=0 }^{ l-2 } 
    \frac{
      16 \, C \, T^{(1+\nicefrac{d}{2q} -\nicefrac{d}{4})}
    }{
      (l-k) \, M^{(1+\nicefrac{d}{2q}-\nicefrac{d}{4})}
    }
  \right]
  =
  \sup_{ M \in \{2,3,\ldots\} }
  \sup_{ l \in \{2,3,\ldots,M\} }
  \left[ 
    \sum_{ k=2 }^{ l } 
    \frac{
      16 \, C \, T^{(1+\nicefrac{d}{2q}-\nicefrac{d}{4})}
    }{
      k \, M^{(1+\nicefrac{d}{2q}-\nicefrac{d}{4}-\varrho)}
    }
  \right]
\\&\leq
  \sup_{ M \in \N }
  \left[ 
    \frac{
      16 \, C \, T^{(1+\nicefrac{d}{2q}-\nicefrac{d}{4})}
    }{
      M^{(1+\nicefrac{d}{2q}-\nicefrac{d}{4}-\varrho)}
    }
    \int_1^M
    \frac{1}{s^{(1-\varepsilon)}} \, ds
  \right]
  =
  \sup_{ M \in \N }
  \left[ 
    \frac{
      16 \, C \, T^{(1+\nicefrac{d}{2q}-\nicefrac{d}{4})}
      \left( M^{\varepsilon} -1 \right)
    }{
      \varepsilon \,
      M^{(1+\nicefrac{d}{2q}-\nicefrac{d}{4}-\varrho)}
    }
  \right] 
\\&\leq
  \frac{
    16 \, C \, T^{(1+\nicefrac{d}{2q}-\nicefrac{d}{4})}
  }{
    \varepsilon
  }
  <
  \infty .
\end{split}
\end{align}
Combining Lemma~\ref{lem:exponential_order},
\eqref{eq:limplicit_order_4_h1}, \eqref{eq:limplicit_order_4_h2},
\eqref{eq:limplicit_order_4_h3}, and \eqref{eq:limplicit_order_4_h4}
completes the proof of Lemma~\ref{lem:limplicit_order_4}.
\end{proof}

\begin{lemma}
\label{lem:wiener}
Assume the setting in Section~\ref{sec:example_setting}, let
$ p \in [1,\infty) $, $ \theta \in (0,\nicefrac{1}{4}) $, $ \D = (0,1) $,
$ \xi \in \L^p(\P;V) $,
and let $ (W_t)_{ t \in [0,T] } $ be an 
$ \Id_H $-cylindrical $ (\F_t)_{ t \in [0,T] } $-Wiener process.
Then there exists an up to indistinguishability unique
stochastic process $ O \colon [0,T] \times \Omega \rightarrow V $ 
with continuous sample paths which satisfies
\begin{enumerate}[(i)]
 \item\label{it:wiener_1}
  that for all $ t \in [0,T] $ it holds that
  $ 
    [ O_t - e^{tA} \xi ]_{ \P, \B(H) }
    = 
    \int_0^t
    e^{(t-s)A} \, dW_s
  $,
  \item\label{it:wiener_2}
  that for all
  $ \omega \in \Omega $,
  $ r \in [0,\theta] $ it holds that
  $
    \sup_{0\leq s < t \leq T}
    \frac{ 
      s\| O_t(\omega) - O_s(\omega) \|_V
    }{ 
      (t-s)^r
    }
    <
    \infty
  $,
  \item\label{it:wiener_3}
  and that
  $
    \sup_{ t \in [0,T] }
    \sup_{ M\in \N }	
    \big\|
      \|
	O_t
      \|_{ V }
      +
      |M\fl{t}|^{\theta} \,
      \|
	O_t - O_{\fl{t}}
      \|_{ V }
    \big\|_{ \L^p(\P;\R) }
    <
    \infty
  $.
\end{enumerate}
\end{lemma}
\begin{proof}[Proof of Lemma~\ref{lem:wiener}]
Throughout this proof
let $ \tilde{O} \colon [0,T] \times \Omega \rightarrow V $
be an up to indistinguishability unique
stochastic process which satisfies
for all $ t \in [0,T] $ that
$ 
  [ \tilde{O}_t ]_{ \P, \B(H) }
  = 
  \int_0^t
  e^{(t-s)A} \, dW_s
$,
which satisfies
for all $ \omega \in \Omega $,
$ r \in (0,\theta] $ that
$
  \sup_{ 0\leq s < t \leq T }
  (t-s)^{-r} \,
  \|
    \tilde{O}_t(\omega) - \tilde{O}_s(\omega)
  \|_V
  <
  \infty
$,
and which satisfies
for all $ r \in (0,\theta] $ that
$
  \sup_{ 0\leq s < t \leq T }
  (t-s)^{-r} \,
  \|
    \tilde{O}_t
    - 
    \tilde{O}_s
  \|_{\L^p(\P;V)}
  <
  \infty
$.
Observe that, e.g., Lemma~4.3 in Bl{\"o}mker \& Jentzen~\cite{bj09}
ensures that such a stochastic process
does indeed exist.
Next let $ O \colon [0,T] \times \Omega \rightarrow V $
be the stochastic process with the property that
for all $ t \in [0,T] $ it holds that
$ O_t = e^{tA}\xi + \tilde{O}_t $.
Note that the fact that for all $ \omega \in \Omega $
it holds that the function $ [0,T] \ni t \mapsto \tilde{O}_t(\omega) \in V $
is H{\"o}lder continuous and the fact that for all $ \omega \in \Omega $
it holds that the function $ [0,T] \ni t \mapsto e^{tA} \xi(\omega) \in V $
is continuous imply that for all $ \omega \in \Omega $
it holds that the function $ [0,T] \ni t \mapsto O_t(\omega) \in V $
is continuous.
Moreover, observe that for all $ t \in [0,T] $ it holds
that 
$ 
  [O_t - e^{tA}\xi]_{\P,\B(H)} 
  =
  [\tilde{O}_t]_{\P,\B(H)} 
  =
  \int_0^t
  e^{(t-s)A} \, dW_s
$.
This implies~\eqref{it:wiener_1}.
In addition, observe that 
Lemma~\ref{lem:semigroup_bound},
the fact that
$
  \forall \,
  s \in [0,\infty)
$,
$ 
  r \in [0,1]
  \colon
  \| (-sA)^r \, e^{sA} \|_{L(H)}
  \leq 1  
$,
and the fact that
$
  \forall \,
  s \in (0,\infty)
$,
$ 
  r \in [0,1]
  \colon
  \| (-sA)^{-r} \, (e^{sA}-\Id_H) \|_{L(H)}
  \leq 1  
$
assure that for all
$ \omega \in \Omega $,
$ r \in (0,\theta] $
it holds that
\begin{align}
\label{eq:wiener_11}
\begin{split}
 &\sup_{ 0 < s < t\leq T }
  \frac{
   s
   \left\|
     O_t(\omega)
     -
     O_s(\omega)
   \right\|_V
  }{
    \left(t-s\right)^r
  }
  =
  \sup_{ 0 < s < t\leq T }
  \frac{
   s \,
   \big\|
     e^{tA}
     \xi(\omega)
     +
     \tilde{O}_t(\omega)
     -
     e^{sA}
     \xi(\omega)
     -
     \tilde{O}_s(\omega)
   \big\|_V
  }{
    (t-s)^r
  }
\\&\leq
  \sup_{ 0 < s<t\leq T }
  \left[ 
    \frac{
      s \,
      \big\|
	e^{sA}
	\left(
	  e^{(t-s)A}
	  -
	  \Id_H
	\right)
	\xi(\omega)
      \big\|_V
    }{
      \left(t-s\right)^r
    }
    +
    \frac{
      s \,
      \big\|
	\tilde{O}_t(\omega)
	-
	\tilde{O}_s(\omega)
      \big\|_V
    }{
      \left(t-s\right)^r
    }
  \right]
\\&\leq 
  \sup_{ 0 < s<t\leq T }
  \left[ 
    \frac{ 
      s
      \left\|
	e^{\nicefrac{s}{2}A}
      \right\|_{ L(H,V) }
      \left\|
	e^{\nicefrac{s}{2}A}
	\left(
	  e^{(t-s)A}
	  -
	  \Id_H
	\right)
      \right\|_{ L(H) }
      \left\|
	\xi(\omega)
      \right\|_{ H }
    }{
      \left( t-s \right)^r 
    }
    +
    \frac{
      T \,
      \big\|
	\tilde{O}_t(\omega)
	-
	\tilde{O}_s(\omega)
      \big\|_V
    }{
      \left(t-s\right)^r
    }
  \right]
\\&\leq
  2
  \left[
    \sup_{ s \in (0,T] }
    s^{(\nicefrac{1}{4}-\nicefrac{1}{4n^2})}
    \left\|
      e^{sA}
    \right\|_{ L(H,V) }
  \right]
  \Bigg[
    \sup_{ 0 < s<t\leq T }
    \left( t - s \right)^{ (1-\nicefrac{1}{4}+\nicefrac{1}{4n^2} - r) } 
    \left\|
      \left(
	-
	\tfrac{s}{2}
	A
      \right)^{(1-\nicefrac{1}{4}+\nicefrac{1}{4n^2}) }
      e^{\nicefrac{s}{2}A}
    \right\|_{ L(H) }
\\&\quad\cdot
    \left\|
      \left( -\left(t-s\right)A \right)^{-(1-\nicefrac{1}{4}+\nicefrac{1}{4n^2}) }
      \left(
	e^{(t-s)A}
	-
	\Id_H
      \right)
    \right\|_{ L(H) }
    \left\|
      \xi(\omega)
    \right\|_{ H }
  \Bigg]
  +
  \sup_{ 0 < s<t\leq T }
  \frac{
    T \,
    \big\|
      \tilde{O}_t(\omega)
      -
      \tilde{O}_s(\omega)
    \big\|_V
  }{
    \left(t-s\right)^r
  }
\\&\leq
  2
  \max\!\left\{1,T\right\}
  \left[
    \sup_{ s \in (0,T] }
    s^{(\nicefrac{1}{4}-\nicefrac{1}{4n^2})}
    \left\|
      e^{sA}
    \right\|_{ L(H,V) }
  \right]
  \left\|
    \xi(\omega)
  \right\|_{ H }
  +
  \sup_{ 0 < s<t\leq T }
  \frac{
    T \,
    \big\|
      \tilde{O}_t(\omega)
      -
      \tilde{O}_s(\omega)
    \big\|_V
  }{
    \left(t-s\right)^r
  }
  <
  \infty .
\end{split}
\end{align}
This shows~\eqref{it:wiener_2}.
Furthermore, note that Lemma~\ref{lem:U_lem1}
implies that
\begin{align}
\label{eq:wiener_1}
\begin{split}
 &\sup_{ t \in [0,T] }
  \left\|
    O_t
  \right\|_{ \L^{p}(\P;V) }
  =
  \sup_{ t \in [0,T] }
  \big\|
    e^{tA} \xi
    +
    \tilde{O}_t
  \big\|_{ \L^{p}(\P;V) }
\\&\leq
  \left[ 
    \sup_{ t \in [0,T] }
    \left\|
      e^{tA}
    \right\|_{ L(V) }
  \right]
  \left\|
    \xi
  \right\|_{ \L^{p}(\P;V) }
  +
  \sup_{ t \in (0,T] }
  \big\|
    \tilde{O}_t
  \big\|_{ \L^{p}(\P;V) }
  \leq
  \left\|
    \xi
  \right\|_{ \L^{p}(\P;V) }
  +
  \sup_{ t \in (0,T] }
  \frac{
    T^{\theta}\,
    \|
      \tilde{O}_t
    \|_{ \L^{p}(\P;V) }
  }{
    t^{\theta}
  }
  <
  \infty .
\end{split}
\end{align}
In the next step observe that
\begin{align}
\begin{split}
 &\sup_{M \in \N}
  \sup_{ t\in [0,T] }
  \Big[
    |M\fl{t}|^{\theta} \,
    \big\|
      O_t 
      -
      O_{\fl{t}} 
    \big\|_{ \L^{p}(\P;V) }
  \Big]
\\&=
  \sup_{M \in \N}
  \sup_{ t\in [0,T] }
  \Big[
    |M\fl{t}|^{\theta} \,
    \big\|
      e^{tA} \xi
      +
      \tilde{O}_t 
      -
      e^{\fl{t}A} \xi
      -
      \tilde{O}_{\fl{t}} 
    \big\|_{ \L^{p}(\P;V) }
  \Big]
\\&\leq
  \sup_{M \in \N}
  \sup_{ \substack {t\in[0,T] \backslash \\  \{0,\frac{T}{M},\ldots,T\} } }
  \left[ 
    |M\fl{t}|^{\theta} \,
    (t-\fl{t})^{\theta} \,
    \frac{ 
      \big\|
	e^{\fl{t}A}
	\big(
	  e^{(t-\fl{t})A}
	  -
	  \Id_H
	\big)
      \big\|_{ L(V) }
    }{
      (t-\fl{t})^{\theta}
    } 
    \left\|
      \xi
    \right\|_{ \L^{p}(\P;V) }
  \right]
\\&\quad+
  \sup_{M \in \N}
  \sup_{ \substack {t\in[0,T] \backslash \\  \{0,\frac{T}{M},\ldots,T\} } }
  \left[ 
    |M T|^{\theta} \,
    (t-\fl{t})^{\theta} \,  
    \frac{ 
      \|
	\tilde{O}_t 
	-
	\tilde{O}_{\fl{t}} 
      \|_{ \L^{p}(\P;V) }
    }{
      (t-\fl{t})^{\theta}
    }
  \right] .
\end{split}
\end{align}
This and, e.g., Lemma~12.36 in 
Renardy \& Rogers~\cite{rr04}
prove that
\begin{align}
\label{eq:wiener_2}
\begin{split}
 &\sup_{M \in \N}
  \sup_{ t\in [0,T] }
  \Big[
    |M\fl{t}|^{\theta} \,
    \big\|
      O_t 
      -
      O_{\fl{t}} 
    \big\|_{ \L^{p}(\P;V) }
  \Big]
\\&\leq
  T^{\theta}
  \Bigg[
    \sup_{M \in \N}
    \sup_{ \substack {t\in[0,T] \backslash \\  \{0,\frac{T}{M},\ldots,T\} } }
    \big\|
      (-\fl{t}A)^{\theta} \,
      e^{\fl{t}A}
    \big\|_{ L(V) }
\\&\quad\cdot
    \big\|
      (-(t-\fl{t})A)^{-\theta} \,
      \big(
	e^{(t-\fl{t})A}
	-
	\Id_H
      \big)
    \big\|_{ L(V) }
  \Bigg]
  \left\|
    \xi
  \right\|_{ \L^{p}(\P;V) }
  +
  \sup_{ 0 \leq s < t \leq T }
  \left[ 
    \frac{ 
      T^{2\theta} \,
      \|
	\tilde{O}_t 
	-
	\tilde{O}_s 
      \|_{ \L^{p}(\P;V) }
    }{
      (t-s)^{\theta}
    }
  \right]
\\&\leq
  T^{\theta}
  \left[
    \sup_{ s\in (0,T) } 
    \big\|
      (-sA)^{\theta}  \,
      e^{sA}
    \big\|_{ L(V) }
  \right]
  \left[
    \sup_{ s\in (0,T) } 
    \big\|
      (-sA)^{-\theta} \,
      \big(
	e^{sA}
	-
	\Id_H
      \big)
    \big\|_{ L(V) }
  \right]
  \left\|
    \xi
  \right\|_{ \L^{p}(\P;V) }
\\&\quad+
  \sup_{ 0 \leq s < t \leq T }
  \left[ 
    \frac{
      T^{2\theta} \,
      \|
	\tilde{O}_t
	-
	\tilde{O}_s 
      \|_{ \L^{p}(\P;V) }
    }{
      (t-s)^{\theta}
    }
  \right]
  <
  \infty .
\end{split}
\end{align}
Combining~\eqref{eq:wiener_1} and~\eqref{eq:wiener_2}
shows~\eqref{it:wiener_3}. The proof of Lemma~\ref{lem:wiener} is thus
completed.
\end{proof}

\subsection{Strong convergence rates for a nonlinearity-truncated
exponential Euler scheme}

\begin{corollary}
\label{cor:exp_euler_convergence}
Assume the setting in Section~\ref{sec:example_setting}, let
$ \vartheta \in (0,\infty) $, 
$ \chi \in (0,\tfrac{1}{2n}] $, $ p \in [2,\infty) $, $ \D = (0,1) $,
$ \xi \in \L^{(3n-1)p\max\{2,2(n-1),2\vartheta,2(n-1)\vartheta\}}\!(\P; V) $,
let $ (W_t)_{ t \in [0,T] } $ be an 
$ \Id_H $-cylindrical $ (\F_t)_{ t \in [0,T] } $-Wiener process,
let $ X \colon [0,T] \times \Omega \rightarrow V $ 
be a stochastic process with continuous sample paths which
satisfies for all $ t \in [0,T] $ that
$ 
  \big[
    X_t 
    -
    e^{tA} \xi
    - 
    \int_0^t 
    e^{(t-s)A} 
    F(X_s) \, ds
  \big]_{\P,\B(H)}
  =
  \int_0^t
  e^{(t-s)A} \, dW_s
$,
and let $ \YM{} \colon [0,T] \times \Omega \rightarrow V $, $ M \in \N $,
be stochastic processes which satisfy for all $ t \in [0,T] $, $ M \in \N $ that
\begin{align}
\begin{split}
 &\left[ 
    \YM{t}
    -
    e^{tA} \xi
    -
    \smallint_0^t
    e^{(t-\fl{s})A} \,
    \one_{ 
      \{
	\| \YM{\fl{s}} \|_{\BaMi}
	\leq
	(M/T)^{\chi}
      \}
    } \,
    F(\YM{\fl{s}}) \, ds
  \right]_{\P,\B(H)}
  =
  \smallint_0^t
  e^{(t-\fl{s})A} \, dW_s .
\end{split}
\end{align}
Then it holds for all $ \theta \in [0,\nicefrac{1}{4}) $ that
\begin{align}
\label{eq:exp_euler_convergence}
\begin{split}
 &\sup_{M \in \N}
  \sup_{t \in [0,T]}
  \left[
    \big\|
      \YM{t}
    \big\|_{ \L^{p\max\{2, 2(n-1), 4\vartheta,4(n-1)\vartheta\}}\!(\P;V) }
    +
    \big\|
      X_t
    \big\|_{ \L^{p}(\P;H) }
    +
    M^{\min\{\vartheta\chi,\theta\}} \,
    \big\|
      X_t 
      -
      \YM{t}
    \big\|_{ \L^p(\P;H) }
  \right]
  <
  \infty .
\end{split}
\end{align}
\end{corollary}
\begin{proof}[Proof of Corollary~\ref{cor:exp_euler_convergence}]
Throughout this proof let
$ \OM{} \colon [0,T] \times \Omega \rightarrow V $, $ M \in \N $, 
be stochastic processes with the property that for all
$ t \in [0,T] $, $ M \in \N $ it holds that
$
  [\OM{t} - e^{tA}\xi]_{ \P, \B(H)} 
  = 
  \int_0^t
  e^{(t-\fl{s})A} \, dW_s
$.
Next note that
Lemma~\ref{lem:wiener} ensures that there exists
a stochastic process $ O \colon [0,T] \times \Omega \rightarrow V $
with continuous sample paths
which satisfies for all
$ t \in [0,T] $
that
$ 
  [O_t - e^{tA}\xi]_{ \P, \B(H)} 
  = 
  \int_0^t
  e^{(t-s)A} \, dW_s
$,
which satisfies for all
$ \omega \in \Omega $ that
$
  \limsup_{ r\searrow 0 }
  \sup_{ 0 \leq s < t \leq T }
  \frac{ 
    s
    \| O_t(\omega) - O_s(\omega) \|_V
  }{ (t-s)^r }
  <
  \infty
$,
and
which satisfies
for all $ \theta \in (0,\nicefrac{1}{4}) $
that
\begin{align}
\begin{split}
\label{eq:w_bounds_a}
 &\sup_{M \in \N}
  \sup_{ t \in [0,T] }
  \big\|
    \|
      O_t
    \|_{ V }
    +
    |M\fl{t}|^{\theta} \,
    \|
      O_t
      -
      O_{\fl{t}}
    \|_V
  \big\|_{ \L^{(3n-1)p\max\{2,2(n-1),2\vartheta,2(n-1)\vartheta\}}\!(\P;\R) }
  <
  \infty .
\end{split}
\end{align}
Furthermore, note that
for all $ q \in [1,\infty) $,
$ \theta \in [0,\nicefrac{1}{4}) $
it holds that
\begin{align}
\label{eq:w_bounds_11}
\begin{split}
 &\sup_{M \in \N}
  \sup_{ t \in [0,T] }
  M^{\theta}
  \left\|
    O_{t}
    -
    \OM{t}
  \right\|_{ \L^q(\P;V) }
\\&\leq
  \sup_{M \in \N}
  M^{\theta}
  \left(
    \E\!\left[ 
      \sup_{ t \in [0,T] }
      \left\|
	O_{t}
	-
	\OM{t}
      \right\|_{ V }^q
    \right]
  \right)^{\!\nicefrac{1}{q}}
  \leq
  \sup_{M \in \N}
  M^{\theta}
  \left(
    \E\!\left[ 
      \sup_{ t \in [0,T] }
      \left\|
	O_{t}
	-
	\OM{t}
      \right\|_{ \C([0,1],\R) }^q
    \right]
  \right)^{\!\nicefrac{1}{q}}
  <
  \infty
\end{split}
\end{align}
(cf., e.g., Theorem~4.13 
in Cox \& van Neerven~\cite{cvn10}).
This, the triangle inequality,
and~\eqref{eq:w_bounds_a} show that
\begin{align}
\label{eq:w_bounds_1}
\begin{split}
 &\sup_{M \in \N}
  \sup_{ t \in [0,T] }
  \left\|
    \OM{t}
  \right\|_{ \L^{(3n-1)p\max\{2,2(n-1),2\vartheta,2(n-1)\vartheta\}}\!(\P;L^{(3n-1)}(\lambda_{(0,1)}; \R)) }
\\&\leq
  \sup_{M \in \N}
  \sup_{ t \in [0,T] }
  \left\|
    O_t
  \right\|_{ \L^{(3n-1)p\max\{2,2(n-1),2\vartheta,2(n-1)\vartheta\}}\!(\P;V) }
\\&\quad+
  \sup_{M \in \N}
  \sup_{ t \in [0,T] }
  \left\|
    O_{t}
    -
    \OM{t}
  \right\|_{ \L^{(3n-1)p\max\{2,2(n-1),2\vartheta,2(n-1)\vartheta\}}\!(\P;V) }
  <
  \infty .
\end{split}
\end{align}
Moreover, note that
Corollary~\ref{cor:U_cor1} ensures
that there exists a real number 
$ K \in (0,\infty) $ which
satisfies for all $ h \in (0,T] $,
$ t \in [0,h] $, $ u, v \in V $
that
\begin{align}
\label{eq:u_ineq_1}
\begin{split} 
 &\big\|
    e^{tA}
    \big[
      u
      +
      t \,
      \one_{ 
	[0, h^{-\chi}]
      }
      ( \left\| u + v \right\|_{ V } ) \,
      F( u + v )
    \big]
  \big\|_{ L^{2n}(\lambda_{(0,1)};\R) }^{2n}
\\&
  \leq
  e^{t}
  \left[
    \|
      u
    \|_{ L^{2n}(\lambda_{(0,1)};\R) }^{2n}
    + 
    t
    K
    \max\!\left\{
      1,
      \|
	v
      \|_{ L^{(3n-1)}(\lambda_{(0,1)};\R) }^{(3n-1)}
    \right\}
  \right] .
\end{split}
\end{align}
In the next step observe that inequalities \eqref{eq:w_bounds_a}--\eqref{eq:w_bounds_1}
imply that for all $ \theta \in (0,\nicefrac{1}{4}) $ it holds that
\begin{align}
\label{eq:w_bounds_12}
\begin{split}
 &\sup_{M \in \N}
  \sup_{ t \in [0,T] }
  \left\|
    \left\|
      \OM{t}
    \right\|_{ L^{2n}(\lambda_{(0,1)}; \R) }^{2n}
    +
    K
    \max\!\left\{
      1,
      \left\|
	\OM{t}
      \right\|_{ L^{(3n-1)}(\lambda_{(0,1)}; \R) }^{(3n-1)}
    \right\}
  \right\|_{ \L^{p\max\{2,2(n-1),2\vartheta,2(n-1)\vartheta\}}\!(\P; \R) }
\\&\quad+
  \sup_{M \in \N}
  \sup_{ t \in [0,T] }
  \Big\| \!
    \left\|
      O_t
    \right\|_V
    +
    M^{\theta} \,
    \big\|
      O_t
      -
      \OM{t}
    \big\|_{ V }
  \Big\|_{ \L^{p\max\{2,2(n-1),4\vartheta,4(n-1)\vartheta\}}\!(\P; \R) }
\\&\quad+
  \sup_{M \in \N}
  \sup_{ t \in [0,T] }
  \Big[
    |M\fl{t}|^{\theta} \,
    \big\|
      O_t
      -
      O_{\fl{t}}
    \big\|_{ \L^{2p}(\P;V) }
  \Big]
  <
  \infty .
\end{split}
\end{align}
Combining this
and the fact that
$
  \forall \,
  t_1, t_2, t_3 \in [0,T]
$
with
$ 
  t_1 < t_2 < t_3 
  \colon 
  e^{(t_3-t_1)A}
  =
  e^{(t_3-t_2)A}
  e^{(t_2-t_1)A}
$
with
\eqref{eq:u_ineq_1},
Lemma~\ref{lem:semigroup_bound},
Lemma~\ref{lem:F_lem1},
Lemma~\ref{lem:F_lem2},
Lemma~\ref{lem:F_lem3},
and Lemma~\ref{lem:exponential_order}
allows us to apply Theorem~\ref{thm:main} (with 
$ \varphi = \max\{1,2(n-1)\} $, $ \phi = \nicefrac{1}{2} $,
$ \SM_{s,t} = e^{(t-s)A} $,
$ \Um(v) = \| v \|_{L^{2n}(\lambda_{(0,1)}; \R) }^{2n} $,
$ \Rm(v) = K\max\{1, \| v \|_{L^{(3n-1)}(\lambda_{(0,1)}; \R) }^{(3n-1)}\} $,
$ \rho = \nicefrac{1}{4} - \nicefrac{1}{4n^2} $,
and $ \varrho = \nicefrac{3}{4} $
for $ v \in V $, $ M \in \N $, $ s, t \in [0,T] $ with $ s < t $
in the notation of Theorem~\ref{thm:main})
to obtain~\eqref{eq:exp_euler_convergence}.
The proof of Corollary~\ref{cor:exp_euler_convergence} is thus completed.
\end{proof}

The following result, Corollary~\ref{cor:exp_euler_convergence_2},
specializes Corollary~\ref{cor:exp_euler_convergence}
to the case where the initial random variable
$ \xi \colon \Omega \rightarrow V $ satisfies
$ \forall \, p \in (0,\infty) \colon \E\big[ \| \xi \|_V^p \big] < \infty $.

\begin{corollary}
\label{cor:exp_euler_convergence_2}
Assume the setting in Section~\ref{sec:example_setting}, let 
$ \D = (0,1) $,
$ \xi \in \cap_{p\in(0,\infty)} \L^p(\P; V) $,
$ \chi \in (0,\tfrac{1}{2n}] $,
let $ (W_t)_{ t \in [0,T] } $ be an 
$ \Id_H $-cylindrical $ (\F_t)_{ t \in [0,T] } $-Wiener process,
let $ X \colon [0,T] \times \Omega \rightarrow V $ 
be a stochastic process with continuous sample paths which
satisfies for all $ t \in [0,T] $ that
$ 
  \big[
    X_t 
    -
    e^{tA} \xi 
    - 
    \int_0^t 
    e^{(t-s)A} 
    F(X_s) \, ds
  \big]_{\P,\B(H)}
  =
  \int_0^t
  e^{(t-s)A} \, dW_s
$,
and let $ \YM{} \colon [0,T] \times \Omega \rightarrow V $, $ M \in \N $,
be stochastic processes which satisfy for all $ t \in [0,T] $, $ M \in \N $ that
\begin{align}
\begin{split}
 &\left[ 
    \YM{t}
    -
    e^{tA} \xi
    -
    \smallint_0^t
    e^{(t-\fl{s})A} \,
    \one_{ 
      \{
	\| \YM{\fl{s}} \|_{\BaMi}
	\leq
	(M/T)^{\chi}
      \}
    } \,
    F(\YM{\fl{s}}) \, ds
  \right]_{ \P, \B(H) }
  =
  \smallint_0^t
  e^{(t-\fl{s})A} \, dW_s .
\end{split}
\end{align}
Then it holds for all $ \theta \in [0,\nicefrac{1}{4}) $,
$ p \in (0,\infty) $ that
\begin{align}
\begin{split}
 &\sup_{M \in \N}
  \sup_{t \in [0,T]}
  \left[
    \big\|
      \YM{t}
    \big\|_{ \L^{p}(\P;V) }
    +
    \big\|
      X_t
    \big\|_{ \L^{p}(\P;H) }
    +
    M^{\theta} \,
    \big\|
      X_t 
      -
      \YM{t}
    \big\|_{ \L^p(\P;H) }
  \right]
  <
  \infty .
\end{split}
\end{align}
\end{corollary}

\subsection[Strong convergence rates for a nonlinearity-truncated
linear-implicit Euler scheme]{Strong convergence rates for a 
nonlinearity-truncated \\ linear-implicit Euler scheme}

\begin{corollary}
\label{cor:li_euler_convergence}
Assume the setting in Section~\ref{sec:example_setting}, let
$ \vartheta \in (0,\infty) $,
$ \chi \in (0,\tfrac{1}{2n}] $, $ p \in [2,\infty) $, $ \D = (0,1) $,
$ \xi \in \L^{(3n-1)p\max\{2,2(n-1),2\vartheta,2(n-1)\vartheta\}}\!(\P; H_{\nicefrac{3}{4}}) $,
let $ (W_t)_{ t \in [0,T] } $ be an 
$ \Id_H $-cylindrical $ (\F_t)_{ t \in [0,T] } $-Wiener process,
let $ X \colon [0,T] \times \Omega \rightarrow V $ 
be a stochastic process with continuous sample paths which
satisfies for all $ t \in [0,T] $ that
$ 
  \big[
    X_t 
    - 
    e^{tA} \xi 
    - 
    \int_0^t 
    e^{(t-s)A} 
    F(X_s) \, ds
  \big]_{\P,\B(H)}
  =
  \int_0^t
  e^{(t-s)A} \, dW_s
$,
and let $ \YM{} \colon [0,T] \times \Omega \rightarrow V $, $ M \in \N $,
be stochastic processes which satisfy for all $ t \in [0,T] $, $ M \in \N $ that
\begin{align}
\begin{split}
 &\bigg[
    \YM{t}
    -
    \big(
      \Id_H - (t-\fl{t})A
    \big)^{-1}
    \big(
      \Id_H - \tfrac{T}{M}A
    \big)^{ -\fl{t}\frac{M}{T} } 
    \xi
\\&-
    \smallint\limits_0^t
    \big(
      \Id_H - (t-\fl{t})A
    \big)^{-1}
    \big(
      \Id_H - \tfrac{T}{M}A
    \big)^{ (\fl{s}-\fl{t})\frac{M}{T} } \,
    \one_{ 
      \{
	\| \YM{\fl{s}} \|_{\BaMi}
	\leq
	(M/T)^{\chi}
      \}
    } \,
    F(\YM{\fl{s}}) \, ds
  \bigg]_{\P,\B(H)}
\\&=
  \smallint\limits_0^t
  \big(
    \Id_H - (t-\fl{t})A
  \big)^{-1}
  \big(
    \Id_H - \tfrac{T}{M}A
  \big)^{ (\fl{s}-\fl{t})\frac{M}{T} } \, dW_s .
\end{split}
\end{align}
Then it holds for all
$ \theta \in [0,\nicefrac{1}{4}) $ that
\begin{align}
\label{eq:li_euler_convergence}
\begin{split}
 &\sup_{M \in \N}
  \sup_{t \in [0,T]}
  \left[
    \big\|
      \YM{t}
    \big\|_{ \L^{p\max\{2,2(n-1),4\vartheta,4(n-1)\vartheta\}}(\P;V) }
    +
    \big\|
      X_t
    \big\|_{ \L^{p}(\P;H) }
    +
    M^{\min\{\vartheta\chi,\theta\}} \,
    \big\|
      X_t 
      -
      \YM{t}
    \big\|_{ \L^p(\P;H) }
  \right]
  <
  \infty .
\end{split}
\end{align}
\end{corollary}
\begin{proof}[Proof of Corollary~\ref{cor:li_euler_convergence}]
Throughout this proof let
$ \OM{} \colon [0,T] \times \Omega \rightarrow V $, $ M \in \N $, 
be stochastic processes with the property that for all
$ t \in [0,T] $, $ M \in \N $ it holds that
$
  \big[
    \OM{t} 
    - 
    (
      \Id_H - (t-\fl{t})A
    )^{-1}
    (
      \Id_H - \tfrac{T}{M}A
    )^{ -\fl{t}\frac{M}{T} } 
    \xi
  \big]_{ \P, \B(H)} 
  = 
  \int_0^t
  (
    \Id_H - (t-\fl{t})A
  )^{-1}
  (
    \Id_H - \tfrac{T}{M}A
  )^{ (\fl{s}-\fl{t})\frac{M}{T} } \, dW_s
$.
Next note that
Lemma~\ref{lem:wiener} 
together with the
fact that $ H_{\nicefrac{3}{4}} \subseteq V $
continuously ensures that there exists
a stochastic process $ O \colon [0,T] \times \Omega \rightarrow V $
with continuous sample paths
which satisfies for all
$ \omega \in \Omega $ that
$
  \limsup_{ r\searrow 0 }
  \sup_{ 0 \leq s < t \leq T }
  \frac{ 
    s
    \| O_t(\omega) - O_s(\omega) \|_V
  }{ (t-s)^r }
  <
  \infty
$,
which satisfies for all
$ t \in [0,T] $
that
$ 
  [O_t - e^{tA}\xi]_{ \P, \B(H)} 
  = 
  \int_0^t
  e^{(t-s)A} \, dW_s
$,
and
which satisfies
for all $ \theta \in (0,\nicefrac{1}{4}) $
that
\begin{align}
\begin{split}
\label{eq:w_bounds}
 &\sup_{M \in \N}
  \sup_{ t \in [0,T] }
  \big\|
    \|
      O_t
    \|_{ V }
    +
    |M\fl{t}|^{\theta} \,
    \|
      O_t
      -
      O_{\fl{t}}
    \|_V
  \big\|_{ \L^{(3n-1)p\max\{2,2(n-1),2\vartheta,2(n-1)\vartheta\}}\!(\P;\R) }
  <
  \infty .
\end{split}
\end{align}
Moreover,
note that for all
$ q \in [1,\infty) $,
$ \theta \in [0, \nicefrac{1}{4}) $
it holds that
\begin{align}
\begin{split}
 &\sup_{M \in \N}
  \sup_{ t \in [0,T] }
  M^{\theta} 
  \left\|
    O_t
    -
    \OM{t}
  \right\|_{ \L^q(\P;V) }
  \leq
  \sup_{M \in \N}
  M^{\theta} 
  \left(
    \E\!\left[ 
      \sup_{ t \in [0,T] }
      \left\|
        O_t
        -
        \OM{t}
      \right\|_{ V }^q
    \right]
  \right)^{\!\nicefrac{1}{q}}
  <
  \infty 
\end{split}
\end{align}
(cf., e.g., Theorem~1.1 
in Cox \& van Neerven~\cite{cvn13}).
The triangle inequality
and~\eqref{eq:w_bounds}
therefore show that
\begin{align}
\label{eq:w_bounds_33}
\begin{split}
 &\sup_{M \in \N}
  \sup_{ t \in [0,T] }
  \left\|
    \OM{t}
  \right\|_{ \L^{(3n-1)p\max\{2,2(n-1),2\vartheta,2(n-1)\vartheta\}}\!(\P;L^{(3n-1)}(\lambda_{(0,1)}; \R)) }
\\&\leq
  \sup_{M \in \N}
  \sup_{ t \in [0,T] }
  \left\|
    O_t
  \right\|_{ \L^{(3n-1)p\max\{2,2(n-1),2\vartheta,2(n-1)\vartheta\}}\!(\P;V) }
\\&\quad+
  \sup_{M \in \N}
  \sup_{ t \in [0,T] }
  \left\|
    O_t
    -
    \OM{t}
  \right\|_{ \L^{(3n-1)p\max\{2,2(n-1),2\vartheta,2(n-1)\vartheta\}}\!(\P;V) }
  <
  \infty .
\end{split}
\end{align}
Next note that
Corollary~\ref{cor:U_cor1} ensures
that there exists a real number
$ K \in (0,\infty) $ which
satisfies for all $ h \in (0,T] $,
$ t \in [0,h] $, $ u, v \in V $
that
\begin{align}
\label{eq:u_ineq_2}
\begin{split} 
 &\big\|
    \big( \Id_H - tA\big)^{-1}
    \big[
      u
      +
      t \,
      \one_{ 
	[0, h^{-\chi}]
      }
      ( \left\| u + v \right\|_{ V } ) \,
      F( u + v )
    \big]
  \big\|_{ L^{2n}(\lambda_{(0,1)};\R) }^{2n}
\\&
  \leq
  e^{t}
  \left[
    \|
      u
    \|_{ L^{2n}(\lambda_{(0,1)};\R) }^{2n}
    + 
    t
    K
    \max\!\left\{
      1,
      \|
	v
      \|_{ L^{(3n-1)}(\lambda_{(0,1)};\R) }^{(3n-1)}
    \right\}
  \right] .
\end{split}
\end{align}
Moreover, inequalities~\eqref{eq:w_bounds}--\eqref{eq:w_bounds_33}
assure that
for all $ \theta \in (0,\nicefrac{1}{4}) $
it holds that
\begin{align}
\label{eq:w_bounds_4}
\begin{split}
 &\sup_{M \in \N}
  \sup_{ t \in [0,T] }
  \left\|
    \left\|
      \OM{t}
    \right\|_{ L^{2n}(\lambda_{(0,1)}; \R) }^{2n}
    +
    K
    \max\!\left\{
      1,
      \left\|
	\OM{t}
      \right\|_{ L^{(3n-1)}(\lambda_{(0,1)}; \R) }^{(3n-1)}
    \right\}
  \right\|_{ \L^{p\max\{2,2(n-1),2\vartheta,2(n-1)\vartheta\}}\!(\P; \R) }
\\&\quad+
  \sup_{M \in \N}
  \sup_{ t \in [0,T] }
  \Big\| \!
    \left\|
      O_t
    \right\|_V
    +
    M^{\theta}
    \left\|
      O_t
      -
      \OM{t}
    \right\|_V
  \Big\|_{ \L^{p\max\{2,2(n-1),4\vartheta,4(n-1)\vartheta\}}\!(\P; \R) }
\\&\quad+
  \sup_{M \in \N}
  \sup_{ t \in [0,T] }
  \Big[ 
    |M\fl{t}|^{\theta} \,
    \big\|
      O_t
      -
      O_{\fl{t}}
    \big\|_{ \L^{2p}(\P;V) }
  \Big]
  <
  \infty .
\end{split}
\end{align}
In addition, observe that 
for all
$
  t_1, t_2, t_3 \in [0,T]
$
with
$ 
  t_1 < t_2 < t_3 
$
it holds that
\begin{align}
\begin{split}
 &\big(
    \Id_H
    -
    (t_1-\fl{t_1}) A
  \big)
  \big(
    \Id_H
    -
    (t_3-\fl{t_3}) A
  \big)^{-1}
  \big(
    \Id_H
    -
    \tfrac{T}{M} A
  \big)^{(\fl{t_1}-\fl{t_3})\frac{T}{M}}
\\&=
  \big(
    \Id_H
    -
    (t_2-\fl{t_2}) A
  \big)
  \big(
    \Id_H
    -
    (t_3-\fl{t_3}) A
  \big)^{-1}
  \big(
    \Id_H
    -
    \tfrac{T}{M} A
  \big)^{(\fl{t_2}-\fl{t_3})\frac{T}{M}}
\\&\quad
  \big(
    \Id_H
    -
    (t_1-\fl{t_1}) A
  \big)
  \big(
    \Id_H
    -
    (t_2-\fl{t_2}) A
  \big)^{-1}
  \big(
    \Id_H
    -
    \tfrac{T}{M} A
  \big)^{(\fl{t_1}-\fl{t_2})\frac{T}{M}}
\end{split}
\end{align}
(cf., e.g., (142)--(146) in Da Prato et al.~\cite{pjr10}).
Combining this, \eqref{eq:u_ineq_2}--\eqref{eq:w_bounds_4},
Lemma~\ref{lem:semigroup_bound},
Lemma~\ref{lem:F_lem1},
Lemma~\ref{lem:F_lem2},
Lemma~\ref{lem:F_lem3},
and Lemma~\ref{lem:limplicit_order_4}
ensures that we can apply 
Theorem~\ref{thm:main} (with 
$ \varphi = \max\{1, 2(n-1) \} $, $ \phi = \nicefrac{1}{2} $,
$ \Um(v) = \| v \|_{L^{2n}(\lambda_{(0,1)}; \R) }^{2n} $,
$ \Rm(v) = K\max\{1, \| v \|_{L^{(3n-1)}(\lambda_{(0,1)}; \R) }^{(3n-1)}\} $,
$ 
  \SM_{s,t} 
  = 
  (
    \Id_H
    -
    (s-\fl{s}) A
  )
  (
    \Id_H
    -
    (t-\fl{t}) A
  )^{-1}
  (
    \Id_H
    -
    \frac{T}{M} A
  )^{(\fl{s}-\fl{t})\frac{T}{M}} 
$,
$ \rho = \nicefrac{1}{4} - \nicefrac{1}{4n^2} $,
and $ \varrho = \nicefrac{3}{4} $
for $ v \in V $, $ M \in \N $, $ s, t \in [0,T] $ with $ s < t $
in the notation of Theorem~\ref{thm:main})
to obtain~\eqref{eq:li_euler_convergence}.
The proof of Corollary~\ref{cor:li_euler_convergence} is thus completed. 
\end{proof}

In our next result, Corollary~\ref{cor:li_euler_convergence_2},
we specialize Corollary~\ref{cor:li_euler_convergence}
to the case where the initial random variable
$ \xi \colon \Omega \rightarrow H_{\nicefrac{3}{4}} $ satisfies
$ \forall \, p \in (0,\infty) \colon \E\big[ \| \xi \|_{H_{\nicefrac{3}{4}}}^p \big] < \infty $.

\begin{corollary}
\label{cor:li_euler_convergence_2}
Assume the setting in Section~\ref{sec:example_setting}, let
$ \D = (0,1) $,
$ \xi \in \cap_{p\in(0,\infty)} \L^p(\P; H_{\nicefrac{3}{4}}) $,
$ \chi \in (0,\tfrac{1}{2n}] $,
let $ (W_t)_{ t \in [0,T] } $ be an 
$ \Id_H $-cylindrical $ (\F_t)_{ t \in [0,T] } $-Wiener process,
let $ X \colon [0,T] \times \Omega \rightarrow V $ 
be a stochastic process with continuous sample paths which
satisfies for all $ t \in [0,T] $ that
$ 
  \big[
    X_t 
    - 
    e^{tA} \xi 
    - 
    \int_0^t 
    e^{(t-s)A} 
    F(X_s) \, ds
  \big]_{\P,\B(H)}
  =
  \int_0^t
  e^{(t-s)A} \, dW_s
$,
and let $ \YM{} \colon [0,T] \times \Omega \rightarrow V $, $ M \in \N $,
be stochastic processes which satisfy for all $ t \in [0,T] $, $ M \in \N $ that
\begin{align}
\begin{split}
 &\bigg[
    \YM{t}
    -
    \big(
      \Id_H - (t-\fl{t})A
    \big)^{-1}
    \big(
      \Id_H - \tfrac{T}{M}A
    \big)^{ -\fl{t}\frac{M}{T} } 
    \xi
\\&-
    \smallint\limits_0^t
    \big(
      \Id_H - (t-\fl{t})A
    \big)^{-1}
    \big(
      \Id_H - \tfrac{T}{M}A
    \big)^{ (\fl{s}-\fl{t})\frac{M}{T} } \,
    \one_{ 
      \{
	\| \YM{\fl{s}} \|_{\BaMi}
	\leq
	(M/T)^{\chi}
      \}
    } \,
    F(\YM{\fl{s}}) \, ds
  \bigg]_{\P,\B(H)}
\\&=
  \smallint\limits_0^t
  \big(
    \Id_H - (t-\fl{t})A
  \big)^{-1}
  \big(
    \Id_H - \tfrac{T}{M}A
  \big)^{ (\fl{s}-\fl{t})\frac{M}{T} } \, dW_s .
\end{split}
\end{align}
Then it holds for all
$ \theta \in [0,\nicefrac{1}{4}) $,
$ p \in (0,\infty) $ that
\begin{align}
\begin{split}
 &\sup_{M \in \N}
  \sup_{t \in [0,T]}
  \left[
    \big\|
      \YM{t}
    \big\|_{ \L^p(\P;V) }
    +
    \big\|
      X_t
    \big\|_{ \L^{p}(\P;H) }
    +
    M^{\theta} \,
    \big\|
      X_t 
      -
      \YM{t}
    \big\|_{ \L^p(\P;H) }
  \right]
  <
  \infty .
\end{split}
\end{align}
\end{corollary}

\subsubsection*{Acknowledgments}
This work has been partially supported
through the SNSF-Research project
200021\_156603
``Numerical approximations of nonlinear
stochastic ordinary and partial differential
equations''.

\bibliographystyle{acm}
\bibliography{bibfile}

\def\cprime{$'$} \def\cprime{$'$} \def\cprime{$'$} \def\cprime{$'$}
  \def\polhk#1{\setbox0=\hbox{#1}{\ooalign{\hidewidth
  \lower1.5ex\hbox{`}\hidewidth\crcr\unhbox0}}}
\begin{thebibliography}{10}

\bibitem{BeynIsaakKruse2014}
{\sc Beyn, W.-J., Isaak, E., and Kruse, R.}
\newblock {Stochastic C-stability and B-consistency of explicit and implicit
  Euler-type schemes}.
\newblock {\em {J. Sci. Comput.}\/} (2015), 1--33.

\bibitem{BeynIsaakKruse2015}
{\sc {Beyn}, W.-J., {Isaak}, E., and {Kruse}, R.}
\newblock {Stochastic C-stability and B-consistency of explicit and implicit
  Milstein-type schemes}.
\newblock {\em arXiv:1512.06905\/} (2015), 35 pages.

\bibitem{bj09}
{\sc Bl{\"o}mker, D., and Jentzen, A.}
\newblock Galerkin approximations for the stochastic {B}urgers equation.
\newblock {\em SIAM J. Numer. Anal. 51}, 1 (2013), 694--715.

\bibitem{Cerrai2001}
{\sc Cerrai, S.}
\newblock {\em Second order {PDE}'s in finite and infinite dimension},
  vol.~1762 of {\em Lecture Notes in Mathematics}.
\newblock Springer-Verlag, Berlin, 2001.
\newblock A probabilistic approach.

\bibitem{cvn10}
{\sc Cox, S., and van Neerven, J.}
\newblock Convergence rates of the splitting scheme for parabolic linear
  stochastic {C}auchy problems.
\newblock {\em SIAM J. Numer. Anal. 48}, 2 (2010), 428--451.

\bibitem{cvn13}
{\sc Cox, S., and van Neerven, J.}
\newblock Pathwise {H}\"older convergence of the implicit-linear {E}uler scheme
  for semi-linear {SPDE}s with multiplicative noise.
\newblock {\em Numer. Math. 125}, 2 (2013), 259--345.

\bibitem{pjr10}
{\sc {Da Prato}, G., {Jentzen}, A., and {Roeckner}, M.}
\newblock {A mild Ito formula for {SPDEs}}.
\newblock {\em arXiv:1009.3526\/} (2012), 39 pages.

\bibitem{dz92}
{\sc Da~Prato, G., and Zabczyk, J.}
\newblock {\em Stochastic equations in infinite dimensions}, vol.~44 of {\em
  Encyclopedia of Mathematics and its Applications}.
\newblock Cambridge University Press, Cambridge, 1992.

\bibitem{dz96}
{\sc Da~Prato, G., and Zabczyk, J.}
\newblock {\em Ergodicity for infinite-dimensional systems}, vol.~229 of {\em
  London Mathematical Society Lecture Note Series}.
\newblock Cambridge University Press, Cambridge, 1996.

\bibitem{DareiotisKumarSabanis2014}
{\sc {Dareiotis}, K., {Kumar}, C., and {Sabanis}, S.}
\newblock {On tamed Euler approximations of SDEs driven by L{\'e}vy noise with
  applications to delay equations}.
\newblock {\em arXiv:1403.0498\/} (2015), 27 pages.

\bibitem{gm05}
{\sc Gy{\"o}ngy, I., and Millet, A.}
\newblock On discretization schemes for stochastic evolution equations.
\newblock {\em Potential Anal. 23}, 2 (2005), 99--134.

\bibitem{GoengySabanisS2014}
{\sc {Gy{\"o}ngy}, I., {Sabanis}, S., and {{\v S}i{\v s}ka}, D.}
\newblock {Convergence of tamed {E}uler schemes for a class of stochastic
  evolution equations}.
\newblock {\em arXiv:1308.1796\/} (2015), 17 pages.

\bibitem{Halidias2013}
{\sc Halidias, N.}
\newblock A novel approach to construct numerical methods for stochastic
  differential equations.
\newblock {\em Numer. Algorithms 66}, 1 (2013), 79--87.

\bibitem{Halidias2015}
{\sc Halidias, N.}
\newblock {Construction of positivity preserving numerical schemes for some
  multidimensional stochastic differential equations}.
\newblock {\em {Discrete Contin. Dyn. Syst. Ser. B} 20}, 1 (2015), 153--160.

\bibitem{HalidiasStamatiou2013}
{\sc {Halidias}, N., and {Stamatiou}, I.~S.}
\newblock {On the numerical solution of some nonlinear stochastic differential
  equations using the semi-discrete method}.
\newblock {\em arXiv:1309.3189\/} (2014), 36 pages.

\bibitem{Hu1996}
{\sc Hu, Y.}
\newblock Semi-implicit {E}uler-{M}aruyama scheme for stiff stochastic
  equations.
\newblock In {\em Stochastic analysis and related topics, {V} ({S}ilivri,
  1994)}, vol.~38 of {\em Progr. Probab.} Birkh\"auser Boston, Boston, MA,
  1996, pp.~183--202.

\bibitem{HutzenthalerJentzen2014}
{\sc {Hutzenthaler}, M., and {Jentzen}, A.}
\newblock {On a perturbation theory and on strong convergence rates for
  stochastic ordinary and partial differential equations with non-globally
  monotone coefficients}.
\newblock {\em arXiv:1401.0295\/} (2014), 41 pages.

\bibitem{HutzenthalerJentzen2012}
{\sc Hutzenthaler, M., and Jentzen, A.}
\newblock Numerical approximations of stochastic differential equations with
  non-globally {L}ipschitz continuous coefficients.
\newblock {\em Mem. Amer. Math. Soc. 236}, 1112 (2015), v+99.

\bibitem{hjk11}
{\sc Hutzenthaler, M., Jentzen, A., and Kloeden, P.~E.}
\newblock Strong and weak divergence in finite time of {E}uler's method for
  stochastic differential equations with non-globally {L}ipschitz continuous
  coefficients.
\newblock {\em Proc. R. Soc. Lond. Ser. A Math. Phys. Eng. Sci. 467\/} (2011),
  1563--1576.

\bibitem{HutzenthalerJentzenKloeden2012}
{\sc Hutzenthaler, M., Jentzen, A., and Kloeden, P.~E.}
\newblock Strong convergence of an explicit numerical method for {SDE}s with
  non-globally {L}ipschitz continuous coefficients.
\newblock {\em Ann. Appl. Probab. 22}, 4 (2012), 1611--1641.

\bibitem{HutzenthalerJentzenKloeden2013}
{\sc Hutzenthaler, M., Jentzen, A., and Kloeden, P.~E.}
\newblock Divergence of the multilevel {M}onte {C}arlo {E}uler method for
  nonlinear stochastic differential equations.
\newblock {\em Ann. Appl. Probab. 23}, 5 (2013), 1913--1966.

\bibitem{HutzenthalerJentzenWang2013}
{\sc {Hutzenthaler}, M., {Jentzen}, A., and {Wang}, X.}
\newblock {Exponential integrability properties of numerical approximation
  processes for nonlinear stochastic differential equations}.
\newblock {\em arXiv:1309.7657\/} (2014), 32 pages.

\bibitem{Jentzen2015SPDElecturenotes}
{\sc Jentzen, A.}
\newblock Numerical analysis of stochastic partial differential equations.
\newblock Lecture notes, ETH Zurich, summer semester 2015, available online at
  http://www.math.ethz.ch/education/bachelor/lectures/fs2015/math/numsol, May
  2015.

\bibitem{JentzenKurniawan2015}
{\sc {Jentzen}, A., and {Kurniawan}, R.}
\newblock {Weak convergence rates for Euler-type approximations of semilinear
  stochastic evolution equations with nonlinear diffusion coefficients}.
\newblock {\em arXiv:1501.03539\/} (2015), 51 pages.

\bibitem{jp2015}
{\sc {Jentzen}, A., and {Pu{\v s}nik}, P.}
\newblock {Strong convergence rates for an explicit numerical approximation
  method for stochastic evolution equations with non-globally Lipschitz
  continuous nonlinearities}.
\newblock {\em arXiv:1504.03523\/} (2015), 38 pages.

\bibitem{kll2015}
{\sc {Kov{\'a}cs}, M., {Larsson}, S., and {Lindgren}, F.}
\newblock {On the discretization in time of the stochastic Allen-Cahn
  equation}.
\newblock {\em arXiv:1510.03684\/} (2015), 34 pages.

\bibitem{KumarSabanis2014}
{\sc {Kumar}, C., and {Sabanis}, S.}
\newblock {On tamed Milstein schemes of SDEs driven by L{\'e}vy noise}.
\newblock {\em arXiv:1407.5347\/} (2015), 45 pages.

\bibitem{KumarSabanis2016}
{\sc {Kumar}, C., and {Sabanis}, S.}
\newblock {On Milstein approximations with varying coefficients: the case of
  super-linear diffusion coefficients}.
\newblock {\em arXiv:1601.02695\/} (2016), 31 pages.

\bibitem{Kurniawan2014}
{\sc Kurniawan, R.}
\newblock {\em Numerical approximations of stochastic partial differential
  equations with non-globally Lipschitz continuous nonlinearities}.
\newblock University of Zurich and ETH Zurich, Zurich, Switzerland, 2014.
\newblock 74 pages. Master thesis.

\bibitem{LiuMao2013}
{\sc Liu, W., and Mao, X.}
\newblock {Strong convergence of the stopped Euler-Maruyama method for
  nonlinear stochastic differential equations}.
\newblock {\em {Appl. Math. Comput.} 223\/} (2013), 389 -- 400.

\bibitem{lr04}
{\sc Lord, G.~J., and Rougemont, J.}
\newblock A numerical scheme for stochastic {PDE}s with {G}evrey regularity.
\newblock {\em IMA J. Numer. Anal. 24}, 4 (2004), 587--604.

\bibitem{lt13}
{\sc Lord, G.~J., and Tambue, A.}
\newblock Stochastic exponential integrators for the finite element
  discretization of {SPDE}s for multiplicative and additive noise.
\newblock {\em IMA J. Numer. Anal. 33}, 2 (2013), 515--543.

\bibitem{l09}
{\sc Lunardi, A.}
\newblock {\em Interpolation theory}, second~ed.
\newblock Appunti. Scuola Normale Superiore di Pisa (Nuova Serie). [Lecture
  Notes. Scuola Normale Superiore di Pisa (New Series)]. Edizioni della
  Normale, Pisa, 2009.

\bibitem{Mao2015}
{\sc Mao, X.}
\newblock {The truncated Euler-Maruyama method for stochastic differential
  equations}.
\newblock {\em {J. Comput. Appl. Math.} 290\/} (2015), 370 -- 384.

\bibitem{Mao2016}
{\sc Mao, X.}
\newblock {Convergence rates of the truncated Euler-Maruyama method for
  stochastic differential equations}.
\newblock {\em {J. Comput. Appl. Math.} 296\/} (2016), 362 -- 375.

\bibitem{NgoLuong2015}
{\sc {Ngo}, H.-L., and {Luong}, D.-T.}
\newblock {Strong rate of tamed Euler-Maruyama approximation for stochastic
  differential equations with H{\"o}lder continuous diffusion coefficient}.
\newblock To appear in \textit{Braz. J. Prob. Stat.}, available online at
  http://imstat.org/bjps/papers/BJPS301.pdf, August 2015, 18 pages.

\bibitem{p83}
{\sc Pazy, A.}
\newblock {\em Semigroups of linear operators and applications to partial
  differential equations}, vol.~44 of {\em Applied Mathematical Sciences}.
\newblock Springer-Verlag, New York, 1983.

\bibitem{rr04}
{\sc Renardy, M., and Rogers, R.~C.}
\newblock {\em An introduction to partial differential equations}, second~ed.,
  vol.~13 of {\em Texts in Applied Mathematics}.
\newblock Springer-Verlag, New York, 2004.

\bibitem{Sabanis2013}
{\sc Sabanis, S.}
\newblock {A note on tamed Euler approximations}.
\newblock {\em Electron. Commun. Probab. 18\/} (2013), no. 47, 1--10.

\bibitem{Sabanis2013E}
{\sc {Sabanis}, S.}
\newblock {Euler approximations with varying coefficients: the case of
  superlinearly growing diffusion coefficients}.
\newblock {\em arXiv:1308.1796\/} (2015), 24 pages.

\bibitem{SongLuLiu2015}
{\sc {Song}, M.~H., {Lu}, Y.~L., and {Liu}, M.~Z.}
\newblock {Convergence of the tamed Euler scheme for stochastic differential
  equations with piecewise continuous arguments under non-Lipschitz continuous
  coefficients}.
\newblock {\em arXiv:1502.00058\/} (2015), 16 pages.

\bibitem{SzpruchZhang2013}
{\sc {Szpruch}, L., and {Zh{\= a}ng}, X.}
\newblock {$V$-integrability, asymptotic stability and comparison theorem of
  explicit numerical schemes for SDEs}.
\newblock {\em arXiv:1310.0785\/} (2015), 33 pages.

\bibitem{TambueMukam2015}
{\sc {Tambue}, A., and {Mukam}, J.~D.}
\newblock {Strong convergence of the tamed and the semi-tamed Euler schemes for
  stochastic differential equations with jumps under non-global Lipschitz
  condition}.
\newblock {\em arXiv:1510.04729\/} (2015), 35 pages.

\bibitem{TretyakovZhang2013}
{\sc Tretyakov, M.~V., and Zhang, Z.}
\newblock {A fundamental mean-square convergence theorem for SDEs with locally
  Lipschitz coefficients and its applications}.
\newblock {\em SIAM J. Numer. Anal. 51}, 6 (2013), 3135--3162.

\bibitem{WangGan2013}
{\sc Wang, X., and Gan, S.}
\newblock {The tamed Milstein method for commutative stochastic differential
  equations with non-globally Lipschitz continuous coefficients}.
\newblock {\em J. Difference Equ. Appl. 19}, 3 (2013), 466--490.

\bibitem{Zhang2014}
{\sc {Zhang}, Z.}
\newblock {New explicit balanced schemes for SDEs with locally Lipschitz
  coefficients}.
\newblock {\em arXiv:1402.3708\/} (2014), 15 pages.

\bibitem{ZongWuHuang2014}
{\sc Zong, X., Wu, F., and Huang, C.}
\newblock {Convergence and stability of the semi-tamed Euler scheme for
  stochastic differential equations with non-Lipschitz continuous
  coefficients}.
\newblock {\em Appl. Math. Comput. 228\/} (2014), 240 -- 250.

\end{thebibliography}

\end{document}